\documentclass[journal]{IEEEtran}

\usepackage{amsmath}
\usepackage{amsthm}
\usepackage{amssymb}
\usepackage{mathrsfs}
\usepackage{amsfonts} %%\mathbb{R} \mathcal \mathscr
\usepackage{CJK}
\usepackage{multicol}
\usepackage{multirow}
\usepackage{diagbox}
\usepackage{graphicx}
\usepackage{graphics}
\usepackage{subfigure}
\usepackage{lineno} %显示行号

\usepackage[font=scriptsize]{subfig} %更改图片子标题大小
\usepackage[justification=centering]{caption}%latex图片标题居中
\usepackage{epstopdf} % transfer the figure with eps formation to figure with pdf formation
\usepackage{pgf,tikz}
\usepackage{tikz}
\usetikzlibrary{arrows,shapes,snakes,shadows,positioning,automata,patterns}
\usetikzlibrary{trees,decorations.pathmorphing,decorations.markings}
\usetikzlibrary{backgrounds}

\allowdisplaybreaks % allow the equation display in different pages

\usepackage{color} %\textup{sign}
\usepackage{float}
\usepackage{cite}
\usepackage{bm}

\newtheorem{rem}{Remark}

\newtheorem{thm}{Theorem}
\newtheorem{cor}{Corollary}
\newtheorem{lem}{Lemma}
\newtheorem{ass}{Assumption}
\newtheorem{defn}{Definition}

\newcommand{\R}{\mathbb{R}}
\newcommand{\N}{\mathbf{N}}
\newcommand{\F}{\mathscr{F}}

\newcommand{\V}{\mathcal{V}}
\newcommand{\T}{\ensuremath{\textit{T}}}

\renewcommand{\P}{\ensuremath{\mathbb{P}}} %概率P
\newcommand{\E}{\ensuremath{\mathbb{E}}} %数学期望E
\renewcommand{\O}{\mathcal{O}}

\DeclareMathOperator*{\col}{col}

\usepackage{algorithm}
\usepackage{algpseudocode}

\def\BibTeX{{\rm B\kern-.05em{\sc i\kern-.025em b}\kern-.08em
		T\kern-.1667em\lower.7ex\hbox{E}\kern-.125emX}}
\usepackage{balance}

\begin{document}
\title{Distributed Generalized Nash Equilibria Learning for Online Stochastic Aggregative Games}
\author{Kaixin~Du and Min~Meng
%First A. Author, \IEEEmembership{Fellow, IEEE}, Second B. Author, and Third C. Author Jr., \IEEEmembership{Member, IEEE}
\thanks{This work was partially supported by National Science and Technology Major Project under grant 2022ZD0119702, the National Natural Science Foundation of China under grant 62473293, 62103305 and 62088101, Shanghai Municipal Science and Technology Major Project under grant 2021SHZDZX0100. ({\em Corresponding author: Min Meng}).}
\thanks{Kaixin Du is with the Shanghai Research Institute for Intelligent Autonomous Systems, Tongji University, Shanghai 200092, China (e-mail: dukx@tongji.edu.cn).}
\thanks{Min Meng is with Department of Control Science and Engineering, College of Electronics and Information Engineering, National Key Laboratory of Autonomous Intelligent Unmanned System, Frontiers Science Center for Intelligent Autonomous Systems, Ministry of Education, Shanghai Research Institute for Intelligent Autonomous Systems, and Shanghai Institute of Intelligent Science and Technology, Tongji University, Shanghai, China (e-mail: mengmin@tongji.edu.cn).}
}
\markboth{}%
{Distributed Generalized Nash Equilibria Learning for Online Stochastic Aggregative Games}

\maketitle

\begin{abstract}
This paper investigates online stochastic aggregative games subject to local set constraints and time-varying coupled inequality constraints, where each player possesses a time-varying expectation-valued cost function relying on not only its own decision variable but also an aggregation of all the players' variables. Each player can only access its local individual cost function and constraints, necessitating partial information exchanges with neighboring players through time-varying unbalanced networks. Additionally, local cost functions and constraint functions are not prior knowledge and only revealed gradually. To learn generalized Nash equilibria of such games, a novel distributed online stochastic algorithm is devised based on push-sum and primal-dual strategies. Through rigorous analysis, high probability bounds on the regrets and constraint violation are provided by appropriately selecting  decreasing stepsizes. Moreover, for a time-invariant stochastic strictly monotone game, it is shown that the generated sequence by the designed algorithm converges to its variational generalized Nash equilibrium  almost surely, and the time-averaged sequence converges sublinearly with high probability when the game is strongly monotone. Finally, the derived theoretical results are illustrated by numerical simulations.
\end{abstract}


\section{Introduction}
Game theory has recently gained much attention in a breadth of fields, such as power grids \cite{Saad12}, social networks \cite{ghaderi2014opinion}, commodity distribution \cite{FRANCI2022}, and so on. In such problems, selfish players delve into optimize their individual private, yet interdependent objective functions. As important solution concepts for noncooperative games, Nash equilibrium (NE) and generalized Nash equilibrium (GNE) mathematically characterize stable and satisfactory states, from which no player possesses an incentive for unilateral deviation.

In last decades, a large body of distributed NEs and GNEs seeking algorithms in noncooperative games have sprung up  \cite{ghaderi2014opinion,facchinei2010generalized,ye2023distributed,Carnevale24Tracking-based,koshal2016Distributed,Belgioioso21,deng2021distributed,Gadjov21,Nguyen23}. In contrast to the full-decision information scenario \cite{ghaderi2014opinion,facchinei2010generalized}, where a coordinator is required to compute and broadcast the data to players, each player in the partial-decision information scenario is capable to make decisions relying only on partial estimates on decisions of others through local information exchanges between neighbouring players \cite{ye2023distributed,Carnevale24Tracking-based,koshal2016Distributed,Belgioioso21,deng2021distributed,Gadjov21,Nguyen23}. As a consequence, distributed algorithms offer advantages in light computation, robustness to link failures, and privacy preservation to some extent. Note that all investigations mentioned above focused on offline games, where cost and constraint functions of rational players are  time-invariant.

However, dynamic environments are commonly encountered in various real applications, including real-time traffic networks, radio resources allocation, and online auction. In such scenarios, the cost functions and/or constraints of players often vary over time because of external factors. Consequently, online games become a hot research topic, which have been extensively investigated \cite{lu21online,meng24ban,zuo2023distributed,duvocelle2023multiagent,liu23,lin24private,chen2024continuous,
li2023survey,meng2023online,DENG2025112200,Wang25online}. In online games, cost and constraint functions are not known as prior knowledge, which instead are revealed to players only after the decisions of players are made at current time. Each player aims to choose its decision by repeatedly adapting other players' decisions in a dynamic way such that its static/dynamic regret, i.e., the difference between  accumulated payoff incurred by a player's real decision and the best fixed/dynamic decision in hindsight, is minimized. Generally speaking, an online algorithm performs well if it satisfies the \emph{no-regret} property, i.e., the regret is sublinear versus the learning time.
To this end, in \cite{duvocelle2023multiagent}, mirror descent-based policies were studied to deal with online games with local set constraints.
%, which apply to both gradient-based and payoff-based feedback.
Utilizing multi-step mirror descents, a distributed online path-length-independent NE seeking algorithm was developed for players to make no-regret decisions \cite{DENG2025112200}.
When handling GNE problems, where players' strategies are coupled via constraints, the Lagrangian function becomes an important tool. It allows us to transfer equality/ineuqality constraint problems to saddle point problems and thus helps design effective primal-dual algorithms. Along this line, a distributed primal-dual algorithm was designed in \cite{lu21online} for online game with time-invariant nonlinear coupled nonlinear inequality constraints, which is separable since each player's constraint function  relies only on its own decision variable. To address time-varying coupled nonlinear inequality constraints, authors in \cite{meng24ban,meng2023online} proposed online distributed learning algorithms with bandit feedback based on the mirror descent and primal-dual strategy, and  a continuous-time algorithm was designed to address nonmonotone online games \cite{chen2024continuous}. Modeling real-world applications, the authors in \cite{Wang25online} designed an online distributed method to seek the GNE of the energy sharing market.  Moreover, as an important class of online noncooperative games, online aggregative games have attracted tremendous research interest \cite{liu23,zuo2023distributed,lin24private}. %due to its wide applications in target surrounding by multi-robots \cite{}, electric vehicle charing \cite{}, energy management in power grids.
Here, players' cost functions depend both on their individual decisions and the aggregate behavior of all players. For instance, in order to defend attacks of intruders, a target is surrounded by a collection of robots, and in this case, the local cost functions of robots rely on their own positions and the center of all robots' positions.
%Note that the algorithms proposed in \cite{lu21online,meng24ban,meng2023online,chen2024continuous,Wang25online} for generally coupled games are not applicable to aggregative games, because it does not exploit the aggregative structure, which is computationally inefficient when used for these types of games.
Note that the algorithms in \cite{lu21online,meng24ban,meng2023online,chen2024continuous,Wang25online} for general games are computationally inefficient when applied to aggregative games, as the aggregative structure was not exploited.
%require  each player to maintain an estimate of all other players' actions, which while can be not scalable/inefficient when applied to aggregative games.
With this regard,
%a distributed online NE tracking algorithm with privacy preservation was devised in \cite{lin24private} by combining projected pseudo-gradient and dynamic average consensus methods.
%Taking  local decision sets and global {\color{blue}time-varying} coupled inequality constraints into account,
authors in \cite{liu23} studied online aggregative games with local decision sets and global time-varying coupled inequality constraints under feedback delays, and \cite{zuo2023distributed} addressed time-varying unbalanced networks. It is noteworthy that a powerful tool for unbalanced  interaction graphs is the push-sum approach, which was originally devised for average consensus problems over unbalanced graphs \cite{dominguez2011distributed}.

It should be noted that all the aforementioned research considered online games with deterministic cost functions. However, due to various information uncertainties in practical situations, the local cost functions are often characterized as the expectation of stochastic functions. For example, in electricity markets, companies producing energy do not know the demand in advance \cite{henrion2007m}. To date, several works have focused on noncooperative games with expectation-valued cost functions \cite{lei22SNE,FRANCI2022,zheng2023distributed,wang2024distributed,Huang23,lei2023distributed}. Particularly, for general stochastic games with coupled linear constraints,
\cite{FRANCI2022,zheng2023distributed,Huang23} developed distributed algorithms via operator splitting schemes. However, these methods cannot handle nonlinear aggregative structure, coupled constraints, or time-varying communication networks, since their convergence analysis relies on operator-theoretic arguments that fail when the underlying mappings vary over time.
%Additionally, the algorithms in \cite{Belgioioso21,wang2024distributed,liu23,lei2023distributed} are applicable only to time-varying balanced networks. It is noteworthy that a powerful tool for unbalanced  interaction graphs is the push-sum approach, which was originally devised for average consensus problems over unbalanced graphs \cite{dominguez2011distributed}.
In addition, the aforementioned stochastic algorithms are not applicable to online cases, and theoretical guarantees were usually provided only in expectation. While guarantees hold in expectation can ensure the average performance of an algorithm over substantial amounts of independent runs, they cannot preclude an extremely bad case \cite{neu2015explore,nemirovski2009robust,liu2023high,li2019convergence}. Additionally, towards real applications, we can perform the algorithm only in a few runs  or even a single run owing to limitation on computing time or resources. For instance, in the distributed tracking problem, sensors require to accurately determine the target's position as quickly as possible, making it desirable to run the algorithm in finite rounds \cite{chen12diffu}. These observations highlight the importance of high probability guarantees of an algorithm in a single run both theoretically and practically. However, there is little research in this direction for online stochastic games. As far as we know, a relevant work is \cite{lu24Ongam}, where an online distributed stochastic mirror descent algorithm was proposed to handle online stochastic noncooperative games with set constraints. It was shown that the regret  increases sublinearly with high probability. Moreover, in \cite{yang2024online} and \cite{yang2025online}, online distributed stochastic gradient (SG) and clipped SG algorithms were developed for online distributed stochastic optimization with set constraints, which achieve sublinear regret bounds in high probability. Note that in distributed optimization, all agents aim to collaboratively seek the global optimal solution rather than compute an  NE/GNE, leading to fundamentally distinct algorithmic designs and analysis. In practical applications such as smart grids, the generation costs and global constraints are often stochastic or time-varying due to the uncertain and unpredictable features of renewable energy and fluctuation of power demands. Meanwhile, the underlying communication network may also be time-varying owing to wireless interference and link failures.
%since the power system requires reliable and safe operation rather than mere average-case performance.
%Hence, our current research is motivated both in theory and practice.
In this case, it is important to investigate effective algorithms with high-probability convergence, since they ensure the safety, stability, and economic efficiency of power systems.
Consequently, motivated by both theoretical and practical considerations, it is crucial to establish high probability bounds on regrets to guarantee the performance of stochastic online NEs (GNEs) seeking algorithms.

Driven by the above observations, we are going to focus on online stochastic aggregative games over time-varying unbalanced networks. Here, players aim to selfishly minimize their individual time-varying expectation-valued cost functions, while satisfying local feasible set constraints and time-varying coupled inequality constraints. Assume that players can make decisions by using their sample estimates on gradients instead of accurate gradients of cost functions. To solve this problem, a distributed online stochastic algorithm is devised based on primal-dual and push-sum schemes, which can achieve sublinear high probability bounds on regrets and constraint violation under mild conditions. When analyzing the high probability guarantees of the devised algorithm, the estimation error between the sampled gradients and exact gradients is inevitable. Note that the expectation on this related term becomes zero if the noisy gradient estimate is unbiased, thereby not influencing the regret and constraint violation bounds in expectation. Unfortunately, to establish effective high probability bounds, this error cannot be eliminated at all, which makes our analysis  more technically challenging. Moreover, since the aggregative variable is unavailable to each player, additional auxiliary variables are required to track it. However, the unbalanced communication networks further complicate our ability to track the aggregative variable precisely. Meanwhile, it is also hard for players to estimate the global optimal dual variable to achieve consensus. The introduced consensus errors will affect the upper bounds of regrets and constraint violation, which must be carefully addressed. %Additionally, due to the consideration of coupled constraints, there are some crossed terms between primal and dual variables, which should be carefully handled.
%The main challenges arise in the following aspects:
%\begin{itemize}
%  \item [i)]The aggregate structure involved in local cost functions complicates the concerned problem, as the aggregative variable and the gradient of the local cost function are unavailable to each player. This necessitates new schemes to address these challenges.
%  \item [ii)] It is essential to select appropriate techniques to efficiently overcome the challenges posed by unbalanced communication graphs, since the aggregative variable cannot be tracked and the global optimal dual variable cannot be estimated well as opposed to balanced graphs. Correspondingly, it is necessary to analyze the introduced consensus errors, which will also bring additional difficulties.
%  \item [iii)] When conducting the theoretical analysis, the estimated error between the sampled gradients and exact gradients is inevitable. However, the expectation on this related term becomes zero if the noisy gradient estimate is unbiased, thus not influencing the regrets and constraint violation bounds in expectation. Unfortunately, when analyzing the high probability bounds, this error cannot be eliminated at all, which makes our proof more technically challenging.
%\end{itemize}
The main contributions are summarized as below.
\begin{itemize}
  \item [i)] To our best knowledge, this paper is the first to study online stochastic aggregative games with local set constraints and time-varying  coupled nonlinear inequality constraints, where each player possesses a private time-varying expectation-valued cost function involved in an aggregative variable. In contrast to deterministic counterparts \cite{zuo2023distributed,liu23,lin24private} and offline stochastic games \cite{lei22SNE,FRANCI2022,zheng2023distributed,wang2024distributed,Huang23,lei2023distributed},  the formulated problem enjoys new characteristics and broadens the scope of practical applications.
  \item [ii)] An online distributed stochastic primal-dual push-sum algorithm is proposed. Unlike the doubly stochastic communication graphs employed in \cite{meng24ban,wang2024distributed,lu24Ongam,liu23,lin24private} to model players' information exchanges, we study more general time-varying unbalanced graphs. For this, a push-sum idea is leveraged for devising our algorithm in order to eliminate the impact on graphs' imbalance. Moreover, our work is more general than \cite{lu24Ongam} by considering an aggregative variable in cost functions and general time-varying convex coupled constraints. Unlike \cite{Huang23}, which utilized an augmented best-response (BR) scheme and thus required a double-loop algorithm, our approach uses a gradient-response scheme, resulting in a simpler and more computationally efficient structure.
\item [iii)] Under mild conditions, it is rigorously proved that the regrets and constraint violation of the proposed algorithm increase sublinearly with high probability. Additionally, the time-invariant stochastic aggregative game is considered as a special case. We show that the decision  sequence generated by the algorithm converges almost surely to the variational GNE (vGNE) for the strictly monotone game. Furthermore, the sublinear convergence rate in high probability of the averaged decision sequence is firstly derived under strong monotonicity. Distinctive from convergence guarantees that hold only in expectation \cite{lei22SNE,meng24ban,wang2024distributed}, our high probability results benefit for ensuring the effectiveness by executing the proposed algorithm only once, which is more mathematically and practically rigorous and desirable. Simultaneously, the theoretical analysis also becomes more complex and challenging.
\end{itemize}

The rest of this paper is organized as follows. Section \ref{sec2} presents some preliminaries and formulates the concerned problem. Section \ref{sec3} develops the algorithm and provides the main results. In Section \ref{sec4}, a numerical simulation is given to support the obtained results. Section \ref{sec5} concludes this paper.

\emph{Notations:} Denote by $\R^n$, $\R^n_+$, and $\R^{m\times n}$ the sets of $n$-dimensional real vectors, $n$-dimensional real vectors with nonnegative entries, and $m\times n$ real matrices, respectively. $\|\cdot\|$ and $\|\cdot\|_1$ represent standard Euclidean norm and $l_1$-norm, respectively.  $\langle x,y\rangle$ represents standard inner product of $x, y\in\R^n$. For vector $x\in\R^n$, $x^{\T}$ and $[x]_i$ denote its transpose and the $i$-th entry, respectively. Let ${\bf 1}$ and ${\bf 0}$ be compatible vectors of all entries equal to $1$ and $0$, respectively. $[N]:=\{1,\ldots,N\}$ represents the index set for an integer $N>0$. $\col\{x_i\}_{i\in[N]}:=(x_1^{\T},\ldots,x_N^{\T})^{\T}$. %For a function $g:\R^d\to\R^n$ with component functions $g_1,\ldots,g_n$, denote $\nabla g(x)\define [\nabla g_1(x),\ldots,\nabla g_n(x)]$.
Given two functions $h_1$ and $h_2$, $h_1=\mathcal{O}(h_2)$ means that there exists a constant $C>0$ such that $|h_1(x)|\leq C|h_2(x)|$ for all $x$ in their domains, and $h_1=o(T)$ indicates that $\lim_{T\to\infty}\frac{h_1}{T}=0$. For function $f(x, y):\R^d\times\R^n\to\R$, let $\nabla_{1} f(x, y):=\frac{\partial f(x, y)^{\T}}{\partial x}$ and $\nabla_2 f(x, y):= \frac{\partial f(x, y)^{\T}}{\partial y}$ denote its gradients with respect to the first variable $x$ and the second variable $y$, respectively. Given a closed convex set $X\subseteq\R^n$, $\N_X(x)=\{s:s^{\T}(y-x)\leq 0,\ \forall y\in X$\} denotes the normal cone of $x\in X$. Let $\Pi_X[x]={\arg\min}_{y\in X}\|y-x\|$ be the projection of a vector $x\in\R^n$ onto X, and $[x]_+:=\Pi_{\R^n_+}[x]$. It is well known that
\begin{align}
&\|\Pi_X[x]-\Pi_X[y]\|\leq \|x-y\|,\ \forall x,y\in X, \label{pro-nonexpen}\\
&\langle x-\Pi_X[x],y-\Pi_X[x]\rangle\leq 0,\ \forall y\in X. \label{pro-optimal}
\end{align}

\section{PRELIMINARIES}\label{sec2}

\subsection{Graph Theory}
At time instant $t$, denote by $\mathcal{G}_t=(\mathcal{V},\mathcal{E}_t)$ a time-varying directed graph, where $\mathcal{V}=[N]$ is the node set and $\mathcal{E}_t\subseteq\mathcal{V}\times\mathcal{V}$ is the edge set.  An edge $(j,i)\in\mathcal{E}_t$ if node $i$ could receive information from node $j$.
%{\color{red}Let $\mathcal{N}_{i,t}^+:=\{j\in[N]:(j,i)\in\mathcal{E}_t\}\cup\{i\} $ and $\mathcal{N}_{i,t}^-:=\{j\in[N]:(i,j)\in\mathcal{E}_t\} \cup\{i\} $ denote the in-neighbor and out-neighbor sets of node $i$, respectively. The in-degree and out-degree of node $i$ are defined by $d_{i,t}^+:=|\mathcal{N}_{i,t}^+|$ and $d_{i,t}^-:=|\mathcal{N}_{i,t}^-|$, respectively.} %It is assumed that $i\in\mathcal{N}_{i,t}^+$ and $i\in\mathcal{N}_{i,t}^-$ for any $i\in[N]$.
The weighted adjacency matrix is $\mathcal{W}_t=[w_{ij,t}]\in\mathbb{R}^{N\times N}$, where $w_{ij,t}>0$ if $(j,i)\in\mathcal{E}_t$ and $w_{ij,t}=0$, otherwise.
It is assumed that $w_{ii,t}>0$ for any $i\in[N]$.
A direct path from $v_1$ to $v_p$ in $\mathcal{G}_t$ is a sequence of distinct nodes $v_i$, $i\in[p]$ such that $(v_m,v_{m+1})\in\mathcal{E}_t$ for all $m\in[p-1]$. $\mathcal{G}_t$ is called strongly connected if there exists a directed path from any node to any other node in the graph.

We make the following assumptions on the interaction networks.
\begin{ass}\label{ass:graph} For any $t\geq 0$, the following conditions are satisfied:
$~~~~~~~~~~~~~$
\begin{itemize}
 \item [i)] For any $i, j \in[N]$, there exists a constant $\underline{w}\in(0,1)$ such that $w_{i j, t} \geq \underline{w}$ if $w_{i j, t}>0$;
\item [ii)] The adjacency matrix $\mathcal{W}_t$ is column-stochastic, i.e., $\sum_{i=1}^N w_{i j, t}=1$ for all $j \in[N]$;
\item [iii)] There exists a constant $U>0$, such that the graph $\left(\mathcal{V}, \cup_{l=0, \ldots, U-1} \mathcal{E}_{t+l}\right)$ is strongly connected.
\end{itemize}
\end{ass}

A typical choice for $w_{ij,t}$ is
\begin{align}\label{def:weight}
	w_{ij,t}:=
	\left\{
	\begin{aligned}
		&1/d_{j,t}^-,  & \text{if}\ j\in \mathcal{N}_{i,t}^+, \\
		&0,      & \text{otherwise},
	\end{aligned}
	\right.
\end{align}
where  $d_{j,t}^-$ and $\mathcal{N}_{i,t}^+$ denote the out-degree and in-neighbor set of node $i$, respectively. Moreover, it
is worth noting that compared to previous works \cite{meng24ban,wang2024distributed,lu24Ongam,liu23,lin24private}, which assumed the communication graphs to be doubly stochastic, Assumption \ref{ass:graph} is less conservative.

%$ w_{ij,t} := \begin{cases}
%    1/d_{j,t}^-, & \text{if } j \in \mathcal{N}_{i,t}^+, \\
%    0, & \text{otherwise}.
%\end{cases} $

\subsection{Problem Formulation}

Consider an online stochastic aggregative game consisting of $N$ players. At each time $t$, every player $i$, $i\in[N]$ aims to selfishly minimize its own cost function, leading to the following optimization problem
\begin{align}\label{prob-opt-i}
  &\min_{x_{i,t}\in X_i}~f_{i, t}(x_{i,t}, \sigma(x_{t}))\notag\\
  &\text{s.t.}~\sum_{i=1}^N g_{i,t}(x_{i,t})\leq{\bf 0} ,\ \forall i\in[N],
\end{align}
where $x_{i,t}\in X_i\subseteq\R^{d_i}$ is the decision of player $i$  with $X_i\subseteq\R^{d_i}$ being its private local set constraint, and $x_{t}:= \col\{x_{i,t}\}_{i\in[N]}\in\R^d$ is the joint decision of all players with $d:=\sum_{i=1}^Nd_i$. Here, player $i$'s private cost function is  $f_{i,t}(x_i,\sigma(x)):=\E[F_{i,t}(x_i,\sigma(x),\xi_i(\omega))]$, where $\sigma(x):=\frac{1}{N}\sum_{i=1}^N\psi_i(x_i)$ is the aggregation of all players' decision variables, with  each private function $\psi_i:\R^{d_i}\to\R^n$ representing player $i$'s contribution to $\sigma(x)$, $\xi_i:\Omega_i\to \R^{q_i}$ is a stochastic vector defined on the probability space $(\Omega_i,\F_i,\P_i)$, $F_{i,t}:\R^{d_i}\times\R^n\times \R^{q_i}\to\R$ is a scalar-valued function, which is differentiable in $(x_i,\sigma)\in X_i\times\R^n$ for any given $\xi_{i}\in\R^{q_i}$,
and $\E$ is the expectation with respect to the probability measure $\P_i$. Moreover, $g_{i,t}:\R^{d_i}\to\R^m$ is the private constraint function of player $i$. Denote by $x_{-i,t}:=\col\{x_{i,t}\}_{i\in[N]\setminus i}$ the joint decision of all the players except $i$. $X_t:=(X_1\times\cdots\times X_N)\cap X_{0,t}$ refers to the time-varying constraints of players, where $X_{0,t}:=\{x=\col\{x_i\}_{i\in[N]}\in\R^d\mid g_t(x):=\sum_{i=1}^Ng_{i,t}(x_i)\leq{\bf 0}\}$ denotes the  time-varying coupled inequality constraints. Then player $i$'s decision set is denoted by $X_{i,t}(x_{-i,t}):=\{x_i\in\R^d_i\mid (x_i,x_{-i,t})\in X_t\}$. Define $f_t:=(f_{1,t},\ldots,f_{N,t})$. Then such a game is denoted as $\Gamma_t:=\Gamma(\V,X_t,f_t)$.

For each player $i\in[N]$, $f_{i,t}$ and $g_{i,t}$-relevant information are revealed only after making its decision $x_{i,t}$ at current time $t$. Then, the observed information is utilized to obtain the next decision $x_{i,t+1}$.

%Hopefully, all players' decisions reach a GNE $x_t^*=(x_{i,t}^*,x_{-i,t}^*)$, from which one cannot benefit from changing its decision unilaterally. Here, a GNE $x_t^*$ of $\Gamma_t:=\Gamma(\V,X_t,f_t)$ satisfies that for any $i\in[N]$, it holds
%\begin{align*}
%f_{i, t}(x_{i, t}^*, \sigma(x_t^*)) \leq f_{i, t}\left(x_{i}, {\color{blue}\frac{1}{N}}\psi_i(x_{i})+\sigma(x_{-i, t}^*)\right),&\notag\\
%\forall x_{i} \in X_{i, t}(x_{-i, t}^*).&
%\end{align*}
%where $X_{i, t}(x_{-i, t}^*):=\{x_i \in \R^{d_i} \mid(x_i, x_{-i, t}^*) \in X_t\}$.

Next, some standard assumptions on the studied game are listed, which are also made in \cite{meng24ban,liu23,lu21online,zuo2023distributed,lin24private}.
\begin{ass}\label{ass:set}
$~~~~~~~~~~$
\begin{itemize}
 \item [i)] The sets $X_i$, $i\in[N]$, are nonempty, convex, and compact.
 \item [ii)] The feasible set $X_t$ is nonempty and satisfies Slater's constraint qualification for any $t\geq 0$.
%\item [vi)] $\|h_{i,t}\|\leq B$.
\item [iii)] For any $i\in[N]$ and $t\geq 0$, $f_{i,t}(x_i,\sigma)$ is differentiable in $(x_i,\sigma)\in X_i\times\R^n$. Moreover, $f_{i,t}(x_i,\sigma(x))$ is convex in $x_i\in X_i$ for any given $x_{-i}\in \R^{d-d_i}$.%, and $g_{ij,t}(x_i)$ is differentiable and convex, where $g_{ij,t}$ is the $j$-th component of $g_{i,t}$.
\item [iv)]  For any $i\in[N]$ and $t\geq 0$, $g_{ij,t}(x_i)$ is differentiable and convex on $X_i$, where $g_{ij,t}$ is the $j$-th component of $g_{i,t}$.
    %Moreover, $\nabla g_{i,t}$ is uniformly bounded, i.e., there exists $M>0$ such that for any $x_i\in X_i$, $\|\nabla g_{i,t}(x_i)\|\leq M$, where $\nabla g_{i,t}(x_i):=(\nabla g_{i1,t}(x_i),\ldots,\nabla g_{im,t}(x_i))$.
\item [v)] For any $i\in[N]$, $\psi_i$ is differentiable and $L_\sigma$-Lipschitz continuous $(L_\sigma>0)$, i.e.,
 \begin{align*}
     \|\psi_i(x_i)-\psi_i(x_i')\|\leq L_\sigma\|x_i-x_i'\|,\ \forall x_i,x_i'\in X_i.
 \end{align*}
\end{itemize}
\end{ass}

Assumption \ref{ass:set}-i) has been widely utilized in online games \cite{lu21online,meng24ban,zuo2023distributed,duvocelle2023multiagent,liu23,lin24private}, mostly due to the fact the decision variables in reality are often bounded, such as the energy consumption in electricity markets. It can be further obtained from Assumption \ref{ass:set} that for any $i\in[N]$, $t\geq 0$, and $x_i\in X_i$, there exist  $L$, $M>0$ such that
\begin{align}\label{eq:bound}
\|x_i\|\leq L,\ \|g_{i,t}(x_i)\|\leq L,\ \|\nabla g_{i,t}(x_i)\|\leq M, %\|J_{i,t}(x)\|\leq L,
\end{align}
where $\nabla g_{i,t}(x_i):=(\nabla g_{i1,t}(x_i),\ldots,\nabla g_{im,t}(x_i))$.
%Moreover, $\nabla g_{i,t}$ is uniformly bounded, i.e., there exists $M>0$ such that for any $x_i\in X_i$, $\|\nabla g_{i,t}(x_i)\|\leq M$, where $\nabla g_{i,t}(x_i):=(\nabla g_{i1,t}(x_i),\ldots,\nabla g_{im,t}(x_i))$.

%A GNE $x_t^*=(x_{i,t}^*,x_{-i,t}^*)$ of $\Gamma_t:=\Gamma(\V,X_t,f_t)$ satisfies that for any $i\in[N]$, it holds
%\begin{align*}
%f_{i, t}(x_{i, t}^*, \sigma(x_t^*)) \leq f_{i, t}\left(x_{i}, {\color{blue}\frac{1}{N}}\psi_i(x_{i})+\sigma(x_{-i, t}^*)\right),&\notag\\
%\forall x_{i} \in X_{i, t}(x_{-i, t}^*).&
%\end{align*}}
%In fact,
 Denote by $\sigma(x_{-i}):=\frac{1}{N}\sum_{j=1,j\neq i}^N \psi_j(x_j)$ the aggregate of all players except player $i$. $x_t^*=(x_{i,t}^*,x_{-i,t}^*)$ is a GNE if and only if for any $i\in [N]$, $x_{i,t}^*$ is a resolution of the following optimization problem
\begin{align}\label{prob-i}
  &\min_{x_{i,t}\in X_i}~f_{i, t}\left(x_{i,t}, \frac{1}{N}\psi_i(x_{i,t})+\sigma(x_{-i, t}^*)\right)\notag\\
  &\text{s.t.}~g_{i,t}(x_{i,t})+\sum_{j=1,j\neq i}^N g_{j,t}(x_{j,t}^*)\leq{\bf 0}.
\end{align}
Define the Lagrangian function of problem (\ref{prob-i}) for each player $i\in [N]$ as
\begin{align*}
L_{i,t}(x_{i},\mu_{i,t};x_{-i,t}^*)&:=f_{i, t}\left(x_{i}, \frac{1}{N}\psi_i(x_{i})+\sigma(x_{-i, t}^*)\right)\notag\\
&\quad+{\mu_{i,t}}^{\T}\left(g_{i,t}(x_{i})+\sum_{j=1,j\neq i}^N g_{j,t}(x_{j,t}^*)\right),
\end{align*}
where $\mu\in\R^{m}_+$ is the Lagrange multiplier (dual variable) associated with the coupled constraints. Then if $x_t^*=(x_{i,t}^*, x_{-i, t}^*)$ is a GNE of $\Gamma_t$, by Karush-Kuhn-Tucker (KKT) conditions, there exist  multipliers $\mu_{i,t}^*$, $i\in[N]$ satisfying
\begin{align}\label{kkt}
&{\bf 0}\in \nabla_{x_i}f_{i,t}(x_{i,t}^*,\sigma(x_t^*))+\nabla g_i(x_{i,t}^*) \mu_{i,t}^* + \N_{X_i}(x_{i,t}^*),\notag\\
&{\bf 0}\in -g_t(x_t^*)+\N_{\R^m_+}(\mu_{i,t}^*),\ i\in[N],
\end{align}
where $\nabla_{x_i}f_{i,t}(x_{i,t}^*,\sigma(x_t^*)):= \nabla_1f_{i,t}(x_{i,t}^*,\sigma(x_t^*))+\frac{\nabla\psi_i(x_{i,t}^*)}{N}\nabla_2 f_{i,t}(x_{i,t}^*,\sigma(x_t^*))$. In the KKT conditions (\ref{kkt}), each $\mu_{i,t}^*$, $i\in[N]$ conforms to the same $g_t(x_t^*)$, indicating that the system is ill-posed. Instead, a vGNE satisfying  (\ref{kkt}) for some $\mu_t^*:=\mu_{i,t}^*=\mu_{j,t}^*$ is economically justifiable \cite{kulkarni2012variational}.  As a consequence, it has been commonly practical to seek the vGNE rather than the GNE \cite{lu21online,meng24ban,zuo2023distributed,liu23}.
%From the perspective of variational inequality, it follows from Assumption \ref{ass:set} that $x_t^*$ is a {\color{blue}vGNE} of $\Gamma_t$ if and only if $x_t^*$ is  a solution to the variational inequality
%\begin{align}\label{optcon}
%G_t(x_t^*)^{\T}(x-x_t^*)\geq 0,\ \forall x\in X_t,
%\end{align}
%where $G_t(x):=\col\{\nabla_{x_i}f_{i,t}(x_i,\sigma(x))\}_{i\in[N]}$ is the pseudo-gradient mapping of $\Gamma_t$.
Furthermore, the existence of $x_t^*$ can be guaranteed by Assumption \ref{ass:set} \cite{facchinei2003finite}.

However, in online game $\Gamma_t$, $f_{i,t}$ and $g_{i,t}$-relevant information are revealed to player $i$ only after the selection of $x_{i,t}$.
%for each $i\in[N]$, $f_{i,t}$ and $g_{i,t}$ are revealed to player $i$ only after making its decision $x_{i,t}$ at time $t$.
It is impossible for players to pre-compute a GNE $x_t^*$ of $\Gamma_t$. Therefore, a performance metric, called regret, is commonly employed for online games.  Motivated by \cite{meng24ban}, define the static regret for player $i$ as
\begin{align}\label{def:reg-i}
\mathcal{R}_i(T)&:=\sum_{t=1}^T\bigg[f_{i,t}(x_{i,t},\sigma(x_t))\notag\\
&\quad-f_{i,t}\bigg(x_{i}^\star, \frac{1}{N}\psi_i(x_{i}^\star)+\sigma(x_{-i, t}){\bigg)}\bigg],
\end{align}
where $T$ is the learning time horizon %and
%\begin{align*}
%x^\star_{i,t}\in\arg\min_{x_i\in{\bigcap_{t=1}^TX_{i,t}(x_{-i,t})}}\sum_{t=1}^TJ_{i,t}(x_{i,t}, x_{-i,t}).
%\end{align*}
and $x^\star_{i}$ is the optimal static solution to the following optimization problem given $x_{-i,t}$, $t\in[T]$:
\begin{align}\label{prob}
  &\min_{x_{i,t}\in X_i}~\sum_{t=1}^Tf_{i, t}\left(x_{i,t}, \frac{1}{N}\psi_i(x_{i,t})+\sigma(x_{-i, t})\right)\notag\\
  &\text{s.t.}~g_{i,t}(x_{i,t})+\sum_{j=1,j\neq i}^N g_{j,t}(x_{j,t})\leq{\bf 0},\ \forall t\in[T].
\end{align}
In order to ensure that problem (\ref{prob}) is feasible, the set $\bigcap_{t=1}^TX_{i,t}(x_{-i,t})$ is assumed to be nonempty.

Furthermore, since the players' decisions cannot satisfy the coupled inequality constraints at every time, it is also indispensable to introduce  the following constraint violation metric
\begin{align}\label{def:reg-g}
\mathcal{R}_g(T):=\left\|\left[\sum_{t=1}^Tg_t(x_t)\right]_+\right\|.
\end{align}

Intuitively, the regret (\ref{def:reg-i}) of player $i$ reflects the difference between player $i$'s actual cumulative payoff and optimal cumulative payoff. Moreover, definition (\ref{def:reg-g}) implicitly allows constraint violations at certain times to be compensated by strictly feasible decisions
at other times. In general, an online algorithm is considered to  be no-regret if all regrets of players  and constraint violation are sublinear with respect to $T$, i.e., $\mathcal{R}_i(T)=o(T)$ and $\mathcal{R}_g(T)=o(T)$. The goal of this paper is to design an online algorithm ensuring that $\mathcal{R}_i(T)$ and $\mathcal{R}_g(T)$ increase sublinearly with high probability.

%For function $f_{i,t}(x_i, \sigma)$, let $\nabla_{1} f_{i,t}(x_i, \sigma):=\frac{\partial f_{i,t}(x_i, \sigma)}{\partial x_i}$ and $\nabla_2 f_{i,t}(x_i, \sigma):= \frac{\partial f_{i,t}(x_i, \sigma)}{\partial \sigma}$ denote its gradients with respect to the first variable $x_i$ and the second variable $\sigma$, respectively.
To proceed, define
\begin{align*}
p_{i,t}(x_i,y) :=\bigg(\nabla_{1} f_{i,t}(x_i, \sigma)+\frac{\nabla\psi_i(x_i)}{N} \nabla_2 f_{i,t}(x_i, \sigma)\bigg)\bigg|_{\sigma=y}.
\end{align*}
One can see that
\begin{align*}
\nabla_{x_i} f_{i,t}(x_i, \sigma(x))=p_{i,t}(x_i,\sigma(x)).
\end{align*}

\begin{ass}\label{ass:fun} For any $i\in [N]$ and $t\geq 0$,
 %\item [i)] $\psi_i$ is differentiable and $L_\sigma$-Lipschitz continuous $(L_\sigma>0)$, i.e.,
 %\begin{align*}
  %   \|\psi_i(x)-\psi_i(x')\|\leq L_\sigma\|x-x'\|,\ \forall x,x'\in X_i.
 %\end{align*}
 $p_{i,t}(x_i,y)$ is $L_f$-Lipschitz continuous $(L_f>0)$ with respect to $y$ for any given $x_i\in X_i$, i.e.,
     \begin{align*}
     \|p_{i,t}(x_i,y)-p_{i,t}(x_i,y')\|\leq L_f\|y-y'\|,\ \forall y, y'\in \R^n.
 \end{align*}
\end{ass}

Assumption \ref{ass:fun} is commonly employed in existing studies  \cite{liu23,zuo2023distributed,lin24private,wang2024distributed}.

%\begin{rem}
%In fact, the compact set $Y$ in Assumption \ref{ass:fun}-ii) can be determined based on   (\ref{lem:mu-sigma-2}) of Lemma \ref{lem-consensus}, which is derived in latter analysis.
%\end{rem}

\subsection{High Probability Bound}

\begin{defn}\cite{lu24Ongam} (High probability bounds)
For regrets $\mathcal{R}_i(T), i \in [N]$ and constraint violation $\mathcal{R}_g(T)$,
$\mathcal{R}_i(T)$ and $\mathcal{R}_g(T)$ are said to have high probability bounds if for any
$\delta \in(0,1)$, there hold that $
\mathcal{R}_i(T) \leq \mathcal{O}\left(r(T)\left(\ln \frac{1}{\delta}\right)^s\right)$ and $\mathcal{R}_g(T) \leq \mathcal{O}\left(r(T)\left(\ln \frac{1}{\delta}\right)^s\right)$ with probability at least $1-\delta$
for some function $r: \mathbb{R} \rightarrow \mathbb{R}$ and $s\in[0,1]$.
Moreover, if $r(T)$ is sublinear with respect to $T$, i.e., $r(T)=o(T)$, $\mathcal{R}_i(T)$ and $\mathcal{R}_g(T)$ are said to increase sublinearly with high probability.
\end{defn}

%\begin{defn} [Convergence in probability]
%Let $\left\{V_t\right\}$ a sequence of random variables and $V$ be a random variable, we say that $V_t$ converges to $V$ in probability if ${\color{blue}\P}\left(\|V_t-V\|\leq \mathcal{O}(r(t) h(\delta))\right) \geq 1-\delta, \ \forall \delta \in(0,1)$ with $r(t)>0$ satisfying $\lim _{t \to \infty} r(t)=0$. Particularly, we say $V_t$ converges to $V$ in low probability if $h(\delta)=\frac{1}{\delta}$, and $V_t$ converges to $V$ in high probability if $h(\delta)=\ln \frac{1}{\delta}$.

%\begin{defn} [Convergence in high probability]
%Let $\left\{V_t\right\}$ a sequence of random variables and $V$ be a random variable, we say that $V_t$ converges to $V$ in high probability if ${\color{blue}\P}\left(\|V_t-V\|\leq \mathcal{O}(r(t) h(\delta))\right) \geq 1-\delta, \ \forall \delta \in(0,1)$ with $r(t)>0$ satisfying $\lim _{t \to \infty} r(t)=0$ and $h(\delta)=\ln \frac{1}{\delta}$.
%\end{defn}

%We require $F_{i,t}(x_{i},\sigma,\xi_{i})$ to be differentiable in $(x_i,\sigma)\in X_i\times\R^n$ for any given $\xi_{i}\in\R^{q_i}$.
For function $F_{i,t}(x_i, \sigma, \xi_i)$, let $\nabla_{1} F_{i,t}(x_i, \sigma, \xi_i):=\frac{\partial F_{i,t}(x_i, \sigma, \xi_{i})^{\T}}{\partial x_i}$ and $\nabla_2 F_{i,t}(x_i, \sigma, \xi_i):= \frac{\partial F_{i,t}(x_i, \sigma, \xi_{i})^{\T}}{\partial \sigma}$ denote gradients with respect to the first variable $x_i$ and the second variable $\sigma$, respectively. Then define
\begin{align}\label{eq-stogra-compu}
q_{i,t}(x_i, y, \xi_{i})&:=\bigg(\nabla_{1} F_{i,t}(x_i, \sigma, \xi_{i})\notag\\
&\quad+\frac{\nabla\psi_i(x_i)}{N} \nabla_2 F_{i,t}(x_i, \sigma, \xi_{i})\bigg)\bigg|_{\sigma=y}.
\end{align}
%In the stochastic scenario, due to the lack of closed-form expressions for gradients of expectation-valued cost functions, each player $i$ can only access the noisy gradient $q_{i,t}(x_i, y, \xi_{i,t})$ of its cost function $f_{i,t}(x_i,\sigma)$.

Assume that at every time $t\geq 0$, each player $i\in[N]$ is able to obtain a random data vector $\xi_{i,t}$, which is randomly sampled from $\xi_i$. Given $x_i$, $y$, and $\xi_{i,t}$, there exists a stochastic first-order oracle that returns a sampled gradient $q_{i,t}(x_i,y,\xi_{i,t})$. %At every time $t\geq 0$, each player $i\in[N]$, given $x_i$, $y$, and $\xi_{i,t}$, is able to query a  stochastic first-order oracle to obtain a sampled gradient $q_{i,t}(x_i,y,\xi_{i,t})$.%, which is an unbiased estimator of $p_{i,t}(x_i,y)$.
Define the natural filtration
\begin{align}\label{def:F}
&\F_0:=\{\emptyset,\Omega\},\notag\\
&\F_t:=\sigma\{\xi_{i,s}: i\in[N], 0\leq s\leq t-1\},\ \forall t\geq 1.
\end{align}
\begin{ass}\label{ass:sg} For any $t\geq 0$, $i\in [N]$, $x_i\in X_i$, $y_i\in\R^n$, and $x_i, y\in\F_t$, there hold
\begin{itemize}
 \item [i)] $\E[q_{i,t}(x_i, y, \xi_{i,t})|\F_t]=p_{i,t}(x_i,y)$;
 \item [ii)] $\E[\exp (\|q_{i,t}(x_i, y, \xi_{i,t})-p_{i,t}(x_i,y)\|^2/\iota^2)|\F_t] \leq \exp (1)$ for some constant $\iota>0$;
\item [iii)] $\|q_{i,t}(x_i, y, \xi_{i,t})\|\leq S$ for some $S>0$.
\end{itemize}
\end{ass}

By Jensen's inequality, it follows from Assumption \ref{ass:sg}-i) and \ref{ass:sg}-iii) that $
\|p_{i,t}(x_i,\sigma(x)\|\leq S$. Assumption \ref{ass:sg}-i) indicates that the noisy gradient $q_{i,t}(x_i, y, \xi_{i,t})$ is conditionally unbiased with respect to $\F_t$.  Assumption \ref{ass:sg}-ii) characterizes a sub-Gaussian noise model, where the tail behavior of the noise distribution is dominated by tails of a Gaussian distribution. Moreover, using the Jensen's inequality, Assumption \ref{ass:sg}-ii) implies that $\E[\|q_{i,t}(x_i, y, \xi_{i,t})-p_{i,t}(x_i,y)\|^2]\leq \iota^2$, which is utilized to analyze the convergence of distributed stochastic algorithms in expectation \cite{lei22SNE,FRANCI2022,zheng2023distributed,wang2024distributed,Huang23,lei2023distributed}. However, Assumption \ref{ass:sg}-ii) is frequently made in literature to establish finer high probability guarantees
%, {\color{red}which is highly relevant in practice }
\cite{nemirovski2009robust,li2019convergence,liu2023high,lu24Ongam}.
In fact, the noises that follow the uniform distribution over a compact set or Gaussian distribution satisfy the sub-Gaussian noise model.
%Also, in real-world applications, various estimation methods exist to achieve stochastic gradients that satisfy Assumption \ref{ass:sg}, such as the score function estimators and pathwise derivative estimator \cite{schulman2015gradient}.
Assumption \ref{ass:sg}-iii) requires the uniform boundedness of noises, which is reasonable, particularly in scenarios such as measurement noises in sensor networks and fluctuations of market demands in economic problems. It is noteworthy that Assumption \ref{ass:sg}-iii) was also employed in prior works \cite{zhang22,lei2023distributed} to derive high probability convergence results.
%{\color{blue}In real-world distributed scenarios, such as measurement noises in sensor networks and fluctuations of market demands in economic problems, the stochastic disturbances are bounded, and thus Assumption \ref{ass:sg}-ii) and \ref{ass:sg}-iii) are generally satisfied.}
%It is noteworthy that Assumption \ref{ass:sg} was also employed in prior works \cite{zhang22,lei2023distributed} to derive high probability convergence results.

\section{MAIN RESULTS}\label{sec3}
In this section, an online distributed stochastic algorithm for our studied games  is proposed. We also analyze  the high probability bounds on the regrets and constraint violation for the proposed algorithm.
%In this section, an online distributed stochastic algorithm based on primal-dual and push-sum schemes is devised. Moreover, the high probabilities bounds on regrets and constraint violation are presented for the proposed algorithm.

To begin with, for each player $i\in[N]$, the regularized Lagrangian function at time instant $t$ is defined as
\begin{align}\label{def:lag}
\mathcal{L}_{i,t}(x_{i},\mu;x_{-i}):=f_{i,t}(x_{i},\sigma(x))+\mu^{\T}g_t(x)-\frac{\beta}{2}\|\mu\|^2,
\end{align}
where $\mu\in\R^{m}_+$ is the Lagrange multiplier (dual variable) associated with the coupled constraints  and $\beta>0$ is the regularization parameter. A standard primal-dual algorithm to seek a saddle point of $\mathcal{L}_{i,t}$ can be designed as
\begin{subequations}\label{alg:p-d-cen}
\begin{align}
x_{i,t+1}&=\Pi_{X_i}[x_{i,t}-\alpha_t(p_{i,t}(x_{i,t},\sigma(x_t))+\nabla g_i(x_{i,t})\mu_t)],\label{alg:p-d-cen-x}\\
\mu_{t+1}&=[\mu_{t}+\gamma_t(g_{t}(x_{t})-\beta_t\mu_{t})]_+,\label{alg:p-d-cen-dual}
\end{align}
\end{subequations}
where $\alpha_t>0$ and $\gamma_t>0$ are stepsizes utilized in the primal and dual updates, respectively. However, in algorithm (\ref{alg:p-d-cen}), it is noted that the exact gradient $p_{i,t}(x_{i,t},\sigma(x_t))$ is usually unavailable. Moreover, each individual player cannot access the information on global aggregative variable $\sigma(x_t)$, nonlinear constraint function $g_t(x_t)$, and common Lagrange multiplier $\mu_t$, thus requiring  a central unit to compute and communicate with all the players. To this end, i) the stochastic sample gradient $q_{i,t}$ rather than the exact gradient $p_{i,t}$ is utilized, where $q_{i,t}$ is defined in \eqref{eq-stogra-compu}.  ii) A local copy $\mu_{i,t}\in\R^m$ of the multiplier $\mu_t$ is introduced, and the  dynamic consensus technique is incorporated to enforce consensus among $\mu_{i,t}$'s. iii) By virtue of the push-sum protocol \cite{nedic15}, the auxiliary variable $z_{i,t}\in\R$ is introduced, which aims to eliminate the imbalance of the communication graphs through tracking the right-hand eigenvector of $\mathcal{W}_t$ related to the eigenvalue ${\bf 1}$. iv) Motivated by the idea of gradient tracking \cite{pu2021DSGT}, the variable $\sigma_{i,t}\in\R^{n}$ of player $i$ is employed to track $\sigma(x_t)$. To conclude, the developed online distributed stochastic primal-dual push-sum algorithm is presented in Algorithm \ref{alg:primaldual}.

\begin{algorithm}[!t]\caption{Online Distributed Stochastic Primal-dual Push-sum}\label{alg:primaldual}

Set $T\geq 4$. Each player $i$, $i\in[N]$ maintains variables $z_{i,t}\in\R$, $x_{i,t}\in\R^{d_i}$,  $\mu_{i,t}\in\R^m$, and $\sigma_{i,t}\in\R^n$ at time $t=0,1,\ldots,T$.

 {\bf Initialization:} For any $i\in[N]$, initialize $z_{i,0}=1$, $x_{i,0}\in X_i$, $\mu_{i,0}={\bf 0}$, and $\sigma_{i,0}=\psi_{i}(x_{i,0})$.

{\bf Iteration:} For $t\geq 0$, each player $i$ updates as follows:
\begin{subequations}\label{alg:dist}
\begin{align}
z_{i,t+1}&=\sum_{j=1}^Nw_{ij,t}z_{j,t},\label{alg:dist_1}\\
\hat{\mu}_{i,t}&=\sum_{j=1}^Nw_{ij,t}\mu_{j,t},\ \hat{\sigma}_{i,t}=\sum_{j=1}^Nw_{ij,t}\sigma_{j,t},\label{alg:dist_2}\\
        s_{i,t+1}&=q_{i,t}\left(x_{i,t},\frac{\hat{\sigma}_{i,t}}{z_{i,t+1}},\xi_{i,t}\right)+\nabla g_{i,t}(x_{i,t})\frac{\hat{\mu}_{i,t}}{z_{i,t+1}},\label{alg:dist_3}\\
        x_{i,t+1}&=\Pi_{X_i}[x_{i,t}-\alpha_ts_{i,t+1}],\label{alg:dist_4}\\
        %x_{i,t+1}&=(1-\eta_t)x_{i,t}+\eta_t \tilde{x}_{i,t+1}\label{alg:dist_5}\\
        \mu_{i,t+1}&=[\hat{\mu}_{i,t}+\gamma_t(g_{i,t}(x_{i,t})-\beta_t\hat{\mu}_{i,t})]_+,\label{alg:dist_6}\\
        \sigma_{i,t+1}&=\hat{\sigma}_{i,t}+\psi_{i}(x_{i,t+1})-\psi_{i}(x_{i,t}),\label{alg:dist_7}
\end{align}
\end{subequations}
where $\alpha_t, \beta_t, \gamma_t \in (0,1]$ are non-increasing parameters to be determined.
\end{algorithm}

Algorithm \ref{alg:primaldual} is distributed since each player is capable to update its variables based solely  on local information exchanges. Based on \eqref{def:F} and iterations in Algorithm 1, it is seen
 $x_{i,t}, \sigma_{i,t}\in\F_t$ and $\mu_{i,t}\in\F_{t-1}$.
%$x_{i,t}$ and $\sigma_{i,t}$ are adapted to $\F_t$, while $\mu_{i,t}$ is adapted to $\F_{t-1}$.

\begin{rem}
%The stochastic error incurred by the noisy gradient can significantly affect the performance of the algorithm, potentially preventing it from converging. Therefore, we need to carefully select stepsizes to suppress the noise so that Algorithm \ref{alg:primaldual} can successfully seek the GNEs. estimated error between the sampled gradients and exact gradients
The utilization of noisy gradients, which is easy to received in practice as opposed to exact gradients, however,  can significantly affect the performance of the algorithm and also  contributes to the technical challenge of our analysis. Moreover, in the dual update (\ref{alg:dist_6}), an additional term $-\gamma_t\beta_t\hat{\mu}_{i,t}$ appears, which originates from  the regularized Lagrangian function (\ref{def:lag}) in order to guarantee the boundedness of dual variable $\mu_{i,t}$. Furthermore, note that the global information $\sigma(x_t)$, then $\nabla_1 F_{i,t}(x_{i,t},\sigma(x_t),\xi_{i,t})$ and $\nabla_2 F_{i,t}(x_{i,k},\sigma(x_t),\xi_{i,t})$ are not accessible to any individual player in the distributed framework. Hence, we introduce $\sigma_{i,t}$ to track $\sigma(x_t)$ for player $i$. Additionally, under unbalanced graphs where $\mathcal{W}_t=[w_{ij,t}]$ is column-stochastic, $\sum_{j=1}^Nw_{ij,t}\mu_{j,t}$, as multipliers in the Lagrangian function, will eventually reach different vector values. To address this, an identical Lagrange multiplier cannot be estimated as expected. Similarly, $\sum_{j=1}^Nw_{ij,t}\sigma_{j,t}$ cannot successfully track $\sigma(x_t)$ alone. To bypass these challenges, the push-sum approach is adopted to address the imbalance of the interaction graphs. Correspondingly, it is necessary to analyze the introduced consensus errors, which will also bring additional difficulties.
%Note that it is reasonable for each agent in the push-sum method to know its own out-degree \cite{}. Actually, as pointed out in \cite{}, the information on the out-degree for each individual agent can be known by virtue of bidirectional exchange of “hello” messages during only a single round of communication.
\end{rem}

\begin{rem}
In \cite{zuo2023distributed}, a distributed online push-sum mirror descent algorithm was devised for online aggregative games in the deterministic case. By contrast, this work utilized stochastic sampled gradients to address the unavailability of exact gradients of expecation-valued cost functions. The offline algorithms in \cite{zheng2023distributed,Gadjov21,FRANCI2022,Belgioioso21} are based on operator splitting schemes and apply only to coupled affine constraints.
In \cite{FRANCI2022,Gadjov21}, the Laplacian matrix was used in algorithm design for fixed undirected graphs, and \cite{liu23,lei2023distributed} considered only balanced networks, while Algorithm \ref{alg:primaldual} employs the push-sum technique, thereby enabling it to handle time-varying unbalanced graphs. Furthermore, unlike the gradient-response scheme adopted in our algorithm, \cite{Huang23} employed an augmented BR scheme, which results in a double-loop algorithm since BR computations are usually challenging. Also, the augmented Lagrangian in \cite{Carnevale24Tracking-based} differs from ours, giving rise to a different algorithmic framework.
%Furthermore, unlike existing offline counterparts in \cite{lei22SNE,wang2024distributed}, where convergence in expectation was established, we delve into providing high probability bounds of regrets and constraint violation under the sub-Gaussian noise model to ensure the effectiveness of Algorithm \ref{alg:primaldual}, which enhances practicality while making theoretical analysis more complex and challenging.
\end{rem}

In what follows, some useful lemmas are provided. {color{blue}Define}
\begin{align*}%\label{eq:def-r}
r:=\inf _{t=0,1, \ldots}\left(\min _{i \in[N]}[\mathcal{W}_t \cdots \mathcal{W}_0 {\bf 1}_N]_i\right).
\end{align*}
 Under Assumption \ref{ass:graph}, it follows from Lemma 3 in \cite{li21dop} that $r\leq 1$.

First, the following lemma provides some bounds on the variables $z_{i,t}$ and $\mu_{i,t}$.
\begin{lem}\label{lem:mu-w-z}
Under Assumptions \ref{ass:graph} and \ref{ass:set}, for any $t\geq 0$ and $i\in[N]$, there hold
\begin{align}
&r\leq z_{i,t}\leq N,\label{lem-z}\\
&\|\mu_{i,t}\|\leq \frac{z_{i,t+1}L}{\beta_tr^2\underline{w}},\ \|\hat{\mu}_{i,t}\|\leq \frac{z_{i,t+1}L}{\beta_tr^2}\label{lem-mu},
\end{align}
where $\underline{w}$ is given in  Assumption \ref{ass:graph}-i).
\end{lem}

\begin{proof}
See Appendix \ref{proof:lem:mu-w-z}.
\end{proof}

Define
\begin{align*}
\tilde{\mu}_{i,t+1}:=\frac{\hat{\mu}_{i,t}}{z_{i,t+1}},\ \bar{\mu}_t:=\frac{1}{N}\sum_{i=1}^N\mu_{i,t},\\
\tilde{\sigma}_{i,t+1}:=\frac{\hat{\sigma}_{i,t}}{z_{i,t+1}},\ \bar{\sigma}_t:=\frac{1}{N}\sum_{i=1}^N\sigma_{i,t}.
\end{align*}

We next estimate  consensus errors $\|\tilde{\mu}_{i,t+1}-\bar{\mu}_t\|$ and $\|\tilde{\sigma}_{i,t+1}-\bar{\sigma}_t\|$.  (\ref{alg:dist_6}) and (\ref{alg:dist_7}) can be rewritten as
\begin{align}
\mu_{i,t+1}&=\hat{\mu}_{i,t}+\varepsilon_{\mu_{i,t+1}},\label{mu}\\
\sigma_{i,t+1}&=\hat{\sigma}_{i,t}+\varepsilon_{\sigma_{i,t+1}},\label{sigma}
\end{align}
where
\begin{align}
\varepsilon_{\mu_{i,t+1}}&:=[\hat{\mu}_{i,t}+\gamma_t(g_{i,t}(x_{i,t})-\beta_t\hat{\mu}_{i,t})]_+ -\hat{\mu}_{i,t},\label{mu-error}\\
\varepsilon_{\sigma_{i,t+1}}&:=\psi_i(x_{i,t+1})-\psi_i(x_{i,t}).\label{sigma-error}
\end{align}

%(\ref{alg:dist_6}) and {\color{blue}(\ref{alg:dist_7})} can be rewritten as
%\begin{align}
%\mu_{i,t+1}&=\hat{\mu}_{i,t}+\varepsilon_{\mu_{i,t+1}},\label{mu}\\
%\sigma_{i,t+1}&=\hat{\sigma}_{i,t}+\varepsilon_{\sigma_{i,t+1}},\label{sigma}
%\end{align}
%where
%\begin{align}
%\varepsilon_{\mu_{i,t+1}}&:=[\hat{\mu}_{i,t}+\gamma_t(g_{i,t}(x_{i,t})-\beta_t\hat{\mu}_{i,t})]_+ -\hat{\mu}_{i,t},\label{mu-error}\\
%\varepsilon_{\sigma_{i,t+1}}&:=\psi_i(x_{i,t+1})-\psi_i(x_{i,t}).\label{sigma-error}
%\end{align}

\begin{lem}\label{lem-consensus}
Under Assumptions  \ref{ass:graph}, \ref{ass:set}, and  \ref{ass:sg}, there holds
\begin{align}
\|\tilde{\mu}_{i,t+1}-\bar{\mu}_t\|&\leq \frac{16N\sqrt{m}L}{r}\sum_{k=1}^t\vartheta^{t-k}\gamma_{k-1},\label{lem:mu-sigma-1}\\
\|\tilde{\sigma}_{i,t+1}-\bar{\sigma}_t\|&\leq \frac{8}{r}\bigg(\vartheta^t\|\sigma_0\|_1+N\sqrt{n}L_\sigma S\sum_{k=1}^t\vartheta^{t-k}\alpha_{k-1}\notag\\
&\quad + \frac{N\sqrt{n}L_\sigma LM}{r^2}\sum_{k=1}^t\vartheta^{t-k}\frac{\alpha_{k-1}}{\beta_{k-1}}\bigg),\label{lem:mu-sigma-2}
\end{align}
where $\vartheta \leq\left(1-\frac{1}{N^{N U}}\right)^{\frac{1}{N U}}\in(0,1)$.
\end{lem}

\begin{proof}
See Appendix \ref{proof:lem-consensus}.
\end{proof}

Next, a crucial lemma is derived, which provides the high probability bound of the term $ \langle p_{i,t}(x_{i,t},\tilde{\sigma}_{i,t+1})-q_{i,t}(x_{i,t},\tilde{\sigma}_{i,t+1},\xi_{i,t}),x_{i,t}-x_i\rangle$. %by elaborately leveraging some stochastic analysis and inequality techniques.
%{\color{blue}This lays the foundation for establishing the high probability guarantees on regrets and constraint violation.}
\begin{lem}\label{lem:hp}
Under Assumptions \ref{ass:set} and \ref{ass:sg}, for any $x_i\in X_i$ and $0<\delta<1$, with
probability at least $1-\delta$,
\begin{align}\label{lem-hp}
&\sum_{t=1}^T\alpha_t\langle p_{i,t}(x_{i,t},\tilde{\sigma}_{i,t+1})-q_{i,t}(x_{i,t},\tilde{\sigma}_{i,t+1},\xi_{i,t}),x_{i,t}-x_i\rangle\notag\\
&\leq 2L\iota\sum_{t=1}^T\alpha_t^2+2L\iota\ln\frac{1}{\delta}.
\end{align}
\end{lem}
\begin{proof}
 See Appendix \ref{proof:lem:hp}.
\end{proof}
\begin{rem}
When analyzing the bound in expectation, Assumption \ref{ass:sg}-i) implies that the term in \eqref{lem-hp} vanishes. Unfortunately, the error between the sampled gradient and exact gradient on the high probability bound cannot be eliminated, and thus must be carefully handled by elaborately leveraging some stochastic analysis and inequality techniques, such as Markov's inequality, exponential inequality, martingale arguments.
%Lemma 3 reflects that the accumulated error grows at most sublinearly with high probability, laying the foundation for establishing the high probability guarantees on regrets and constraint violation.
\end{rem}

The recursive inequality relations on the errors of the primal and dual variables are revealed in the subsequent two lemmas,  respectively.

\begin{lem}\label{lem:optimal}
Under Assumptions \ref{ass:graph}-\ref{ass:sg}, for any $x_i\in X_i$, there holds
\begin{align}\label{lem-optimal}
&\langle p_{i,t}(x_{i,t},\sigma(x_t)),x_{i,t}-x_i\rangle\notag\\
&\leq S^2\alpha_t+\frac{16Ll_f}{r}\|\sigma_0\|_1\vartheta^t\notag\\
&\quad+\frac{16LNSl_f\sqrt{n}L_\sigma}{r}\sum_{k=1}^t\vartheta^{t-k}\alpha_{k-1}\notag\\
&\quad+ \frac{16N\sqrt{n}L_\sigma L^2l_fM}{r^3}\sum_{k=1}^t\vartheta^{t-k}\frac{\alpha_{k-1}}{\beta_{k-1}}\notag\\
&\quad+\frac{32 L^2N\sqrt{m}}{r} \sum_{k=1}^{t} \vartheta^{t-k}\gamma_{k-1}\notag\\
&\quad+\langle p_{i,t}(x_{i,t},\tilde{\sigma}_{i,t+1}))-q_{i,t}(x_{i,t},\tilde{\sigma}_{i,t+1},\xi_{i,t}),x_{i,t}-x_i\rangle\notag\\
&\quad+\frac{1}{2\alpha_t}\|x_i-x_{i,t}\|^2-\frac{1}{2\alpha_t}\|x_i-x_{i,t+1}\|^2\notag\\
&\quad + \bar{\mu}_t^{\T}(g_{i, t}(x_i)-g_{i, t}(x_{i, t}))+\frac{M^2L^2}{ r^4}\frac{\alpha_t}{\beta_t^2}.
\end{align}
\end{lem}

\begin{proof}
 See Appendix \ref{proof:lem:optimal}.
\end{proof}

\begin{lem}\label{lem:mu-rel}
Under Assumptions  \ref{ass:graph} and \ref{ass:set}, there holds
\begin{align}\label{lem-mu-rel}
&\|\bar{\mu}_{t+1}-\mu\|^2\notag \\
&\leq \|\bar{\mu}_{t}-\mu\|^2+2\left(1+\frac{2}{r}\right)\left(1+\frac{N^2}{r^4}\right)L^2\gamma_t^2\notag \\
&\quad+\frac{32N\sqrt{m}L^2}{r}\left(2+\frac{N}{r^2}\right)\gamma_t\sum_{k=1}^t\vartheta^{t-k}\gamma_{k-1}\notag\\
&\quad+\frac{2\gamma_t}{N}g_t(x_t)^{\T}(\bar{\mu}_t-\mu)+N\gamma_t\beta_t\|\mu\|^2.
\end{align}
\end{lem}

\begin{proof}
 See Appendix \ref{proof:lem:mu-rel}.
\end{proof}

Equipped with the above preparations, we are going to provide the high probability bounds on $\mathcal{R}_i(T)$ and $\mathcal{R}_g(T)$.

\begin{lem}\label{lem-regret}
Let Assumptions \ref{ass:graph}-\ref{ass:sg} be satisfied. For any $i\in[N]$ and $0<\delta<1$, with probability at least $1-\delta$, the regret (\ref{def:reg-i}) and constraint violation  (\ref{def:reg-g}) generated by Algorithm \ref{alg:primaldual} are bounded by
\begin{align}
\mathcal{R}_i(T)&\leq S_1(T)+S_2(T)+S_3(T), \label{lem-regret-i}\\
%&\leq \frac{16Ll_f\vartheta\|\sigma_0\|_1}{r(1-\vartheta)}+B^2\sum_{t=1}^T\alpha_t+\frac{2LNSl_f\sqrt{N}L_\sigma}{1-\vartheta}\sum_{t=1}^{T}\alpha_{t-1}\notag\\
%&\quad+\frac{16(2 N+3N^2)L^2\sqrt{m}}{r(1-\vartheta)}\sum_{t=1}^{T}\gamma_{t-1} +2N(1+\frac{2}{r})L^2\sum_{t=1}^T\gamma_t\notag\\
%&\quad+2L\sigma\frac{\sum_{t=1}^T\alpha_t^2}{\alpha_T}+\frac{2L\sigma\ln\frac{1}{\delta}}{\alpha_T}+\frac{2L^2}{\gamma_{T+1}}+\frac{L^2}{2\gamma_{T+1}\beta_{T+1}^2}\notag\\
%&\quad +\frac{M^2L^2}{r^2}\sum_{t=1}^T\frac{\alpha_t}{\beta_t^2},
\mathcal{R}_g(T)&\leq\sqrt{S_4(T)(S_1(T)+S_2(T)+S_5(T))}, \label{lem-regret-g}
%&\leq \frac{16Ll_f\vartheta\|\sigma_0\|_1}{r(1-\vartheta)}+B^2\sum_{t=1}^T\alpha_t+\frac{2LNSl_f\sqrt{N}L_\sigma}{1-\vartheta}\sum_{t=1}^{T}\alpha_{t-1}\notag\\
%&\quad+\frac{16(2 N+3N^2)L^2\sqrt{m}}{r(1-\vartheta)}\sum_{t=1}^{T}\gamma_{t-1} +2N(1+\frac{2}{r})L^2\sum_{t=1}^T\gamma_t\notag\\
%&\quad+2L\sigma\frac{\sum_{t=1}^T\alpha_t^2}{\alpha_T}+\frac{2L\sigma\ln\frac{1}{\delta}}{\alpha_T}+\frac{2L^2}{\alpha_{T+1}}+A(\mu_0)\notag\\
%&\quad +\frac{M^2L^2}{ r^2}\sum_{t=1}^T\frac{\alpha_t}{\beta_t^2}-\frac{(\mathcal{R}_g(T))^2}{2N^2\sum_{t=1}^T\beta_t}.
\end{align}
where
\begin{align*}
S_1(T)&:=\frac{16Ll_f\vartheta\|\sigma_0\|_1}{r(1-\vartheta)}+A_1\sum_{t=1}^{T}\alpha_{t-1}+A_2\sum_{t=1}^T\gamma_{t-1}\notag\\
&\quad+2L\iota\frac{\sum_{t=1}^T\alpha_t^2}{\alpha_T}+\frac{2L^2}{\alpha_{T+1}}+\frac{M^2L^2}{ r^4}\sum_{t=1}^T\frac{\alpha_t}{\beta_t^2}\notag\\
&\quad  + \frac{16N\sqrt{n}L_\sigma L^2l_fM}{r^3(1-\vartheta)}\sum_{t=1}^T\frac{\alpha_{t-1}}{\beta_{t-1}},\notag\\
S_2(T)&:=\frac{2L{\color{blue}\iota}\ln\frac{1}{\delta}}{\alpha_T},\notag\\
S_3(T)&:=\frac{NL^2}{2r^4\underline{w}^2\gamma_{T+1}\beta_{T+1}^2},\notag\\
S_4(T)&:=2N^2\sum_{t=1}^T\beta_t,\notag\\
S_5(T)&:= \frac{N^2L^2}{r^4\underline{w}^2\gamma_{T+1}\beta_{T+1}^2}+\frac{T^2L^2}{N\gamma_{T+1}\left(\sum_{t=1}^T\beta_t\right)^2},
\end{align*}
with $A_1:=S^2+\frac{16LNSl_f\sqrt{n}L_\sigma}{r(1-\vartheta)}$ and $A_2:=\left(1+\frac{2}{r}\right)\left(1+\frac{N^2}{r^4}\right)NL^2+\frac{16N\sqrt{m}L^2}{r(1-\vartheta)}\left(2(N+1)+\frac{N^2}{r^2}\right)$.
\end{lem}

\begin{proof}
See Appendix \ref{proof:lem-regret}.
\end{proof}

Let us now present the main results of this paper.  The parameters $\alpha_t$, $\beta_t$ and $\gamma_t$ are selected to be common as done in \cite{li21dop,meng2023online,meng24ban}.

\begin{thm}\label{thm:reg}
Suppose that Assumptions  \ref{ass:graph}-\ref{ass:sg} hold.  Let $\alpha_0=\beta_0=\gamma_0=1$, and for $t\geq 1$,
\begin{align}
\alpha_t=\frac{1}{t^{a_1}},\ \beta_t=\frac{1}{t^{a_2}},\ \gamma_t=\frac{1}{t^{a_3}},
\end{align}
where $0<a_1, a_3<1$, $a_1>2a_2>0$, and $2a_2+a_3<1$, then for any $i\in[N]$ and $0<\delta<1$, with probability at least $1-\delta$, the sequence $\{x_{i,1},\ldots,x_{i,T}\}$ generated by Algorithm \ref{alg:primaldual} satisfies
\begin{align}
\mathcal{R}_i(T)&=\O\bigg(T^{\max\{a_1,1-a_3,1-a_1+2a_2,2a_2+a_3\}}+T^{a_1}\ln\frac{1}{\delta}\bigg),\label{thm:reg-1}\\
\mathcal{R}_g(T)&=\O\bigg(T^{\max\{1-\frac{a_2+a_3}{2},1-\frac{a_1+a_2}{2},\frac{1+a_1+a_2}{2}
,\frac{1-a_1+a_3}{2}\}}\notag\\
&\quad+T^{\frac{a_1+a_2+a_3}{2}}\sqrt{\ln \frac{1}{\delta}}\bigg).\label{thm:reg-2}
\end{align}
\end{thm}
\begin{proof}
See Appendix \ref{proof:thm:reg}.
\end{proof}

%\begin{cor}\label{cor:sublinear}
%Under the same conditions as in Theorem \ref{thm:reg},   for any $i\in[N]$ and $0<\delta<1$,  with probability at least $1-\delta$, Algorithm \ref{alg:primaldual} achieves sublinearly bounded regrets and constraint violation.
%\end{cor}
%
%\begin{proof}
%According to the selection of $a_1, a_2$, and $a_3$, one has that $a_1, 1-a_3, 1-a_1+2a_2, 2a_2+a_3, 1-\frac{a_2+a_3}{2},1-\frac{a_1+a_2}{2},\frac{1+a_1+a_2}{2}
%,\frac{1-a_1+a_3}{2},\frac{a_1+a_2+a_3}{2}
%\in(0,1)$. This ends the proof.
%\end{proof}

%\begin{cor}\label{cor:spec}
%Under the same conditions as in Theorem \ref{thm:reg} with $a_1=a_3=\frac{1}{2}$ and $a_2\in(0,1/2)$, then  with probability at least $1-\delta$, it holds that
%\begin{align}
%\mathcal{R}_i(T)&=\O\left(T^{\frac{1}{2}+2a_2}+T^{\frac{1}{2}}\ln\frac{1}{\delta}\right),\label{cor-spec-1}\\
%\mathcal{R}_g(T)&=\O\bigg(T^{\frac{3}{4}+\frac{a_2}{2}}
%+T^{\frac{1}{2}+\frac{a_2}{2}}\sqrt{\ln \frac{1}{\delta}}\bigg).\label{cor-spec-2}
%\end{align}
%\end{cor}
%
%\begin{cor}\label{cor:spec}
%Under the same conditions as in Theorem \ref{thm:reg} with $a_1=a_3=\frac{1}{2}+a_2$ and $a_2\in(0,1/2)$, then  with probability at least $1-\delta$, it holds that
%\begin{align}
%\mathcal{R}_i(T)&=\O\left(T^{\frac{1}{2}+3a_2}+T^{\frac{1}{2}+a_2}\ln\frac{1}{\delta}\right),\label{cor-spec-11}\\
%\mathcal{R}_g(T)&=\O\bigg(T^{\frac{3}{4}-a_2}
%+T^{\frac{1}{2}+\frac{3a_2}{2}}\sqrt{\ln \frac{1}{\delta}}\bigg).\label{cor-spec-22}
%\end{align}
%\end{cor}

By specifying parameters $a_1$, $a_2$, and $a_3$ in Theorem \ref{thm:reg}, the regret $\mathcal{R}_i(T)$ and $\mathcal{R}_g(T)$ generated by Algorithm 1 can reach optimal bounds as declared below.
\begin{cor}\label{cor:spec}
Under the same conditions as in Theorem \ref{thm:reg}, let $a_1=\frac{1}{2}+a_2$, $a_3=\frac{1}{2}-a_2$, and $a_2\in(0,\frac{1}{2})$, then for any $i\in[N]$ and $0<\delta<1$, with probability at least $1-\delta$, there holds
\begin{align}
\mathcal{R}_i(T)&=\O\left(T^{\frac{1}{2}+a_2}+T^{\frac{1}{2}+a_2}\ln\frac{1}{\delta}\right),\label{cor-spec-111}\\
\mathcal{R}_g(T)&=\O\bigg(T^{\frac{3}{4}+a_2}
+T^{\frac{1+a_2}{2}}\sqrt{\ln \frac{1}{\delta}}\bigg),\label{cor-spec-222}
\end{align}
\end{cor}

%\begin{proof}
%Letting $a_1=1-a_3=1-a_1+2a_2=2a_2+a_3$, it can be derived that $a_1=\frac{1}{2}+a_2$ and $a_3=\frac{1}{2}-a_2$, which implies that $a_1=1-a_3=1-a_1+2a_2=2a_2+a_3=\frac{1}{2}+a_2$. Moreover, $\max\left\{1-\frac{a_2+a_3}{2},1-\frac{a_1+a_2}{2},\frac{1+a_1+a_2}{2}
%,\frac{1-a_1+a_3}{2}\right\}=\frac{3}{4}$. Hence, the proof is given.
%\end{proof}

\begin{proof}
See Appendix \ref{proof:cor:spec}.
\end{proof}

%\begin{proof}
%The proof follows by choosing $a_1=\frac{1}{2}+a_2$ and $a_3=\frac{1}{2}-a_2$.
%\end{proof}

\begin{rem}
It can be seen from Corollary \ref{cor:spec} that $\mathcal{R}_i(T)$ reaches $\mathcal{O}\left(T^\frac{1}{2}+T^\frac{1}{2}\ln\frac{1}{\delta}\right)$ with high probability when $a_2$ is sufficiently small, and simultaneously, $\mathcal{R}_g(T)$ achieves  $\mathcal{O}\left(T^{\frac{3}{4}}+T^{\frac{1}{2}}\sqrt{\ln \frac{1}{\delta}}\right)$.  Note that these two bounds also depend on $T^{\frac{1}{2}}\ln\frac{1}{\delta}$ and $T^{\frac{1}{2}}\sqrt{\ln \frac{1}{\delta}}$, respectively. Actually, $\ln\frac{1}{\delta}$ increases slowly as the failure probability $\delta$ decreases. For instance, in view of the fact that $\ln 10^2=4.61$,  $\ln 10^3=6.91$, and $\ln 10^4=9.21$, the term $T^{\frac{1}{2}}\ln\frac{1}{\delta}$ sublinearly increases as $4.61T^{\frac{1}{2}}$,  $6.91T^{\frac{1}{2}}$, and $9.21T^{\frac{1}{2}}$ with probabilities at least $99.99\%$,  $99.999\%$, and  $99.9999\%$, respectively. Consequently, the sublinear bounds on $\mathcal{R}_i(T)$ and $\mathcal{R}_g(T)$ can be ensured with a  probability close to one by  Algorithm \ref{alg:primaldual}. Note that our result on $\mathcal{R}_i(T)$ almost covers the existing optimal bound $\mathcal{O}(\sqrt{T})$ in the deterministic case \cite{Abernethy2008OptimalSA}.  Note that $\mathcal{O}(\cdot)$ absorbs constants that relate to the underlying problem structure and network topology. Actually, Lemma \ref{lem-regret} indicates that smoother costs $f_i$, lower problem dimensions $m, n$, better network connectivity (i.e., smaller $U$), and fewer players (i.e., smaller $N$) generally lead to smaller $R_i(T)$ and $R_g(T)$. Our established regrets and constraint violation bounds provide qualitative guidance from a control engineering perspective.
\end{rem}

\begin{rem}
Although an online distributed stochastic mirror descent algorithm was proposed in \cite{lu24Ongam} to handle online stochastic games, and the sublinearity of high probability bounds on regrets were provided, both the aggregative term or coupled constraints are not considered. As such, our problem setting introduces additional technical challenges for both algorithm design and theoretical analysis. Moreover, it is worth mentioning that   the dynamic regret bounds obtained in \cite{lu24Ongam,zuo2023distributed} rely on the strong monotonicity of the pseudo-gradient. In contrast, our analysis does not require this assumption, as we adopt the more relaxed static regret metric. For deterministic online games where no aggregative variable was involved in cost functions, in \cite{lu21online}, the time-invariant coupled constraints were discussed. In contrast, we investigate  time-varying coupled inequality constraints. While time-varying coupled constraints were  considered in \cite{meng2023online}, the proposed algorithms are decentralized rather than distributed, requiring a cental coordinator to bidirectionally communicate information with all players, which incurs reduced robustness and lower privacy protection.
It is important to note that guarantees hold in expectation \cite{wang2024distributed,lei2023distributed} do not capture favorable behaviors  when performing the algorithm in a small number of runs or even a single run.
%Recently, stochastic aggregative games were considered in \cite{wang2024distributed,lei2023distributed} and convergence rates of the devised algorithms were established in expectation. However, they were just for time-invariant settings and did not account for any coupled constraints. It is important to note that guarantees hold in expectation do not capture favorable behaviors  when performing the algorithm in a small number of runs or even a single run. In comparison, high probability guarantees are more important and practical, though significantly harder to obtain. Furthermore, studies in \cite{Huang23,wang2024distributed,lei2023distributed,lu24Ongam,meng2023online} were only confined to balanced communication graphs, while our work considers more general unbalanced graphs.  %Furthermore, unlike \cite{}, the total number of rounds need not to be known in advance to execute our algorithm.
\end{rem}

As a special case of online game $\Gamma_t:=\Gamma(\V,X_t,f_t)$, consider the time-invariant version, i.e., $f_{i,t}$ and $g_{i,t}$ for any $i\in[N]$ are independent of time $t$ and are simply denoted by $f_{i}$ and $g_{i}$, respectively. The corresponding time-invariant game is denoted as $\Gamma:=\Gamma(\V,X,f)$.
%Next, we aim to explore another issue that whether  Algorithm 1 ensures that the actions of players converge or track the GNEs for the time-invariant game $\Gamma:=\Gamma(\V,X,f)$, where $X:=X_{0}\bigcap(X_1\times\cdots\times X_N)$  with $X_{0}:=\{x=\col\{x_i\}_{i\in[N]}\in\R^d\mid x_i\in\R^{d_i},i\in[N],g(x):=\sum_{i=1}^Ng_{i}(x_i)\leq{\bf 0}_m\}$, and  $f=(f_{1},\ldots,f_{N})$. Here, $f_{i}(x_i,\sigma(x)):={\color{blue}\E}[F_{i}(x_i,\sigma(x),\xi_{i}(\omega))]$, where $\sigma(x):=\frac{1}{N}\sum_{i=1}^N\psi_i(x_i)$. Here, $J_i(x_i,x_{-i})$ being convex and differentiable with respect to $x_i$.
Next, we are committed to another key issue that whether the decision sequence of players generated by Algorithm \ref{alg:primaldual} converges to the GNEs of game $\Gamma$. In general, it is challenging to identify all GNEs for a game. Therefore, we will concentrate on tracking the vGNE of $\Gamma$.
%, as done in \cite{wang2024distributed,FRANCI2022}. %In fact, the unique variational GNE enjoys good stability and exhibits no price discrimination from the economics perspective \cite{nedic09appro}.
To this end, the following assumption is typically required.

%\begin{ass}\label{ass:mo}
%$G$ is strictly monotone, i.e., for any $x,x'\in X_1\times\cdots\times X_N$ and $x\neq x'$,
%\begin{align}\label{mo}
%\langle G(x)-G(x'),x-x'\rangle > 0,
%\end{align}
%where $G(x):=\col\{\nabla_iJ_i(x_i,x_{-i})\}_{i\in[N]}$ is the pseudo-gradient mapping of game $\Gamma$.
%\end{ass}

\begin{ass}\label{ass:strict-mo}
$G$ is strictly monotone, i.e., for any $x,x'\in X_1\times\cdots\times X_N$, $x\neq x'$,
\begin{align}\label{stri-mo}
\langle G(x)-G(x'),x-x'\rangle> 0,
\end{align}
where $G(x):=\col\{\nabla_{x_i}f_{i,t}(x_i,\sigma(x))\}_{i\in[N]}$ is the pseudo-gradient mapping of game $\Gamma$.
\end{ass}

%Under Assumption \ref{ass:set}-i) and iii), it follows from Proposition 1.4.2 in \cite{facchinei2003finite} that a solution $x\in X$ to the variational inequality
%\begin{align}\label{optcon}
%G(x^*)^{\T}(x-x^*)\geq 0,\ \forall x\in X
%\end{align}
%is equivalent to an GNE of $\Gamma$.
Assumption \ref{ass:strict-mo} guarantees the uniqueness of the vGNE \cite{facchinei2003finite}.
By Assumption \ref{ass:set}-ii), the optimal Lagrange multiplier $\mu^*\in\R^m_+$ of the Lagrange function $L_{i}(x_i,\mu;x_{-i}):=f_{i}(x_i,\sigma(x))+\mu^{\T}g(x)$ associated with the game $\Gamma$ is bounded \cite{nedic09appro}. Specifically, there exists  $\chi>0$ such that %$\|\mu^*\|\leq \chi$.
\begin{align}\label{bound:mu}
\|\mu^*\|\leq \chi.
\end{align}

At present, let us show the almost sure convergence of Algorithm \ref{alg:primaldual} for time-invariant game $\Gamma$.
\begin{thm}\label{thm:conver}
Let Assumptions \ref{ass:graph}-\ref{ass:fun}, \ref{ass:sg}-i), \ref{ass:sg}-iii), and {\color{blue}\ref{ass:strict-mo}} hold. If $\alpha_t$, $\beta_t$, and $\gamma_t$ are selected  satisfying $\alpha_t=\gamma_t$ and
\begin{align}\label{eq-step-con}
&\sum_{t=1}^\infty\alpha_t=\infty,\ \sum_{t=1}^\infty\alpha_{t}^2<\infty,\notag\\ &\sum_{t=1}^\infty\frac{\alpha_{t}^2}{\beta_{t}^2}<\infty,\ \sum_{t=1}^\infty\alpha_{t}\beta_{t}<\infty,\
\sum_{t=1}^\infty\frac{\alpha_{t}^2}{\beta_{t}}<\infty,
\end{align}
then the sequence $\{x_t\}$ generated by Algorithm \ref{alg:primaldual} converges to the vGNE almost surely.
\end{thm}

\begin{proof}
See Appendix \ref{proof:thm:conver}.
\end{proof}

Theorem \ref{thm:conver} reveals the almost sure convergence of the play sequence, while it is difficult to derive the convergence rate. Instead, we derive the high probability rate of the average decision sequence in what follows. For this, the game structure is further restricted to a strongly monotone game.

\begin{ass}\label{ass:strong-mo}
$G$ is $\mu$-strongly monotone, i.e., for any $x,x'\in X_1\times\cdots\times X_N$,
\begin{align}\label{st-mo}
\langle G(x)-G(x'),x-x'\rangle\geq\mu\|x-x'\|^2.
\end{align}
%where $G(x):=\col\{\nabla_{x_i}f_{i,t}(x_i,\sigma(x))\}_{i\in[N]}$ is the pseudo-gradient mapping of game $\Gamma$.
\end{ass}

\begin{thm}\label{thm:track}
Let Assumptions \ref{ass:graph}-\ref{ass:sg} and \ref{ass:strong-mo} hold. If $\alpha_t$, $\beta_t$, and $\gamma_t$ are selected  satisfying $\alpha_0=\beta_0=\gamma_0=1$, and for $t\geq 1$,
\begin{align}
\alpha_t=\frac{1}{t^{a_1}},\ \beta_t=\frac{1}{t^{a_2}},\ \gamma_t=\frac{1}{t^{a_3}},
\end{align}
where $0<a_1, a_3<1$, $a_1>2a_2>0$, and $2a_2+a_3<1$, then for any  $0<\delta<1$, with probability at least $1-\delta$, the sequence $\{x_t\}$ generated by Algorithm \ref{alg:primaldual} satisfies
\begin{align}\label{thm-track}
\|\bar{x}_T-x^*\|^2=&\mathcal{O}\bigg(T^{-\min\{1-a_1,1-2a_2-a_3,a_3,a_2,a_1-2a_2\}}\notag\\
&\quad+T^{-(1-a_1)}\ln\frac{1}{\delta}\bigg),
\end{align}
where $\bar{x}_T:=\frac{1}{T}\sum_{t=1}^Tx_t$ and $x^*=\col\{x_i^*\}_{i\in[N]}$ is the vGNE of game $\Gamma$.
\end{thm}

\begin{proof}
See Appendix \ref{proof:thm:track}.
\end{proof}

By specifying parameters $a_1$, $a_2$, and $a_3$ in Theorem \ref{thm:track}, an optimal upper bound  on $\|\bar{x}_T-x^*\|^2$ under Algorithm \ref{alg:primaldual} can be derived.
\begin{cor}\label{cor:track}
Under the same conditions as in Theorem \ref{thm:track}, let $a_1=3/4$ and $a_2=a_3=1/4$, then  with probability at least $1-\delta$, the best convergence rate of $\|\bar{x}_T-x^*\|^2$ is given as follows:
\begin{align}\label{cor-track}
\|\bar{x}_T-x^*\|^2=\mathcal{O}\bigg(T^{-\frac{1}{4}}+T^{-\frac{1}{4}}\ln\frac{1}{\delta}\bigg).
\end{align}
\end{cor}
\begin{proof}
See Appendix \ref{proof:cor:track}.
\end{proof}

\begin{rem}
By Corollary \ref{cor:track}, it can be founded that the average decision sequence converges to the unique vGNE of $\Gamma$ in high probability, and the convergence rate  follows  $\mathcal{O}\left(T^{-\frac{1}{4}}+T^{-\frac{1}{4}}\ln\frac{1}{\delta}\right)$.
Recently, \cite{wang2024distributed} proved an $\mathcal{O}(1/t)$ rate for stochastic aggregative games with set constraints, while it considers weight-balanced graphs and did not account for any coupled constraints.
Moreover, the authors in \cite{tatarenko2025convergence} proposed one-point and two-point payoff-based GNE learning algorithms in deterministic setup and, by drawing an analogy to stochastic gradient methods, established convergence rates of arbitrary close to $\mathcal{O}(1/t^{1/4})$ and $\mathcal{O}(1/t^{1/2})$ in expectation, respectively.
Although the convergence rate of the algorithm under two-point feedback is better than ours, it requires a central coordinator to compute and broadcast dual updates to all the players, and the coupled constraints are restricted to the linear forms.
%Furthermore, studies in \cite{Huang23,wang2024distributed,lei2023distributed,lu24Ongam,meng2023online} were only confined to balanced communication graphs, while our work considers more general unbalanced graphs.}
In \cite{FRANCI2022,zheng2023distributed,Huang23}, stochastic games with just affine coupled constraints were investigated and the almost sure convergence was proved. However, these works did not yield any convergence rate. To the best of our knowledge, the obtained result in terms of convergence rate here is the first for stochastic aggregative games with general coupled inequality constraints, especially, in high probability. %{\color{blue}Moreover, the devised algorithm can work effectively over time-varying unbalanced graphs compared to static undirected graphs studied in \cite{FRANCI2022,zheng2023distributed,Huang23}.}
%It is important to note that guarantees hold in expectation do not capture favorable behaviors  when performing the algorithm in a small number of runs or even a single run. In comparison, high probability guarantees are more important and practical, though significantly harder to obtain.
\end{rem}

\section{Numerical Simulation}\label{sec4}
Consider an electricity market with $N$ generator systems. At time $t$, let $x_{i,t} \in X_i=[0,20]$ be the output power of generator $i$ and $x_t=\col\{x_{i,t}\}_{i\in[N]}$. In order to capture variability in renewable energy, fluctuations in fossil fuel price, and randomness in demand or marginal return, the cost function of generator $i$ is formulated as
\begin{align*}
F_{i,t}(x_{i,t},\sigma(x_{t}),\xi_{i})=l_{i,t}(x_{i,t})- d_{i,t}(\sigma(x_{t}),\xi_{i})x_{i,t},
\end{align*}
%$l_{i,t}(x_i)=a_i+b_i x_i+c_ix_i^2$
in which $l_{i,t}(x_{i,t})=a_{i,t}+b_{i,t} x_{i,t}+c_{i,t}x_{i,t}^2$ is the generation cost of generator $i$ with constants $a_{i,t}$, $b_{i,t}$, and $c_{i,t}$ being its characteristics.
%while $\sin(t/12)$ reflects seasonal or periodic variations, such as hydropower generation cycles.
The electricity pricing function is given by $d_{i,t}(\sigma(x),\xi_{i})=p_0-\xi_{i}-\gamma \sigma(x)$  with $p_0$ being the baseline electricity price, $\sigma(x)=\sum_{i=1}^N x_i$, and $\gamma>0$ being the price sensitivity to the total generation, and $\xi_{i}\in[-0.5,0.5]$ being a random variable satisfying uniform distribution, which represents uncertainty in marginal return or user demand. Here,  $d_{i,t}(\sigma(x_{t}),\xi_{i})x_{i,t}$ indicates the income of generation $i$.
Moreover, the transmission grid security is characterized by shared coupled inequality constraints $\sum_{i=1}^Ng_{i,t}(x_{i,t})\leq G$,
where $g_{i,t}(x_{i,t})=r_{i,t}+u_{i,t} x_{i,t}+v_{i,t}x_{i,t}^2$ indicates the impact of generator $i$ on the load of a specific line and
%$\sin(t/12)$ models time-varying network conditions and
$G$ denotes the safe transmission capacity.
%This constraint enforces that the total power flow across the entire network must not exceed the safe capacity.
This constraint enforces that the total load across the transmission grid does not exceed
%the maximal safe capacity.
the prescribed security limit.

In the simulation, for any $i\in[N]$ and $t\geq 0$, $a_{i,t}$, $b_{i,t}$, $c_{i,t}$, $r_{i,t}$, $u_{i,t}$, and $v_{i,t}$ are sampled from uniform distributions over $[1, 9]$, $[5, 15]$, $[6, 10]$, $[1, 2]$, $[2, 5]$, and $[1, 3]$, respectively. Let $p_0=40$, $\gamma=0.8$, and $G=150$. The communication networks are characterized as a time-varying graph switching among the four graphs successively in Fig. \ref{fig:network}, whose weighted adjacency matrices are generated according to \eqref{def:weight}. Assumption \ref{ass:graph} is verified to be satisfied.
In the online case, choose $a_1=0.8$, $a_2=0.2$, and $a_3=0.2$. Executing Algorithm \ref{alg:primaldual}, Fig. \ref{fig:ri} and Fig.\ref{fig:rg} depict the evolutions of $\mathcal{R}_i(T)/T$, $i\in[N]$ and $\mathcal{R}_g(T)/T$, respectively. The results show that the trajectories approximately tend to zero as iteration goes on, verifying the effectiveness of Algorithm \ref{alg:primaldual}.
%indicating the regrets $\mathcal{R}_i(T)$, $i\in[N]$ increase sublinearly,
%which validates the effectiveness of Algorithm \ref{alg:primaldual}. From the simulation results, one can see that the trajectories approximately tend to zero as iteration goes on, indicating the effectiveness of Algorithm 1.
The actual trajectory of $g_t(x_t)$ is shown in Fig. \ref{fig:gt}, demonstrating that players can dynamically adapt their decisions.
Furthermore, we also plot the trajectories of $\sum_{i=1}^N\|\tilde{\mu}_{i,t+1}-\bar{\mu}_t\|$ and $\sum_{i=1}^N\|\tilde{\sigma}_{i,t+1}-\bar{\sigma}_t\|$.
As presented in Fig. \ref{fig:consensus_error}, all $\tilde{\mu}_{i,t}$'s and $\tilde{\sigma}_{i,t}$'s achieve consensus asymptotically. For comparison, we also simulate the algorithm in \cite{lu21online}, where the true value of the aggregative term $\sigma(x_k)$ is used since no aggregative variable is involved in their setting. As observed in Fig. \ref{fig:compare_ri} and Fig. \ref{fig:compare_rg}, the average regrets and constraint violations fail to converge to zero. Therefore, the algorithm in \cite{lu21online} cannot be applied to our problem.

Consider the offline scenario. Set $a_1=0.75$ and $a_2=a_3=0.25$. The evolution of the residual $\|x_t-x^*\|$ is presented in Fig. \ref{fig:off_error}, showing that Algorithm \ref{alg:primaldual} succeeds in finding the unique vGNE $x^*$. Fig. \ref{fig:off_error_ave} demonstrates that the residual $\|\bar{x}_T-x^*\|^2$ decays to zero sublinearly, which is consistent with our theoretical result. Moreover, we modify the operator $G_{x,i}$ in \cite{Carnevale24Tracking-based} accordingly to make it accommodate nonlinear constraints. Fig. \ref{fig:off_error} and  Fig. \ref{fig:off_error_ave} show that our proposed algorithm is superior than that in \cite{Carnevale24Tracking-based}.
\begin{figure}[t!]
    \centering
    % 第一行
    \subfigure[]{
        \begin{minipage}[b]{0.42\linewidth}
        \includegraphics[width=3.3cm,height=2.5cm]{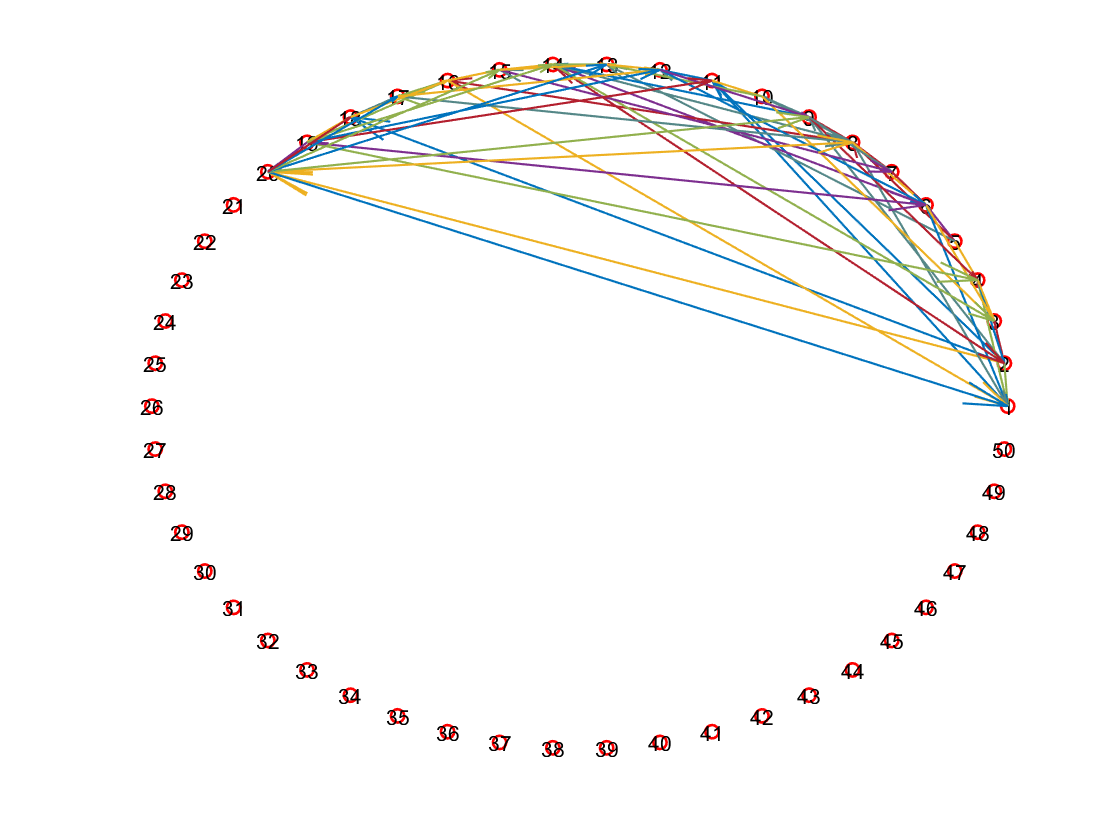}
        \end{minipage}
    }
    \hspace{0.2em} % 调整水平间距
    \subfigure[]{
        \begin{minipage}[b]{0.45\linewidth}
        \includegraphics[width=3.36cm,height=2.5cm]{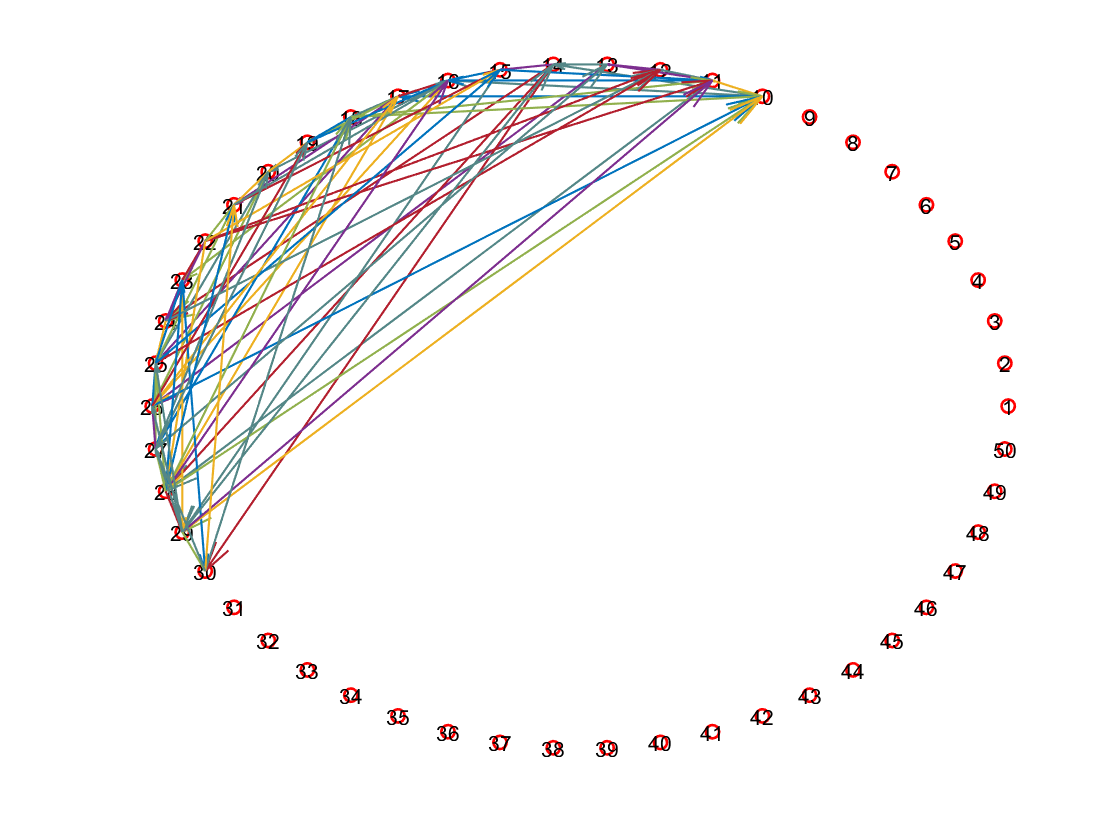}
        \end{minipage}
    }

    \vspace{-0.8em} % <<< 减小上下间距

    % 第二行
    \subfigure[]{
        \begin{minipage}[b]{0.42\linewidth}
        \includegraphics[width=3.3cm,height=2.5cm]{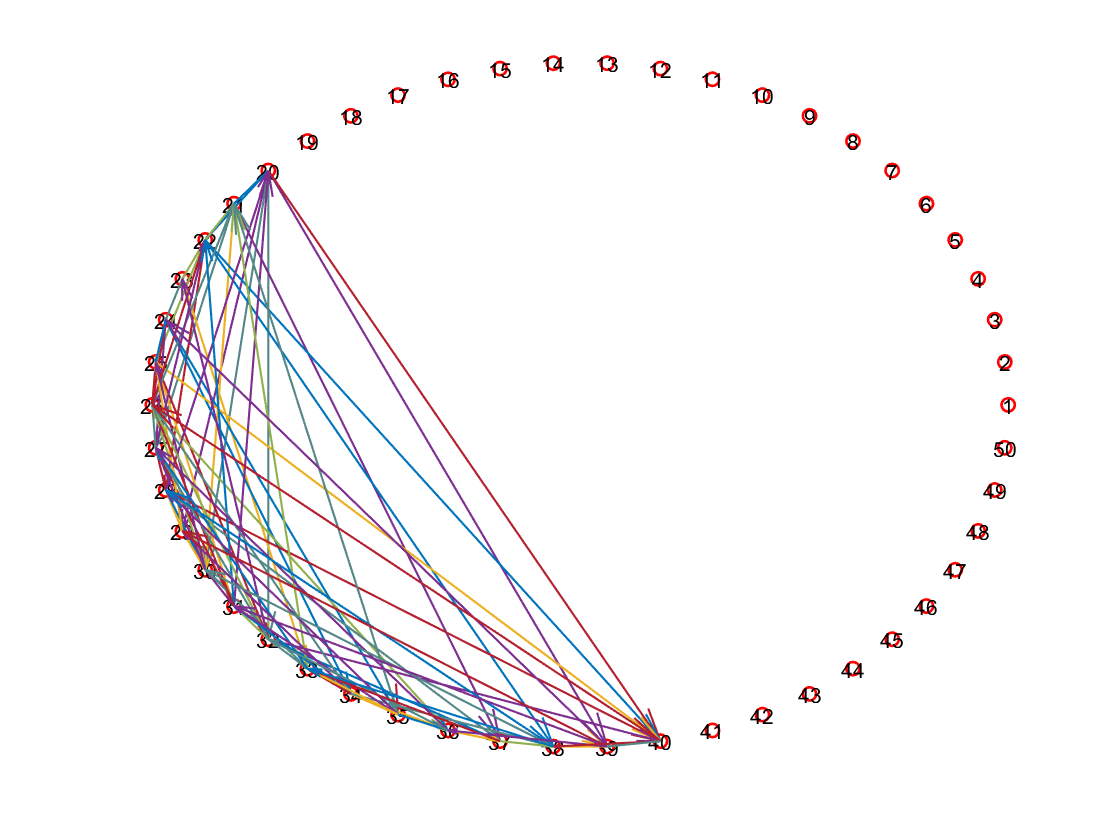}
        \end{minipage}
    }
    \hspace{0.2em}
    \subfigure[]{
        \begin{minipage}[b]{0.45\linewidth}
        \includegraphics[width=3.3cm,height=2.5cm]{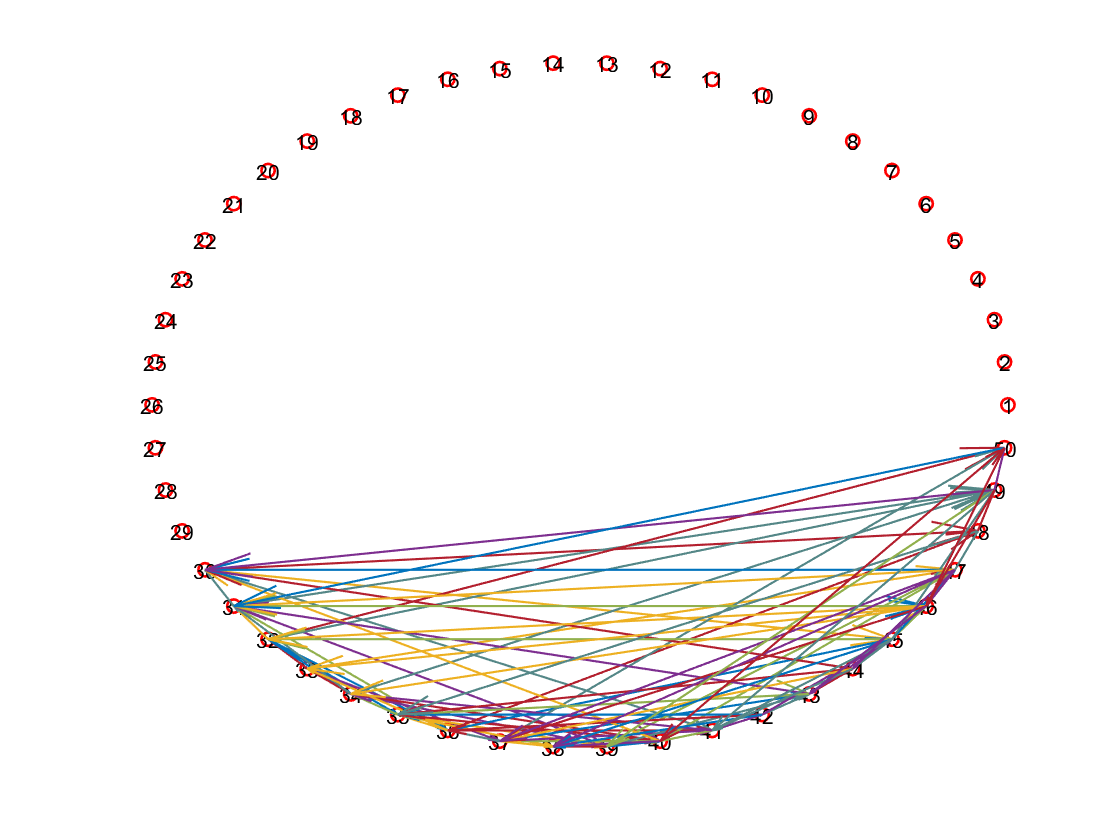}
        \end{minipage}
    }

    \vspace{0em} % <<< 图与caption更紧凑
    \caption{Schematic illustration of four switching graphs.}
    \label{fig:network}
\end{figure}

\begin{figure}[t!]
\centering
\includegraphics[width=6.5cm,height=5cm]{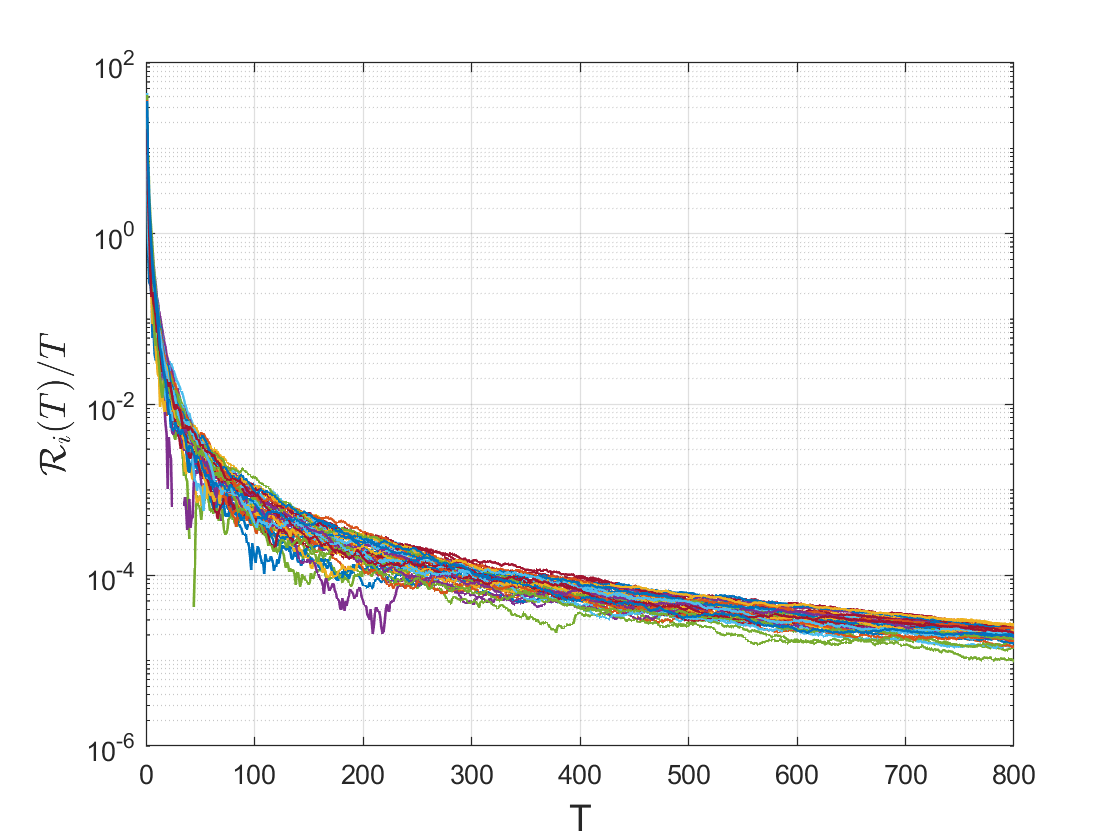}
\caption{Trajectories of $\mathcal{R}_i(T)/T$, $i\in[N]$ by Algorithm \ref{alg:primaldual}.\label{fig:ri}}
\end{figure}

\begin{figure}[t!]
\centering
\includegraphics[width=6.5cm,height=5cm]{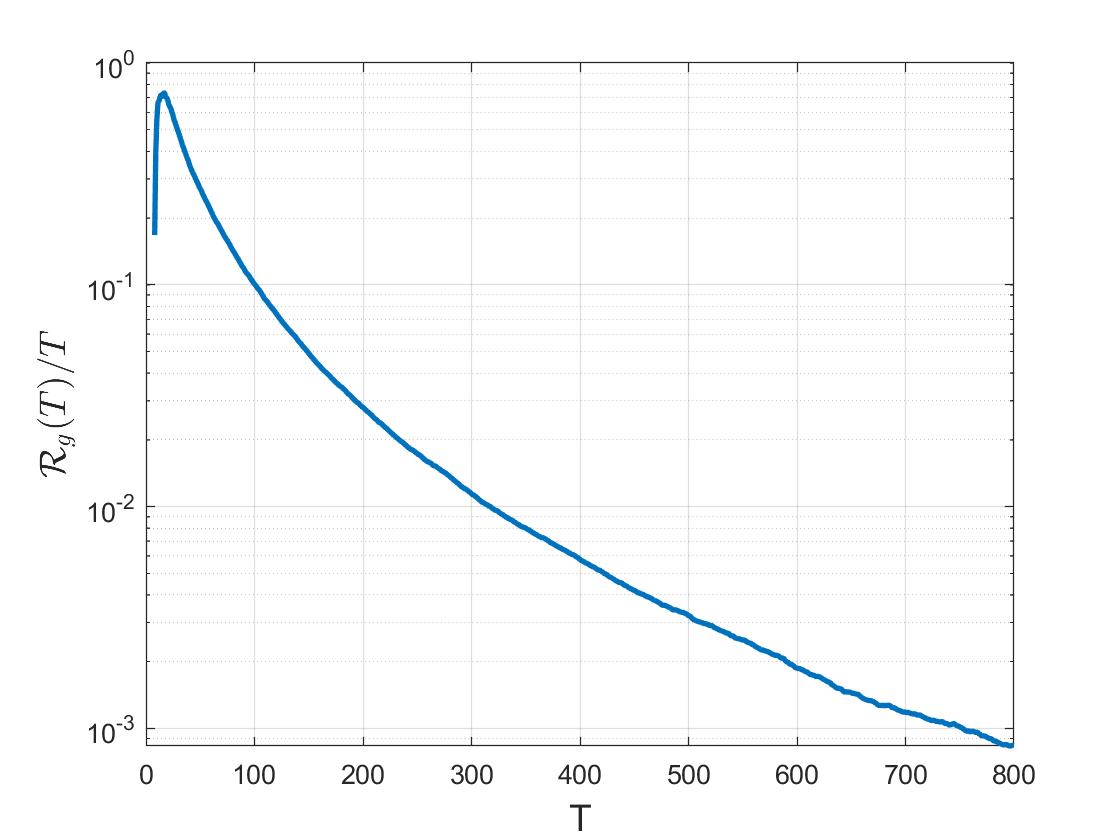}
\caption{The trajectory of $\mathcal{R}_g(T)/T$ by Algorithm \ref{alg:primaldual}.\label{fig:rg}}
\end{figure}

\begin{figure}[t!]
\centering
\includegraphics[width=6.5cm,height=5cm]{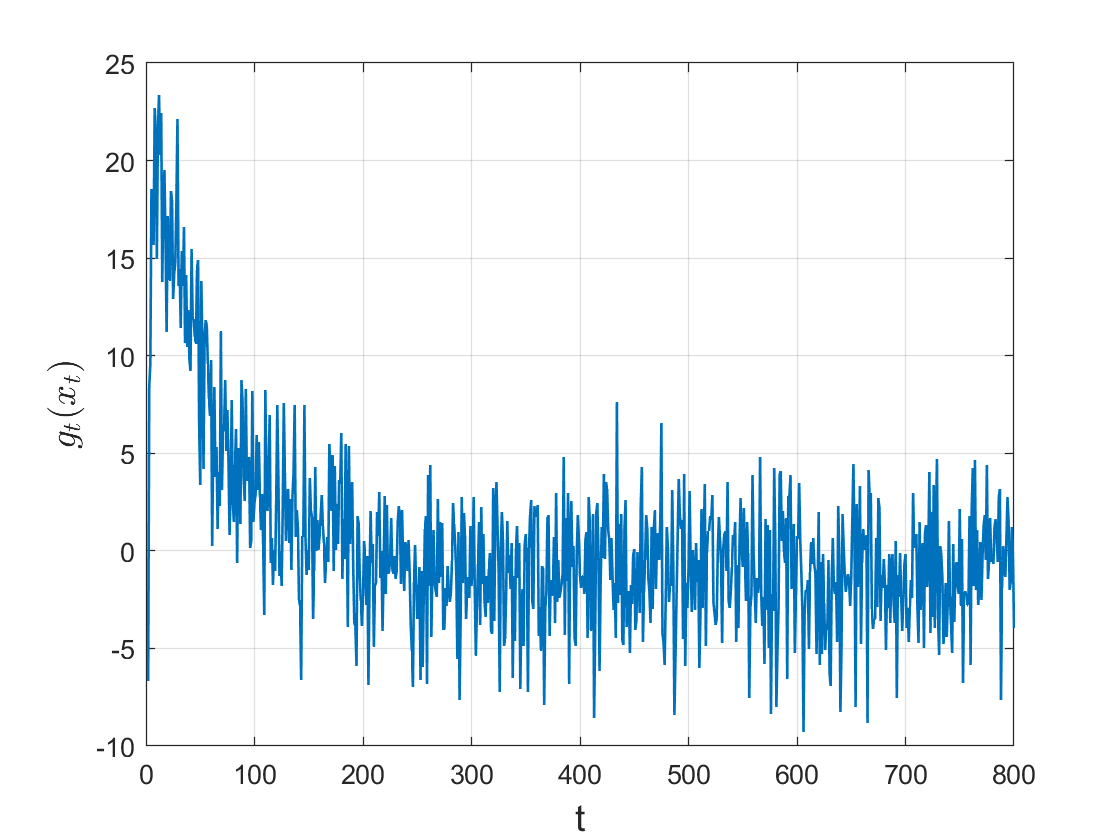}
\caption{The trajectory of $g_t(x_t)$ by Algorithm \ref{alg:primaldual}.\label{fig:gt}}
\end{figure}

\begin{figure}[t!]
\centering
\includegraphics[width=6.5cm,height=5cm]{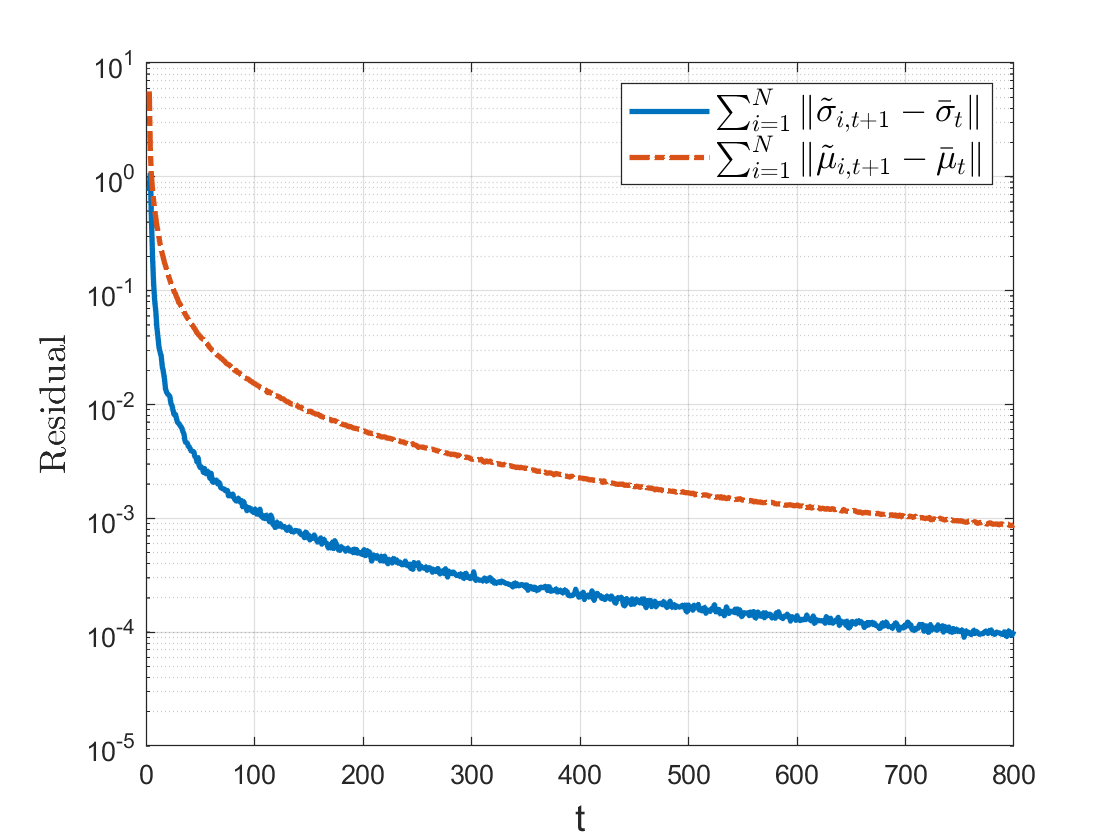}
\caption{Trajectories of consensus errors by Algorithm \ref{alg:primaldual}.\label{fig:consensus_error}}
\end{figure}

\begin{figure}[t!]
\centering
\includegraphics[width=6.5cm,height=5cm]{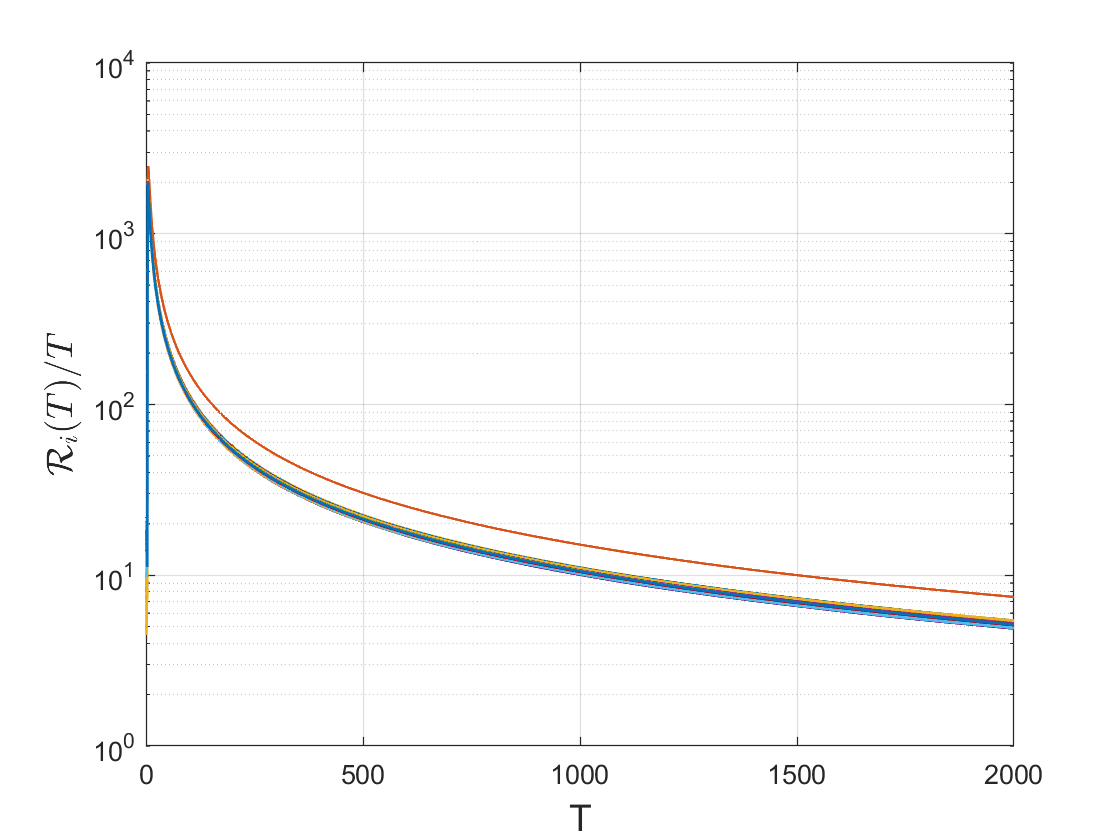}
\caption{Trajectories of $\mathcal{R}_i(T)/T$, $i\in[N]$ by algorithm (11) in \cite{lu21online}.\label{fig:compare_ri}}
\end{figure}

\begin{figure}[t!]
\centering
\includegraphics[width=6.5cm,height=5cm]{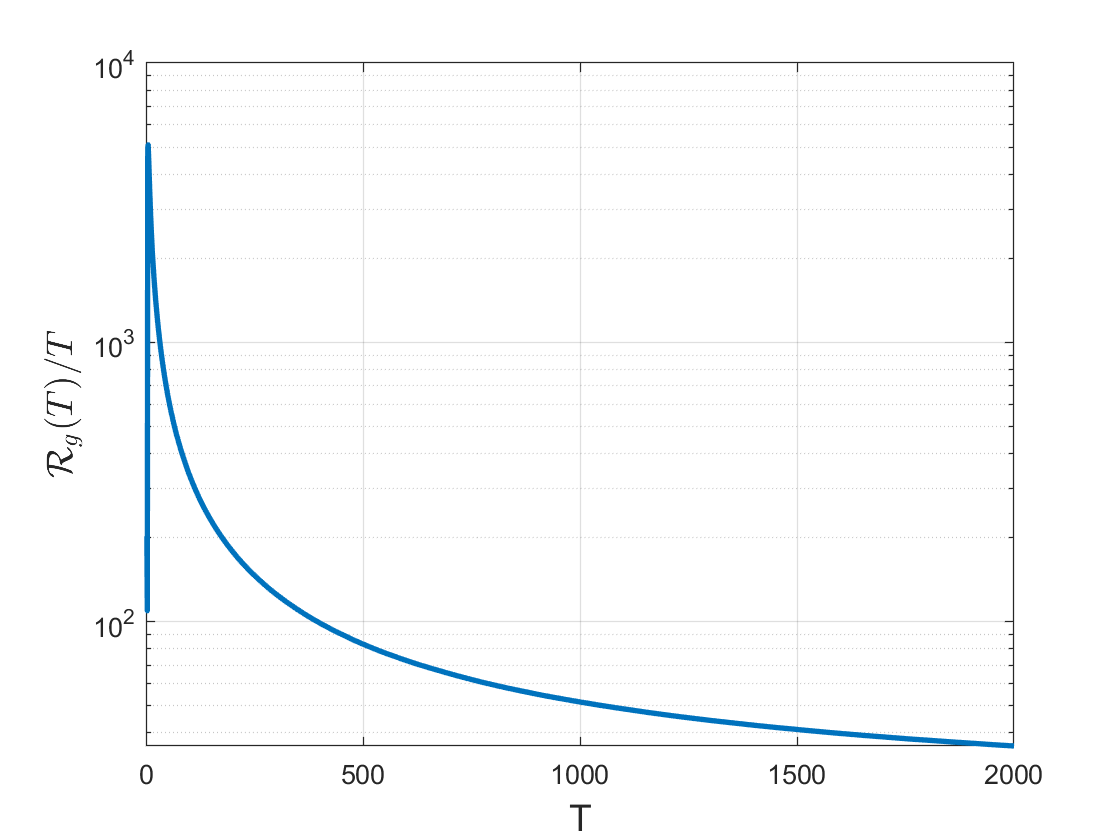}
\caption{The trajectory of $\mathcal{R}_g(T)/T$ by algorithm (11) in \cite{lu21online}.\label{fig:compare_rg}}
\end{figure}

\begin{figure}[t]
\centering
\includegraphics[width=6.5cm,height=5cm]{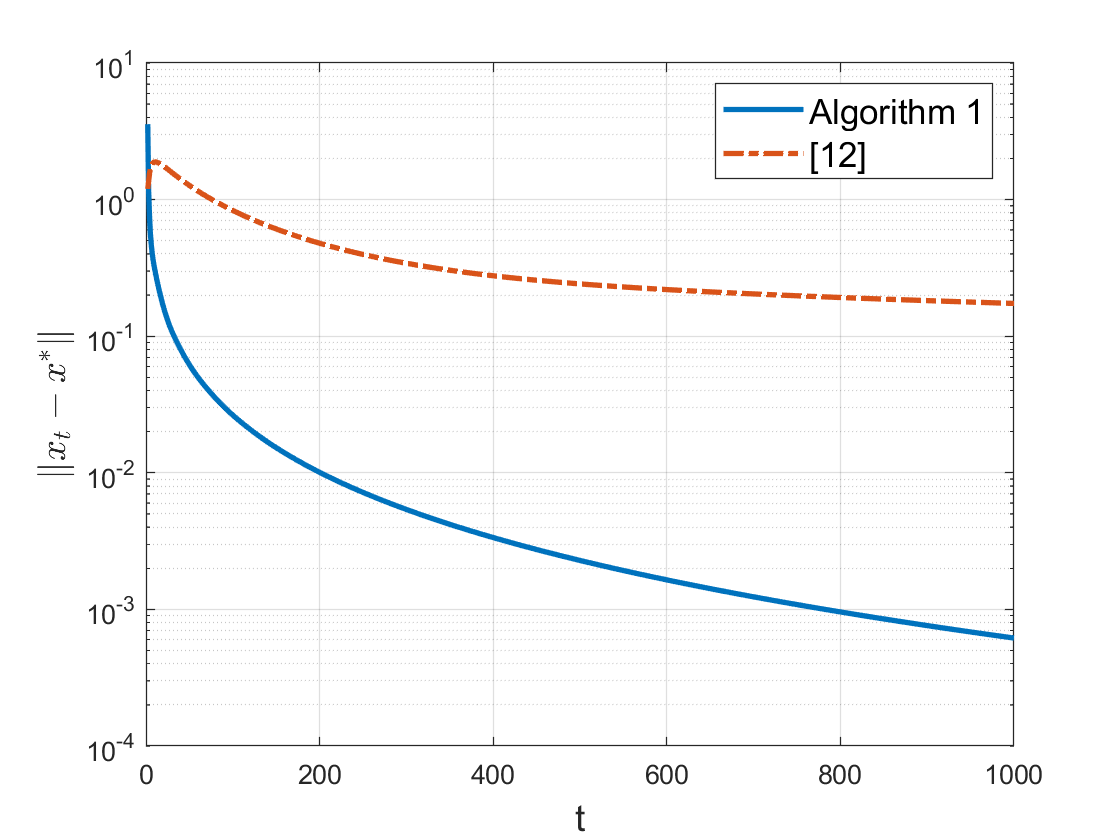}
\caption{Comparison of the trajectory $\|x_t-x^*\|$ between Algorithm \ref{alg:primaldual} and Algorithm 2 in \cite{Carnevale24Tracking-based} in offline case.\label{fig:off_error}}
\end{figure}

\begin{figure}[t!]
\centering
\includegraphics[width=6.5cm,height=5cm]{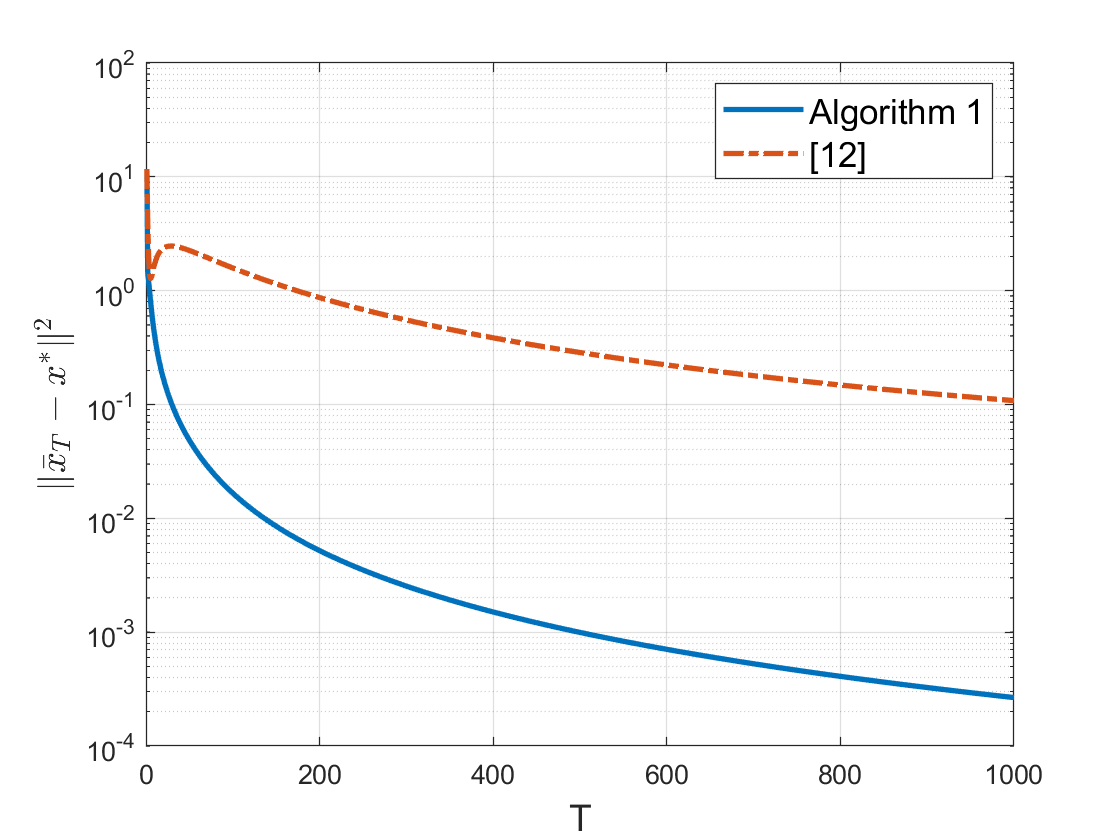}
\caption{Comparison of the trajectory $\|\bar{x}_T-x^*\|^2$ between Algorithm \ref{alg:primaldual} and Algorithm 2 in \cite{Carnevale24Tracking-based} in offline case.\label{fig:off_error_ave}}
\end{figure}

%8 5.8

\section{CONCLUSION}\label{sec5}
In this paper, distributed learning for stochastic online aggregative games satisfying local set constraints  and time-varying general convex constraints was studied for the first time. An online distributed stochastic primal-dual push-sum algorithm was devised, which is applicable to the unbalanced networks. It was rigorously shown that the regrets and constraint violation increase sublinearly with high probability by utilizing the sub-Gaussian noise model. In addition, in the offline case, the developed algorithm was proved to almost surely converge to the vGNE for strictly monotone games, and the high probability convergence rate of  the time-averaged decision sequence was also analyzed for strongly monotone games. The algorithm's performance was validated by a numerical application. Future work may focus on nonconvex cost functions.

\begin{appendix}

\subsection{Proof of Lemma \ref{lem:mu-w-z}}\label{proof:lem:mu-w-z}
\begin{proof}
The proof of (\ref{lem-z}) can be found in Lemma 3 of \cite{li21dop}. Next, we derive (\ref{lem-mu}) by mathematical induction. Due to the fact that $\mu_{i,0}=0$,  $\forall i\in[N]$, it is obvious that  $\hat{\mu}_{i,0}\leq\frac{z_{i,1}L}{\beta_0r^2}$,  $\forall i\in[N]$. Assume that (\ref{lem-mu}) holds for any time $t$ and $i\in[N]$. We show that it still holds at time $t+1$. To begin with, one can obtain
\begin{align*}%\label{eq-mu-1}
\|\mu_{i,t+1}\|&\leq (1-\gamma_t\beta_t)\|\hat{\mu}_{i,t}\|+\gamma_t\|g_{i,t}(x_{i,t})\|\notag\\
&\leq (1-\gamma_t\beta_t) \frac{z_{i,t+1}L}{\beta_tr^2}+\frac{\gamma_tL}{r}\notag\\
&=\left(1-\gamma_t\beta_t+\frac{\gamma_t\beta_tr}{z_{i,t+1}}\right)\frac{z_{i,t+1}L}{\beta_tr^2}\notag\\
&\leq \frac{z_{i,t+1}L}{\beta_tr^2},
\end{align*}
where  the first inequality adopts $r\leq 1$  and the last inequality is due to $z_{i,t+1}\geq r$. Hence, it further yields that
\begin{align*}%\label{eq-mu-2}
\|\hat{\mu}_{i,t+1}\|&=\sum_{j=1}^N w_{i j, t+1} \mu_{j, t+1}\notag\\
&\leq \frac{L}{\beta_tr^2}\sum_{j=1}^Nw_{i j, t+1}z_{j,t+1}\notag\\
&\leq \frac{w_{i,t+2}L}{\beta_{t+1}r^2},
\end{align*}
where $\beta_{t+1}\leq \beta_t$ and (\ref{alg:dist_1}) have been employed to obtain the last inequality.

Note that $\hat{\mu}_{i,t}=\sum_{j=1}^N w_{i j, t}\mu_{j,t}\geq w_{ii,t}\mu_{i,t}$, by virtue of $w_{ii,t}\geq w$ from Assumption \ref{ass:graph}-i), one can obtain $\|\mu_{i,t}\|\leq\frac{\|\hat{\mu}_{i,t}\|}{w_{ii,t}}\leq \frac{z_{i,t+1}L}{\beta_tr^2\underline{w}}$. The proof is completed.
\end{proof}

\subsection{Proof of Lemma \ref{lem-consensus}}\label{proof:lem-consensus}
\begin{proof}
It follows directly from Lemma 1 in \cite{nedic15} that
\begin{align*}
\|\tilde{\mu}_{i,t+1}-\bar{\mu}_t\|\leq \frac{8}{r}\left(\vartheta^t\|\mu_0\|_1+\sum_{k=1}^t\vartheta^{t-k}\|\varepsilon_{\mu_{k}}\|_1\right),
\end{align*}
where %{\color{blue} $\vartheta \leq\left(1-\frac{1}{N^{N U}}\right)^{\frac{1}{N U}}\in(0,1)$},
$\varepsilon_{\mu_{k}}:=\col\{\varepsilon_{\mu_{i,k}}\}_{i\in[N]}$ and $\mu_0:=\col\{\mu_{i,0}\}_{i\in[N]}$. For the term $\|\varepsilon_{\mu_{k}}\|_1$, one has
\begin{align*}
\|\varepsilon_{\mu_{i,k}}\|_1&\leq \sqrt{m}\|\varepsilon_{\mu_{i,k}}\|\notag\\
&\leq \gamma_{k-1}\sqrt{m}\|g_{i,k-1}(x_{i,k-1})-\beta_{k-1}\hat{\mu}_{i,k-1}\|\notag\\
&\leq 2\gamma_{k-1}\sqrt{m}L,
\end{align*}
where the first inequality leverages the relation $\|a\|_1\leq \sqrt{p}\|a\|$ for $a\in\R^p$, and the second inequality is derived from (\ref{eq:bound}) and Lemma \ref{lem:mu-w-z}. Hence, in view of $\|\varepsilon_{\mu_{k}}\|_1=\sum_{i=1}^N\|\varepsilon_{\mu_{i,k}}\|_1$ and $\mu_{i,0}=0$, $i\in[N]$,  (\ref{lem:mu-sigma-1}) holds.

Again using Lemma 1 in \cite{nedic15}, it gives
\begin{align*}%\label{eq-consen-1}
\|\tilde{\sigma}_{i,t+1}-\bar{\sigma}_t\|\leq \frac{8}{r}\left(\vartheta^t\|\sigma_0\|_1+\sum_{k=1}^t\vartheta^{t-k}\|\varepsilon_{\sigma_{k}}\|_1\right),
\end{align*}
where $\varepsilon_{\sigma_{k}}:=\col\{\varepsilon_{\sigma_{i,k}}\}_{i\in[N]}$ and $\sigma_0:=\col\{\sigma_{i,0}\}_{i\in[N]}$. Note that
\begin{align*}%\label{eq-consen-2}
\|\varepsilon_{\sigma_{i,k}}\|_1&\leq \sqrt{n}\|\varepsilon_{\sigma_{i,k}}\|\notag\\
&\leq L_\sigma\sqrt{n}\|x_{i,k}-x_{i,k-1}\|\notag\\
&\leq \alpha_{k-1}L_\sigma\sqrt{n}\|s_{i,k}\|\notag\\
&\leq \alpha_{k-1}L_\sigma\sqrt{n}\left(S+\frac{LM}{\beta_{k-1}r^2}\right),
\end{align*}
where the second inequality applies Assumption \ref{ass:set}-v), the third inequality holds due to (\ref{alg:dist_4}), and the last inequality is obtained based on (\ref{alg:dist_3}), Assumptions \ref{ass:sg}-iii), \ref{ass:set}-iv),  and (\ref{lem-mu}) in Lemma \ref{lem:mu-w-z}. Hence, (\ref{lem:mu-sigma-2}) can be derived, which completes the proof.
\end{proof}

\subsection{Proof of Lemma \ref{lem:hp}}\label{proof:lem:hp}
\begin{proof}
Define $\Delta_t:=p_{i,t}(x_{i,t},\tilde{\sigma}_{i,t+1})-q_{i,t}(x_{i,t},\tilde{\sigma}_{i,t+1},\xi_{i,t})$.
In view of the exponential inequality $\exp(a)\leq \exp(a^2)+a$ for any $a\in\R$ and invoking $\|x_{i,t}-x_i\|^2\leq4L^2$ by (\ref{eq:bound}) one has
\begin{align*}%\label{lemhp-1}
&\exp(\alpha_t\langle \Delta_t,x_{i,t}-x_i\rangle)\notag\\
&\leq \exp(4L^2\alpha_t^2\| \Delta_t\|^2)+\alpha_t\langle \Delta_t,x_{i,t}-x_i\rangle.
\end{align*}
Taking conditional expectation with respect to $\F_t$ from both sides of the above inequality, it derives that
\begin{align}\label{lemhp-2}
&\E\left[\exp\left(\frac{\alpha_t}{2L\iota}\langle \Delta_t,x_{i,t}-x_i\rangle\right)\bigg|\F_t\right]\notag\\
&\leq \E\left[\exp(\alpha_t^2\| \Delta_t\|^2/\iota^2)|\F_t\right]\notag\\
&\leq \left(\E[\exp(\| \Delta_t\|^2/\iota^2)|\F_t]\right)^{\alpha_t^2}\notag\\
&\leq\exp(\alpha_t^2),
\end{align}
where the first inequality holds since $\E[\langle \Delta_t,x_{i,t}-x_i\rangle|\F_t]=\E[\Delta_t^{\T}|\F_t](x_{i,t}-x_i)=0$ by Assumption \ref{ass:sg}-i), the second inequality holds by leveraging Jensen's inequality and the fact that $\alpha_t\in[0,1]$, and the last inequality is based on Assumption \ref{ass:sg}-ii).

Next, define
$
v_{i,t}:=\frac{\alpha_t}{2L\iota}\langle \Delta_t,x_{i,t}-x_i\rangle)
$
and consider the dynamics $u_{i,t+1}=\exp(v_{i,t}-\alpha_t^2)u_{i,t}$ with $u_{i,1}=1$. Then it can be observed that $u_{i,t+1}=\exp(\sum_{l=1}^t(v_{i,l}-\alpha_l^2))$. Furthermore, it follows from (\ref{lemhp-2}) that $\E[u_{i,t+1}|\F_t]=\E[\exp(v_{i,t}-\alpha_t^2)u_{i,t}|\F_t]\leq u_{i,t}$.
Taking full expectation leads to that
\begin{align}\label{lemhp-3}
\E[u_{i,t+1}]\leq\E[u_{i,t}]\leq\cdots\leq u_{i,1}=1.
\end{align}
Note that for any $\kappa>0$, it has
\begin{align*}
&\P\left[\sum_{t=1}^T(v_{i,t}-\alpha_t^2)\geq \kappa\right]\notag\\
&=\P\left[\exp\left(\sum_{t=1}^T(v_{i,t}-\alpha_t^2)\right)\geq\exp(\kappa)\right]\notag\\
&=\P[u_{i,t+1}\geq\exp(\kappa)]\notag\\
&\leq \frac{\E[u_{i,t+1}]}{\exp(\kappa)}\notag\\
&\leq \exp(-\kappa),
\end{align*}
where the first inequality holds owing to Markov's inequality, and the last inequality is derived based on (\ref{lemhp-3}). Letting $\kappa=\ln\frac{1}{\delta}$ with $0<\delta<1$, it yields that with probability at least $1-\delta$,
\begin{align*}
\sum_{t=1}^T\alpha_t\langle \Delta_t,x_{i,t}-x_i\rangle\leq2L\iota\sum_{t=1}^T\alpha_t^2+2L\iota\ln\frac{1}{\delta}.
\end{align*}
The proof is completed.
\end{proof}

\subsection{Proof of Lemma \ref{lem:optimal}}\label{proof:lem:optimal}
\begin{proof}
It follow from (\ref{alg:dist_4}) and (\ref{pro-optimal}) that for $x_i\in X_i$,
\begin{align}
\langle x_{i,t}-\alpha_ts_{i,t+1}-x_{i,t+1}, x_i-x_{i,t+1}\rangle\leq 0.
\end{align}
Hence
\begin{align}\label{eq:optimal-1}
&\alpha_t \langle x_{i,t+1}-x_i, s_{i,t+1}\rangle\notag\\
&\leq \langle x_{i,t}-x_{i,t+1}, x_{i,t+1}-x_{i,t}\rangle\notag\\
&=\frac{1}{2}(\|x_{i,t}-x_i\|^2-\|x_{i,t+1}-x_i\|^2-\|x_{i,t+1}-x_{i,t}\|^2).
\end{align}
Note that
\begin{align}\label{eq:optimal-3}
&\alpha_t\langle x_i-x_{i,t+1},\nabla g_{i,t}(x_{i,t})\tilde{\mu}_{i,t+1}\rangle\notag\\
&=  \alpha_t(x_i-x_{i, t})^{\T} \nabla g_{i, t}(x_{i, t}) \tilde{\mu}_{i,t+1}\notag\\
&\quad +\alpha_t(x_{i, t}-x_{i, t+1})^{\T} \nabla g_{i, t}\left(x_{i, t}\right) \tilde{\mu}_{i, t+1}\notag \\
&\leq  \alpha_t \tilde{\mu}_{i, t+1}^{\T}(g_{i, t}(x_i)-g_{i, t}(x_{i, t})) \notag\\
&\quad +\alpha_t(x_{i, t}-x_{i, t+1})^{\T} \nabla g_{i, t}(x_{i, t}) \tilde{\mu}_{i, t+1}\notag\\
& =\alpha_t(\tilde{\mu}_{i, t+1}-\bar{\mu}_t)^{\T}(g_{i, t}(x_i)-g_{i, t}(x_{i, t}))\notag\\
&\quad +\alpha_t \bar{\mu}_t^{\T}(g_{i, t}(x_i)-g_{i, t}(x_{i, t}))\notag\\
&\quad +\alpha_t(x_{i, t}-x_{i, t+1})^{\T} \nabla g_{i, t}(x_{i, t}) \tilde{\mu}_{i, t+1}\notag\\
&\leq 2 L \alpha_t\|\tilde{\mu}_{i, t+1}-\bar{\mu}_t\|+\alpha_t \bar{\mu}_t^{\T}(g_{i, t}(x_i)-g_{i, t}(x_{i, t}))\notag\\
&\quad +\frac{1}{4}\|x_{i, t}-x_{i, t+1}\|^2+\alpha_t^2\|\nabla g_{i, t}(x_{i, t})\|^2\|\tilde{\mu}_{i, t+1}\|^2\notag\\
&\leq \frac{32 L^2N\sqrt{m}}{r}  \alpha_t \sum_{k=1}^{t} \vartheta^{t-k}\gamma_{k-1} +\alpha_t \bar{\mu}_t^{\T}(g_{i, t}(x_i)-g_{i, t}(x_{i, t})\notag \\
&\quad+\frac{1}{4}\|x_{i, t}-x_{i, t+1}\|^2+\frac{M^2L^2}{ r^4}\frac{\alpha_t^2}{\beta_t^2},
\end{align}
where the first inequality is obtained by using the convexity of $g_{ij,t}$ from Assumption \ref{ass:set}-iii)  and $\tilde{\mu}_{i,t+1}\geq {\bf 0}$, the second inequality is derived by (\ref{eq:bound}), the Young's inequality and  Cauchy-Schwarz inequality, and the last inequality holds based on (\ref{lem:mu-sigma-1}) in Lemma \ref{lem-consensus}, (\ref{lem-mu}) in Lemma \ref{lem:mu-w-z}, and \eqref{eq:bound}.
%Assumption \ref{ass:set}-iv).
Then it can be derived from (\ref{eq:optimal-1}) and (\ref{eq:optimal-3}) that
\begin{align}\label{eq:optimal-2}
&\alpha_t \langle x_{i,t+1}-x_i, q_{i,t}(x_{i,t},\tilde{\sigma}_{i,t+1},\xi_{i,t})\rangle\notag\\
&\leq \langle x_{i,t}-x_{i,t+1}, x_{i,t+1}-x_{i,t}\rangle\notag\\
&\quad+\alpha_t\langle x_i-x_{i,t+1},\nabla g_{i,t}(x_{i,t})\tilde{\mu}_{i,t+1}\rangle\notag\\
&=\frac{1}{2}(\|x_{i,t}-x_i\|^2-\|x_{i,t+1}-x_i\|^2-\|x_{i,t+1}-x_{i,t}\|^2)\notag\\
&\quad+\alpha_t\langle x_i-x_{i,t+1},\nabla g_{i,t}(x_{i,t})\tilde{\mu}_{i,t+1}\rangle\notag\\
&\leq \frac{1}{2}\|x_{i,t}-x_i\|^2-\frac{1}{2}\|x_{i,t+1}-x_i\|^2-\frac{1}{4}\|x_{i,t+1}-x_{i,t}\|^2\notag\\
&\quad+\frac{32 L^2N\sqrt{m}}{r}  \alpha_t \sum_{k=1}^{t} \vartheta^{t-k}\gamma_{k-1}\notag \\
&\quad +\alpha_t \bar{\mu}_t^{\T}(g_{i, t}(x_i)-g_{i, t}(x_{i, t})+\frac{M^2L^2}{ r^4}\frac{\alpha_t^2}{\beta_t^2}.
\end{align}
Hence
\begin{align*}%\label{eq:optimal-3}
&\alpha_t\langle p_{i,t}(x_{i,t},\sigma(x_t)),x_{i,t}-x_i\rangle\notag\\
&=\alpha_t\langle p_{i,t}(x_{i,t},\sigma(x_t))-p_{i,t}(x_{i,t},\tilde{\sigma}_{i,t+1})),x_{i,t}-x_i\rangle\notag\\
&\quad+\alpha_t\langle q_{i,t}(x_{i,t},\tilde{\sigma}_{i,t+1},\xi_{i,t}),x_{i,t}-x_{i,t+1}\rangle\notag\\
&\quad+\alpha_t\langle \Delta_t,x_{i,t}-x_i\rangle+\alpha_t\langle q_{i,t}(x_{i,t},\tilde{\sigma}_{i,t+1},\xi_{i,t}),x_{i,t+1}-x_i\rangle\notag\\
&\leq\alpha_t\|p_{i,t}(x_{i,t},\sigma(x_t))-p_{i,t}(x_{i,t},\tilde{\sigma}_{i,t+1}))\|\|x_{i,t}-x_i\|\notag\\
&\quad+\alpha_t^2\|q_{i,t}(x_{i,t},\tilde{\sigma}_{i,t+1},\xi_{i,t})\|^2+\frac{1}{4}\|x_{i,t+1}-x_{i,t}\|^2\notag\\
&\quad+\alpha_t\langle \Delta_t,x_{i,t}-x_i\rangle+\alpha_t\langle q_{i,t}(x_{i,t},\tilde{\sigma}_{i,t+1},\xi_{i,t}),x_{i,t+1}-x_i\rangle\notag\\
&\leq 2Ll_f\alpha_t\|\sigma(x_t)-\tilde{\sigma}_{i,t+1}\|+S^2\alpha_t^2+\frac{1}{4}\|x_{i,t+1}-x_{i,t}\|^2\notag\\
&\quad+\alpha_t\langle \Delta_t,x_{i,t}-x_i\rangle+\alpha_t\langle q_{i,t}(x_{i,t},\tilde{\sigma}_{i,t+1},\xi_{i,t}),x_{i,t+1}-x_i\rangle\notag\\
&\leq S^2\alpha_t^2+\frac{16Ll_f}{r}\|\sigma_0\|_1\alpha_t\vartheta^t+\frac{32 L^2N\sqrt{m}}{r}  \alpha_t \sum_{k=1}^{t} \vartheta^{t-k}\gamma_{k-1}\notag\\
&\quad+\frac{16LNSl_f\sqrt{n}L_\sigma}{r}\alpha_t\sum_{k=1}^t\vartheta^{t-k}\alpha_{k-1}\notag\\
&\quad+ \frac{16N\sqrt{n}L_\sigma L^2l_fM}{r^3}\alpha_t\sum_{k=1}^t\vartheta^{t-k}\frac{\alpha_{k-1}}{\beta_{k-1}}\notag\\
&\quad+\alpha_t\langle \Delta_t,x_{i,t}-x_i\rangle+\frac{1}{2}\|x_i-x_{i,t}\|^2-\frac{1}{2}\|x_i-x_{i,t+1}\|^2\notag\\
&\quad +\alpha_t \bar{\mu}_t^{\T}(g_{i, t}(x_i)-g_{i, t}(x_{i, t})) +\frac{M^2L^2}{ r^4}\frac{\alpha_t^2}{\beta_t^2},
\end{align*}
where the first inequality results from the Cauchy-Schwarz inequality and Young's inequality, the second inequality utilizes  Assumption \ref{ass:fun}, \ref{ass:sg}-iii), and (\ref{eq:bound}), and the last inequality holds due to (\ref{eq:optimal-2}) and (\ref{lem:mu-sigma-2}) in Lemma \ref{lem-consensus}. Dividing $\alpha_t$ from both sides of the above inequality immediately yields assertion (\ref{lem-optimal}), which completes the proof.
\end{proof}

\subsection{Proof of Lemma \ref{lem:mu-rel}}\label{proof:lem:mu-rel}
\begin{proof}
It follows from (\ref{mu}) and (\ref{mu-error}) that
\begin{align}
\bar{\mu}_{t+1}=\bar{\mu}_t+\frac{1}{N}\sum_{i=1}^N\varepsilon_{\mu_{i,t+1}}.
\end{align}
Hence, for any $\mu\in\R^m_+$, one has
\begin{align}\label{eq-1}
&\|\bar{\mu}_{t+1}-\mu\|^2\notag \\
&=\left\|\bar{\mu}_t-\mu+\frac{\sum_{i=1}^N \varepsilon_{\mu_{i, t+1}}}{N}\right\|^2\notag \\
& \leq\|\bar{\mu}_t-\mu\|^2+\frac{1}{N} \sum_{i=1}^N\|\varepsilon_{\mu_{i, t+1}}\|^2+\frac{2}{N} \sum_{i=1}^N \varepsilon_{\mu_{i, t+1}}^{\T}(\bar{\mu}_t-\mu)\notag \\
& \leq\|\bar{\mu}_t-\mu\|^2+\frac{\gamma_t^2}{N} \sum_{i=1}^N\|g_{i,t}(x_{i,t})-\beta_t \hat{\mu}_{i, t}\|^2\notag \\
&\quad+\frac{2}{N} \sum_{i=1}^N \varepsilon_{\mu_{i, t+1}}^{\T}(\bar{\mu}_t-\mu),
\end{align}
where the inequality $\|\sum_{i=1}^sa_i\|^2\leq s\sum_{i=1}^s\|a_i\|^2$  has been leveraged to obtain the first inequality, and the second inequality has utilized (\ref{pro-nonexpen}) and (\ref{mu-error}).

Next, let us deal with the last term  in the last inequality of (\ref{eq-1}). It can be derived that
\begin{align*}
& \varepsilon_{\mu_{i, t+1}}^{\T}(\bar{\mu}_t-\mu)\notag \\
%& =\frac{1}{z_{i, t+1}} \varepsilon_{\mu_{i, t+1}}^{\T}(z_{i, t+1} \bar{\mu}_t-z_{i, t+1} \mu)\notag \\
& =\frac{1}{z_{i, t+1}} \varepsilon_{\mu_{i, t+1}}^{\T}(\hat{\mu}_{i, t}-z_{i, t+1} \mu)+\varepsilon_{\mu_{i, t+1}}^{\T}(\bar{\mu}_t-\tilde{\mu}_{i, t+1})\notag \\
& =(\varepsilon_{\mu_{i, t+1}}-\gamma_t(g_{i,t}(x_{i,t})-\beta_t \hat{\mu}_{i, t}))^{\T} \frac{\left(\hat{\mu}_{i, t}-w_{i, t+1} \mu\right)}{z_{i, t+1}}\notag \\
&\quad +\gamma_t(g_{i,t}(x_{i,t})-\beta_t \hat{\mu}_{i, t})^{\T}(\tilde{\mu}_{i, t+1}-\mu)\notag\\
&\quad+\varepsilon_{\mu_{i, t+1}}^{\T}(\bar{\mu}_t-\tilde{\mu}_{i, t+1})\notag \\
& =(\varepsilon_{\mu_{i, t+1}}-\gamma_t(g_{i,t}(x_{i,t})-\beta_t \hat{\mu}_{i, t}))^{\T}  \frac{(\mu_{i, t+1}-z_{i, t+1} \mu)}{z_{i, t+1}}\notag \\
&\quad +(\varepsilon_{\mu_{i, t+1}}-\gamma_t(g_{i,t}(x_{i,t})-\beta_t \hat{\mu}_{i, t}))^{\T} \frac{(\hat{\mu}_{i, t}-\mu_{i, t+1})}{z_{i, t+1}}\notag\\
&\quad +\gamma_t(g_{i,t}(x_{i,t})-\beta_t \hat{\mu}_{i, t})^{\T}(\tilde{\mu}_{i, t+1}-\mu)\notag\\
&\quad+\varepsilon_{\mu_{i, t+1}}^{\T}(\bar{\mu}_t-\tilde{\mu}_{i, t+1})\notag\\
& \leq \frac{\gamma_t}{z_{i, t+1}}(g_{i,t}(x_{i,t})-\beta_t \hat{\mu}_{i, t})^{\T}(\mu_{i, t+1}-\hat{\mu}_{i, t})\notag\\
&\quad +\gamma_t(g_{i,t}(x_{i,t})-\beta_t \hat{\mu}_{i, t})^{\T}(\tilde{\mu}_{i, t+1}-\mu)\notag\\
&\quad+\varepsilon_{\mu_{i, t+1}}^{\T}(\bar{\mu}_t-\tilde{\mu}_{i, t+1})\notag\\
& \leq \frac{\gamma_t^2}{r}\|g_{i,t}(x_{i,t})-\beta_t \hat{\mu}_{i, t}\|^2\notag \\
&\quad +\gamma_t(g_{i,t}(x_{i,t})-\beta_t \hat{\mu}_{i, t})^{\T}(\tilde{\mu}_{i, t+1}-\mu)\notag\\
&\quad +\gamma_t\|g_{i,t}(x_{i,t})-\beta_t \hat{\mu}_{i, t}\|\|\tilde{\mu}_{i, t+1}-\bar{\mu}_t\|,
\end{align*}
where the first inequality is obtained by (\ref{pro-optimal}), and the last inequality follows from update (\ref{alg:dist_6}), Cauchy-Schwarz inequality, (\ref{lem-z}) in Lemma \ref{lem:mu-w-z}, and relation (\ref{pro-nonexpen}).

Then substituting the inequality above into (\ref{eq-1}) leads to
\begin{align}\label{eq-3}
&\|\bar{\mu}_{t+1}-\mu\|^2\notag \\
& \leq\|\bar{\mu}_t-\mu\|^2+\left(1+\frac{2}{r}\right)\frac{\gamma_t^2}{N} \sum_{i=1}^N\|g_{i,t}(x_{i,t})-\beta_t \hat{\mu}_{i, t}\|^2\notag \\
&\quad +\frac{2\gamma_t}{N}\sum_{i=1}^Ng_{i,t}(x_{i,t})^{\T}(\tilde{\mu}_{i, t+1}-\mu)  \notag\\ &\quad-\frac{\gamma_t\beta_t}{N}\sum_{i=1}^N 2\hat{\mu}_{i, t}^{\T}(\tilde{\mu}_{i, t+1}-\mu)\notag\\
&\quad +\frac{2\gamma_t}{N}\sum_{i=1}^N\|g_{i,t}(x_{i,t})-\beta_t \hat{\mu}_{i, t}\|\|\tilde{\mu}_{i, t+1}-\bar{\mu}_t\|.
%&\leq \|\bar{\mu}_{t+1}-\mu\|^2+4(1+\frac{2}{r})L^2\gamma_t^2\notag \\
%&\quad+\frac{2\gamma_t}{N}L\sum_{i=1}^N\|\tilde{\mu}_{i, t+1}-\bar{\mu}_t\|+\frac{2\gamma_t}{N}\sum_{i=1}^Ng_{i,t}(x_{i,t})^{\T}(\bar{\mu}_t-\mu)\notag\\
%&\quad +N\gamma_t\beta_t\|\mu\|^2 +\frac{2\gamma_t}{N}\sum_{i=1}^N 2L\|\tilde{\mu}_{i, t+1}-\bar{\mu}_t\|\notag\\
%&\leq \|\bar{\mu}_{t+1}-\mu\|^2+4(1+\frac{2}{r})L^2\gamma_t^2\notag \\
%&\quad+\frac{2\gamma_t}{N}3LN \frac{16N\sqrt{m}L}{r}\sum_{k=1}^t\vartheta^{t-k}\gamma_{k-1}\notag\\
%&\quad+\frac{2\gamma_t}{N}\sum_{i=1}^Ng_{i,t}(x_{i,t})^{\T}(\bar{\mu}_t-\mu)+N\gamma_t\beta_t\|\mu\|^2
\end{align}

For the second term on the right of (\ref{eq-3}), by (\ref{eq:bound}) and (\ref{lem-z}) in Lemma \ref{lem:mu-w-z}, one has
\begin{align}\label{eq-4}
\|g_{i,t}(x_{i,t})-\beta_t \hat{\mu}_{i, t}\|^2&\leq 2\|g_{i,t}(x_{i,t})\|^2+2\|\beta_t \hat{\mu}_{i, t}\|^2\notag\\
&\leq 2L^2+\frac{2L^2N^2}{r^4}.
\end{align}

Regarding the third term on the right of (\ref{eq-3}), it can be derived from (\ref{eq:bound}) that
\begin{align}\label{eq-5}
&g_{i,t}(x_{i,t})^{\T}(\tilde{\mu}_{i, t+1}-\mu)\notag\\
&= g_{i,t}(x_{i,t})^{\T}(\tilde{\mu}_{i, t+1}-\bar{\mu}_t)+g_{i,t}(x_{i,t})^{\T}(\bar{\mu}_t-\mu)\notag\\
&\leq L\|\tilde{\mu}_{i, t+1}-\bar{\mu}_t\|+g_{i,t}(x_{i,t})^{\T}(\bar{\mu}_t-\mu).
\end{align}

As for the fourth  term on the right of (\ref{eq-3}), invoking Lemma \ref{lem:mu-w-z}, one can obtain
\begin{align}\label{eq-6}
-2\hat{\mu}_{i, t}^{\T}(\tilde{\mu}_{i, t+1}-\mu)&=-2z_{i,t+1}\tilde{\mu}_{i, t+1}(\tilde{\mu}_{i, t+1}-\mu)\notag\\
&\leq z_{i,t+1}(\|\mu\|^2-\|\tilde{\mu}_{i, t+1}\|^2)\notag\\
&\leq N\|\mu\|^2.
\end{align}

Substituting (\ref{eq-4})-(\ref{eq-6}) and (\ref{lem:mu-sigma-1}) in Lemma \ref{lem-consensus} into (\ref{eq-3}) gives rise to (\ref{lem-mu-rel}). Hence, the proof is completed.
\end{proof}

\subsection{Proof of Lemma \ref{lem-regret}}\label{proof:lem-regret}

\begin{proof}
Setting $x_i=x_i^\star$ in (\ref{lem-optimal}) in Lemma \ref{lem:optimal}, it has
\begin{align}\label{lem-reg-1}
&\langle p_{i,t}(x_{i,t},\sigma(x_t)),x_{i,t}-x_i^\star\rangle\notag\\
&\leq S^2\alpha_t+\frac{16Ll_f}{r}\|\sigma_0\|_1\vartheta^t\notag\\
&\quad+\frac{16LNSl_f\sqrt{n}L_\sigma}{r}\sum_{k=1}^t\vartheta^{t-k}\alpha_{k-1}\notag\\
&\quad+ \frac{16N\sqrt{n}L_\sigma L^2l_fM}{r^3}\sum_{k=1}^t\vartheta^{t-k}\frac{\alpha_{k-1}}{\beta_{k-1}}\notag\\
&\quad+\frac{32 L^2N\sqrt{m}}{r} \sum_{k=1}^{t} \vartheta^{t-k}\gamma_{k-1}+\langle \Delta_t,x_{i,t}-x_i^\star\rangle\notag\\
&\quad+\frac{1}{2\alpha_t}\|x_i^\star-x_{i,t}\|^2-\frac{1}{2\alpha_t}\|x_i^\star-x_{i,t+1}\|^2\notag\\
&\quad + \bar{\mu}_t^{\T}(g_{i, t}(x_i^\star)-g_{i, t}(x_{i, t}))+\frac{M^2L^2}{ r^4}\frac{\alpha_t}{\beta_t^2}.
\end{align}

%\sum_{\substack{j=1\\j\neq i}}^N
Since $x_i^\star\in X_{i,t}(x_{-i,t})$, that is, $g_{i,t}(x_i^\star)+\sum_{j=1,j\neq i}^Ng_{j,t}(x_{j,t})\leq {\bf 0}$ and noting that $\bar{\mu}_t\geq {\bf 0}$, one can obtain that
\begin{align}\label{lem-reg-2}
&\bar{\mu}_t^{\T}(g_{i, t}(x_i^\star)-g_{i, t}(x_{i, t}))\notag\\
&=\bar{\mu}_t^{\T}\left(g_{i, t}(x_i^\star)+\sum_{j=1,j\neq i}^Ng_{j,t}(x_{j,t})-g_t(x_t)\right)\notag\\
&\leq -\bar{\mu}_t^{\T}g_t(x_t).
\end{align}
Substituting (\ref{lem-reg-2}) into (\ref{lem-reg-1}) results in
\begin{align}\label{lem-reg-3}
&\langle p_{i,t}(x_{i,t},\sigma(x_t)),x_{i,t}-x_i^\star\rangle\notag\\
&\leq S^2\alpha_t+\frac{16Ll_f}{r}\|\sigma_0\|_1\vartheta^t\notag\\
&\quad+\frac{16LNSl_f\sqrt{n}L_\sigma}{r}\sum_{k=1}^t\vartheta^{t-k}\alpha_{k-1}\notag\\
&\quad+ \frac{16N\sqrt{n}L_\sigma L^2l_fM}{r^3}\sum_{k=1}^t\vartheta^{t-k}\frac{\alpha_{k-1}}{\beta_{k-1}}\notag\\
&\quad+\frac{32 L^2N\sqrt{m}}{r} \sum_{k=1}^{t} \vartheta^{t-k}\gamma_{k-1}- \bar{\mu}_t^{\T}g_t(x_t)\notag\\
&\quad+\frac{M^2L^2}{ r^4}\frac{\alpha_t}{\beta_t^2}+\langle \Delta_t,x_{i,t}-x_i^\star\rangle\notag\\
&\quad+\frac{1}{2\alpha_t}\|x_i^\star-x_{i,t}\|^2-\frac{1}{2\alpha_t}\|x_i^\star-x_{i,t+1}\|^2.
\end{align}

Dividing both sides of (\ref{lem-mu-rel}) by $\frac{2\gamma_t}{N}$, one can obtain
\begin{align}\label{lem-reg-4}
&\frac{N}{2\gamma_t}\|\bar{\mu}_{t+1}-\mu\|^2\notag \\
&\leq \frac{N}{2\gamma_t}\|\bar{\mu}_{t}-\mu\|^2+\left(1+\frac{2}{r}\right)\left(1+\frac{N^2}{r^4}\right)NL^2\gamma_t\notag \\
&\quad+\frac{16N^2\sqrt{m}L^2}{r}\left(2+\frac{N}{r^2}\right)\sum_{k=1}^t\vartheta^{t-k}\gamma_{k-1}\notag\\
&\quad+g_t(x_t)^{\T}(\bar{\mu}_t-\mu)+\frac{N^2}{2}\beta_t\|\mu\|^2.
\end{align}

Then combining (\ref{lem-reg-4}) and (\ref{lem-reg-3}) yields that
\begin{align}\label{lem-reg-5}
&\langle p_{i,t}(x_{i,t},\sigma(x_t)),x_{i,t}-x_i^\star\rangle\notag\\
&\leq S^2\alpha_t+\frac{16Ll_f}{r}\|\sigma_0\|_1\vartheta^t- \mu^{\T}g_t(x_t)+\langle \Delta_t,x_{i,t}-x_i^\star\rangle\notag\\
&\quad+\frac{16LNSl_f\sqrt{n}L_\sigma}{r}\sum_{k=1}^t\vartheta^{t-k}\alpha_{k-1}\notag\\
&\quad+ \frac{16N\sqrt{n}L_\sigma L^2l_fM}{r^3}\sum_{k=1}^t\vartheta^{t-k}\frac{\alpha_{k-1}}{\beta_{k-1}}+\frac{M^2L^2}{ r^4}\frac{\alpha_t}{\beta_t^2} \notag\\
&\quad+\frac{1}{2\alpha_t}\|x_i^\star-x_{i,t}\|^2-\frac{1}{2\alpha_t}\|x_i^\star-x_{i,t+1}\|^2\notag\\
&\quad - \mu^{\T}g_t(x_t)+\frac{N^2}{2}\beta_t\|\mu\|^2+\left(1+\frac{2}{r}\right)\left(1+\frac{N^2}{r^4}\right)NL^2\gamma_t\notag \\
&\quad+\frac{16N\sqrt{m}L^2}{r}\left(2(N+1)+\frac{N^2}{r^2}\right)\sum_{k=1}^{t} \vartheta^{t-k}\gamma_{k-1}\notag\\
&\quad+\frac{N}{2\gamma_t}\|\bar{\mu}_{t}-\mu\|^2-\frac{N}{2\gamma_t}\|\bar{\mu}_{t+1}-\mu\|^2.
\end{align}

Summing over $t=1,\ldots,T$ from both sides of (\ref{lem-reg-5}), it is straightforward to obtain that
\begin{align}\label{lem-reg-7}
&\sum_{t=1}^T\langle p_{i,t}(x_{i,t},\sigma(x_t)),x_{i,t}-x_i^\star\rangle\notag\\
&\leq S^2\sum_{t=1}^T\alpha_t+\frac{16Ll_f}{r}\|\sigma_0\|_1\sum_{t=1}^T\vartheta^t+\sum_{t=1}^T\langle \Delta_t,x_{i,t}-x_i^\star\rangle\notag\\
&\quad+\sum_{t=1}^T\frac{1}{2\alpha_t}(\|x_i^\star-x_{i,t}\|^2-\|x_i^\star-x_{i,t+1}\|^2)\notag\\
&\quad+\left(1+\frac{2}{r}\right)\left(1+\frac{N^2}{r^4}\right)NL^2\sum_{t=1}^T\gamma_t\notag \\
&\quad+\frac{16LNSl_f\sqrt{n}L_\sigma}{r}\sum_{t=1}^T\sum_{k=1}^t\vartheta^{t-k}\alpha_{k-1}\notag\\
&\quad+ \frac{16N\sqrt{n}L_\sigma L^2l_fM}{r^3}\sum_{t=1}^T\sum_{k=1}^t\vartheta^{t-k}\frac{\alpha_{k-1}}{\beta_{k-1}}\notag\\
&\quad+\frac{16N\sqrt{m}L^2}{r}\left(2(N+1)+\frac{N^2}{r^2}\right)\sum_{t=1}^T\sum_{k=1}^{t} \vartheta^{t-k}\gamma_{k-1}\notag\\
&\quad - \mu^{\T}\sum_{t=1}^Tg_t(x_t)+\frac{N^2}{2}\|\mu\|^2\sum_{t=1}^T\beta_t+\frac{M^2L^2}{ r^4}\sum_{t=1}^T\frac{\alpha_t}{\beta_t^2}\notag \\
&\quad+\sum_{t=1}^T\frac{N}{2\gamma_t}(\|\bar{\mu}_{t}-\mu\|^2-\|\bar{\mu}_{t+1}-\mu\|^2).
\end{align}

Note that
\begin{align}\label{lem-reg-9}
\sum_{t=1}^T\sum_{k=1}^t\vartheta^{t-k}\alpha_{k-1}&=\sum_{t=0}^{T-1}\vartheta^t\sum_{k=0}^{T-t-1}\alpha_k\notag\\
%&\leq \sum_{t=0}^{T-1}\vartheta^t\sum_{k=0}^{T-1}\eta_k\notag\\
&\leq\frac{1}{1-\vartheta}\sum_{t=0}^{T-1}\alpha_t.
\end{align}
Similarly, one has
\begin{align}\label{lem-reg-10}
&\sum_{t=1}^T\sum_{k=1}^t\vartheta^{t-k}\frac{\alpha_{k-1}}{\beta_{k-1}}\leq \frac{1}{1-\vartheta}\sum_{t=0}^{T-1} \frac{\alpha_{t}}{\beta_{t}},\notag\\
&\sum_{t=1}^T\sum_{k=1}^t\vartheta^{t-k}\gamma_{k-1}\leq\frac{1}{1-\vartheta}\sum_{t=0}^{T-1}\gamma_t.
\end{align}

For the third term on the right of (\ref{lem-reg-7}), in view of (\ref{lem-hp}) in Lemma \ref{lem:hp} and  $\alpha_{t+1}\leq \alpha_t$, it can be obtained that with probability at least $1-\delta$,
\begin{align}\label{lem-reg-8}
\sum_{t=1}^T\langle \Delta_t,x_{i,t}-x_i^\star\rangle\leq \frac{2L\iota\sum_{t=1}^T\alpha_t^2}{\alpha_T}+\frac{2L\iota\ln\frac{1}{\delta}}{\alpha_T}.
\end{align}

For the fourth term on the right of (\ref{lem-reg-7}), it derives that
\begin{align}\label{lem-reg-11}
&\sum_{t=1}^T\frac{1}{2\alpha_t}(\|x_i^\star-x_{i,t}\|^2-\|x_i^\star-x_{i,t+1}\|^2)\notag\\
&=\sum_{t=1}^T\left(\frac{1}{2\alpha_t}\|x_i^\star-x_{i,t}\|^2-\frac{1}{2\alpha_{t+1}}\|x_i^\star-x_{i,t+1}\|^2\right)\notag\\
&\quad +\sum_{t=1}^T\left(\frac{1}{2\alpha_{t+1}}-\frac{1}{2\alpha_t}\right)\|x_i^\star-x_{i,t+1}\|^2\notag\\
&\leq \frac{1}{\alpha_1}(\|x_i^\star\|^2+\|x_{i,1}\|^2)\notag\\
&\quad+\sum_{t=1}^T\left(\frac{1}{\alpha_{t+1}}-\frac{1}{\alpha_t}\right)(\|x_i^\star\|^2+\|x_{i,t+1}\|^2)\notag\\
&\leq \frac{2L^2}{\alpha_{T+1}},
\end{align}
where (\ref{eq:bound}) has been used to get the last inequality.

Denote by $A(\mu):=\sum_{t=1}^T\frac{N}{2\gamma_t}(\|\bar{\mu}_{t}-\mu\|^2-\|\bar{\mu}_{t+1}-\mu\|^2)$ and $B(\mu):=- \mu^{\T}\sum_{t=1}^Tg_t(x_t)+\frac{N^2}{2}\|\mu\|^2\sum_{t=1}^T\beta_t$. Following the same line as (\ref{lem-reg-11}), one can obtain that
%\begin{align}\label{lem-reg-13}
%&\sum_{t=1}^T\frac{N}{2\gamma_t}(\|\bar{\mu}_{t}-\mu\|^2-\|\bar{\mu}_{t+1}-\mu\|^2)\notag\\
%&\leq \frac{N\|\bar{\mu}_1-\mu\|^2}{2\alpha_1}+\sum_{t=1}^T(\frac{N}{2\alpha_{t+1}}-\frac{N}{2\alpha_t})\|\bar{\mu}_{t+1}-\mu\|^2.
%&\leq\frac{L^2}{2\alpha_{T+1}\beta_{T+1}^2},
%\end{align}
\begin{align}\label{lem-reg-12}
A({\bf 0})&=\sum_{t=1}^T\frac{N}{2\gamma_t}(\|\bar{\mu}_{t}\|^2-\|\bar{\mu}_{t+1}\|^2)\notag\\
&\leq \frac{N\|\bar{\mu}_1\|^2}{2\gamma_1}+\sum_{t=1}^T\left(\frac{N}{2\gamma_{t+1}}-\frac{N}{2\gamma_t}\right)\|\bar{\mu}_{t+1}\|^2\notag\\
&\leq\frac{N^2L^2}{2r^4\underline{w}^2\gamma_{T+1}\beta_{T+1}^2},
\end{align}
where we  utilize $\|\bar{\mu}_{t}\|^2\leq\frac{1}{N}\|\mu_{i,t}\|^2\leq \frac{NL^2}{r^4\underline{w}^2\beta_t^2}\leq \frac{NL^2}{r^4\underline{w}^2\beta_{T+1}^2}$ by Lemma \ref{lem:mu-w-z} to obtain the last inequality.

Note that
\begin{align*}
\mathcal{R}_i(T)\leq \sum_{t=1}^T\langle p_{i,t}(x_{i,t},\sigma(x_t)),x_{i,t}-x_i^\star\rangle,
\end{align*}
then substituting (\ref{lem-reg-9})-(\ref{lem-reg-12}) into (\ref{lem-reg-7}), and invoking $\alpha_t\leq\alpha_{t-1}$ and $\gamma_t\leq\gamma_{t-1}$, and $B({\bf 0})={\bf 0}$, assertion (\ref{lem-regret-i}) can be obtained.

In addition, through a simple calculation, one can obtain that $B(\mu)$ achieves its minimal value $-\frac{\|[\sum_{t=1}^Tg_t(x_t)]_+\|^2}{2N^2\sum_{t=1}^T\beta_t}$ when $\mu=\mu_0:=\frac{[\sum_{t=1}^Tg_t(x_t)]_+}{N^2\sum_{t=1}^T\beta_t}$.

%%\begin{align}\label{lem-regret-13}
%%&\sum_{t=1}^T\langle p_{i,t}(x_{i,t},\sigma(x_t)),x_{i,t}-x_i^\star\rangle\notag\\
%%&\leq \frac{16Ll_f\vartheta\|\sigma_0\|_1}{r(1-\vartheta)}+B^2\sum_{t=1}^T\alpha_t+\frac{2LNSl_f\sqrt{N}L_\sigma}{1-\vartheta}\sum_{t=1}^{T}\alpha_{t-1}\notag\\
%%&\quad+\frac{16(2 N+3N^2)L^2\sqrt{m}}{r(1-\vartheta)}\sum_{t=1}^{T}\gamma_{t-1} +2N(1+\frac{2}{r})L^2\sum_{t=1}^T\gamma_t\notag\\
%%&\quad+2L\sigma\frac{\sum_{t=1}^T\alpha_t^2}{\alpha_T}+\frac{2L\sigma\ln\frac{1}{\delta}}{\alpha_T}+\frac{2L^2}{\alpha_{T+1}}+A(\mu)\notag\\
%%&\quad +\frac{M^2L^2}{ r^2}\sum_{t=1}^T\frac{\alpha_t}{\beta_t^2}+B(\mu).
%%\end{align}

%\begin{align}\label{lem-regret-14}
%&\sum_{t=1}^T\langle p_{i,t}(x_{i,t},\sigma(x_t)),x_{i,t}-x_i^\star\rangle\notag\\
%&\leq \frac{16Ll_f\vartheta\|\sigma_0\|_1}{r(1-\vartheta)}+B^2\sum_{t=1}^T\alpha_t+\frac{2LNSl_f\sqrt{N}L_\sigma}{1-\vartheta}\sum_{t=1}^{T}\alpha_{t-1}\notag\\
%&\quad+\frac{16(2 N+3N^2)L^2\sqrt{m}}{r(1-\vartheta)}\sum_{t=1}^{T}\gamma_{t-1} +2N(1+\frac{2}{r})L^2\sum_{t=1}^T\gamma_t\notag\\
%&\quad+2L\sigma\frac{\sum_{t=1}^T\alpha_t^2}{\alpha_T}+\frac{2L\sigma\ln\frac{1}{\delta}}{\alpha_T}+\frac{2L^2}{\alpha_{T+1}}+A(\mu_0)\notag\\
%&\quad +\frac{M^2L^2}{ r^2}\sum_{t=1}^T\frac{\alpha_t}{\beta_t^2}-\frac{(\mathcal{R}_g(T))^2}{2N^2\sum_{t=1}^T\beta_t}.
%\end{align}

In view of (\ref{eq:bound}), it obtains that $\|\mu_0\|\leq \frac{TL}{N\sum_{t=1}^T\beta_t}$. Hence, resorting to the similar argument as (\ref{lem-reg-11}), it gives
\begin{align}\label{lem-reg-16}
%A({\mu_0})&=\sum_{t=1}^T\frac{N}{2\gamma_t}(\|\bar{\mu}_{t}-\mu_0\|^2-\|\bar{\mu}_{t+1}-\mu_0\|^2)\notag\\
%&\leq \frac{N(\|\bar{\mu}_1\|^2+\|\mu_0\|^2)}{\gamma_1}\notag\\
%&\quad+\sum_{t=1}^T(\frac{N}{\gamma_{t+1}}-\frac{N}{\gamma_t})(\|\bar{\mu}_{t+1}\|^2+\|\mu_0\|^2)\notag\\
%&\leq\frac{T^2L^2}{\gamma_{T+1}(\sum_{t=1}^T\beta_t)^2},
A({\mu_0})\leq\frac{N^2L^2}{r^4w^2\gamma_{T+1}\beta_{T+1}^2}+\frac{T^2L^2}{N\gamma_{T+1}\left(\sum_{t=1}^T\beta_t\right)^2},
\end{align}

Incorporating (\ref{lem-reg-7})-(\ref{lem-reg-11}) and (\ref{lem-reg-16}), it can be concluded that
%\begin{align}\label{lem-regret-14}
%&0\notag\\
%&\leq \frac{16Ll_f\vartheta\|\sigma_0\|_1}{r(1-\vartheta)}+B^2\sum_{t=1}^T\alpha_t+\frac{2LNSl_f\sqrt{N}L_\sigma}{1-\vartheta}\sum_{t=1}^{T}\alpha_{t-1}\notag\\
%&\quad+\frac{16(2 N+3N^2)L^2\sqrt{m}}{r(1-\vartheta)}\sum_{t=1}^{T}\gamma_{t-1} +2N(1+\frac{2}{r})L^2\sum_{t=1}^T\gamma_t\notag\\
%&\quad+2L\sigma\frac{\sum_{t=1}^T\alpha_t^2}{\alpha_T}+\frac{2L\sigma\ln\frac{1}{\delta}}{\alpha_T}+\frac{2L^2}{\alpha_{T+1}}\notag\\
%&\quad +\frac{L^2}{\gamma_{T+1}\beta_{T+1}^2}+\frac{T^2L^2}{\gamma_{T+1}(\sum_{t=1}^T\beta_t)^2}\notag\\
%&\quad +\frac{M^2L^2}{ r^2}\sum_{t=1}^T\frac{\alpha_t}{\beta_t^2}-\frac{(\mathcal{R}_g(T))^2}{2N^2\sum_{t=1}^T\beta_t}.
%\end{align}
\begin{align*}%\label{lem-reg-17}
&\sum_{t=1}^T\langle p_{i,t}(x_{i,t},\sigma(x_t)),x_{i,t}-x_i^\star\rangle\notag\\
&\leq S_1(T)+S_2(T)+S_5(T)-\frac{\mathcal{R}^2_g(T)}{S_4(T)},
\end{align*}
and thus by noting that $\langle p_{i,t}(x_{i,t},\sigma(x_t)),x_{i,t}-x_i^\star\rangle\geq0$, it holds
\begin{align*}%\label{lem-reg-18}
&\mathcal{R}^2_g(T)\leq S_4(T)(S_1(T)+S_2(T)+S_5(T)).
\end{align*}
The proof is completed.

%\begin{align}
%&(\mathcal{R}_g(T))^2\leq 2N^2\sum_{t=1}^T\beta_t\notag\\
%&[\frac{16Ll_f\vartheta\|\sigma_0\|_1}{r(1-\vartheta)}+B^2\sum_{t=1}^T\alpha_t+\frac{2LNSl_f\sqrt{N}L_\sigma}{1-\vartheta}\sum_{t=1}^{T}\alpha_{t-1}\notag\\
%&\quad+\frac{16(2 N+3N^2)L^2\sqrt{m}}{r(1-\vartheta)}\sum_{t=1}^{T}\gamma_{t-1} +2N(1+\frac{2}{r})L^2\sum_{t=1}^T\gamma_t\notag\\
%&\quad+2L\sigma\frac{\sum_{t=1}^T\alpha_t^2}{\alpha_T}+\frac{2L\sigma\ln\frac{1}{\delta}}{\alpha_T}+\frac{2L^2}{\alpha_{T+1}}\notag\\
%&\quad +\frac{L^2}{\gamma_{T+1}\beta_{T+1}^2}+\frac{T^2L^2}{\gamma_{T+1}(\sum_{t=1}^T\beta_t)^2}\notag\\
%&\quad +\frac{M^2L^2}{ r^2}\sum_{t=1}^T\frac{\alpha_t}{\beta_t^2}]
%\end{align}
\end{proof}

\subsection{Proof of Theorem \ref{thm:reg}}\label{proof:thm:reg}
\begin{proof}
Note that for any $0<a\neq 1$,
\begin{align}
\sum_{t=1}^T \frac{1}{t^a} \leq 1+\int_1^T \frac{1}{t^a} d t \leq 1+\frac{T^{1-a}-1}{1-a} \leq \frac{T^{1-a}}{1-a},
\end{align}
and for any $T\geq 4$ and $a\in(0,1/2]$,%a\in(0.\frac{1}{4}
\begin{align}
\sum_{t=1}^T \frac{1}{t^a}\geq \int_1^T\frac{1}{t^a} d t =\frac{T^{1-a}-1}{1-a} \geq \frac{T^{1-a}}{2(1-a)}.
\end{align}
Therefore, simple manipulations lead to
\begin{align*}
S_1(T)=\O(T^{1-a_3}+T^{1-a_1+2a_2}+T^{a_1}), \notag\\ S_2(T)=\O\left(T^{a_1}\ln\frac{1}{\delta}\right),\ S_3(T)=\O(T^{2a_2+a_3}),\notag\\
S_4(T)=\O(T^{1-a_2}),\ S_5(T)=\O(T^{2a_2+a_3}).
\end{align*}
Subsequently, by virtue of Lemma \ref{lem-regret},  assertions (\ref{thm:reg-1}) and (\ref{thm:reg-2}) can be derived.
%\begin{align*}
%\mathcal{R}_i(T)=\O\left(T^{\max\{a_1,1-a_3,1-a_1+2a_2,2a_2+a_3\}}+T^{a_1}\ln\frac{1}{\delta}\right),
%\end{align*}
%and
%\begin{align*}
%\mathcal{R}_g(T)&=\O\bigg(T^{\max\{1-\frac{a_2}{2}-\frac{a_3}{2},1-\frac{a_1}{2}+\frac{a_2}{2},\frac{1}{2}+\frac{a_1}{2}-\frac{a_2}{2}
%,\frac{1}{2}-\frac{a_1}{2}+\frac{a_3}{2}\}}\\
%&\quad+T^{\frac{a_1}{2}+\frac{a_2}{2}+\frac{a_3}{2}}\sqrt{\ln \frac{1}{\delta}}\bigg).
%\end{align*}
This completes the proof.
\end{proof}

\subsection{Proof of Corollary \ref{cor:spec}}\label{proof:cor:spec}
\begin{proof}
Let $a_1=\frac{1}{2}+a_2$ and $a_3=\frac{1}{2}-a_2$, then $a_1=1-a_3=1-a_1+2a_2=2a_2+a_3=\frac{1}{2}+a_2$. Moreover, $\max\left\{1-\frac{a_2+a_3}{2},1-\frac{a_1+a_2}{2},\frac{1+a_1+a_2}{2}
,\frac{1-a_1+a_3}{2}\right\}=\frac{3}{4}+a_2$ and $\frac{a_1+a_2+a_3}{2}=\frac{1+a_2}{2}$. Hence, the proof is completed.
\end{proof}

\subsection{Proof of Theorem \ref{thm:conver}}\label{proof:thm:conver}

\begin{proof}
Setting $x_i=x_i^*$ in (\ref{lem-optimal}) of Lemma \ref{lem:optimal},
%\begin{align*}
%&\|x_{i,t+1}-x_i^*\|^2\notag\\
%&\leq\|x_{i,t}-x_i^*\|^2-2\alpha_t\langle p_{i,t}(x_{i,t},\sigma(x_t)),x_{i,t}-x_i^*\rangle\notag\\
%&\quad + 2B^2\alpha_t^2+\frac{32Ll_f}{r}\|\sigma_0\|_1\vartheta^t\alpha_t\notag\\
%&\quad+4LNSl_f\sqrt{N}L_\sigma\alpha_t\sum_{k=1}^t\vartheta^{t-k}\alpha_{k-1}\notag\\
%&\quad+\frac{64 L^2N\sqrt{m}}{r}\alpha_t\sum_{k=1}^{t} \vartheta^{t-k}\gamma_{k-1}\notag\\
%&\quad+2\alpha_t\langle p_{i,t}(x_{i,t},\tilde{\sigma}_{i,t+1}))-q_{i,t}(x_{i,t},\tilde{\sigma}_{i,t+1},\xi_{i,t}),x_{i,t+1}-x_i^*\rangle\notag\\
%&\quad +2\alpha_t \bar{\mu}_t^{\T}(g_{i}(x_i^*)-g_{i}(x_{i, t}))+\frac{2M^2L^2}{ r^2}\frac{\alpha_t^2}{\beta_t^2}
%\end{align*}
then summing over $i\in [N]$ and rearranging the terms, one can obtain
\begin{align}\label{NE-2}
&\|x_{t+1}-x^*\|^2\notag\\
&\leq\|x_{t}-x^*\|^2-2\alpha_t\sum_{i=1}^N\langle p_{i}(x_{i,t},\sigma(x_t)),x_{i,t}-x_i^*\rangle\notag\\
&\quad +2\alpha_t\sum_{i=1}^N\langle \Delta_t,x_{i,t}-x_i^*\rangle + 2NS^2\alpha_t^2+\frac{32NLl_f}{r}\|\sigma_0\|_1\vartheta^t\alpha_t\notag\\
&\quad+\frac{32LN^2Sl_f\sqrt{n}L_\sigma}{r}\alpha_t\sum_{k=1}^t\vartheta^{t-k}\alpha_{k-1}\notag\\
&\quad+ \frac{32N^2\sqrt{n}L_\sigma L^2l_fM}{r^3}\alpha_t\sum_{k=1}^t\vartheta^{t-k}\frac{\alpha_{k-1}}{\beta_{k-1}}\notag\\
&\quad+\frac{64 L^2N^2\sqrt{m}}{r}\alpha_t\sum_{k=1}^{t} \vartheta^{t-k}\gamma_{k-1}\notag\\
&\quad +2\alpha_t \bar{\mu}_t^{\T}(g(x^*)-g(x_{ t}))+\frac{2NM^2L^2}{ r^4}\frac{\alpha_t^2}{\beta_t^2}.
\end{align}
Note that
\begin{align}\label{NE-3}
&\sum_{i=1}^N\langle p_{i}(x_{i,t},\sigma(x_t)),x_i^*-x_{i,t}\rangle\notag\\
%&= \langle x^*-x_t,G(x_t)-G(x^*)\rangle+\langle x^*-x_t,G(x^*)\rangle\notag\\
&= \langle x^*-x_t,G(x_t)-G(x^*)\rangle+\langle x^*-x_t,G(x^*)+\nabla g(x^*)\mu^*\rangle\notag\\
&\quad+\langle x_t-x^*,\nabla g(x^*)\mu^*\rangle\notag\\
&\leq \langle x^*-x_t,G(x_t)-G(x^*)\rangle + {\mu^*}^{\T}(g(x_t)-g(x^*))\notag\\
&= \langle x^*-x_t,G(x_t)-G(x^*)\rangle + {\mu^*}^{\T}g(x_t),
\end{align}
where
%the first inequality is due to Assumption \ref{ass:strong-mo},
the inequality is derived by noting that $(x^*,\mu^*)$ is a saddle point of  the Lagrangian function $L_i$, the convexity of $g_{ij}$ and nonnegativity of $\mu^*$, and the last equality holds by virtue of the KKT condition ${\mu^*}^{\T}g(x^*)=0$.

Substituting (\ref{NE-3}) into (\ref{NE-2}), and noting that $\bar{\mu}_t^{\T}g(x^*)\leq 0$, it can be derived that
\begin{align}\label{NE-4}
&\frac{1}{2\alpha_t}\|x_{t+1}-x^*\|^2\notag\\
&\leq\frac{1}{2\alpha_t}\|x_{t}-x^*\|^2+ (\mu^*-\bar{\mu}_t)^{\T}g(x_t)\notag\\
&\quad +\langle x^*-x_t,G(x_t)-G(x^*)\rangle\notag\\
&\quad +\sum_{i=1}^N\langle \Delta_t,x_{i,t}-x_i^*\rangle + NB^2\alpha_t+\frac{16NLl_f}{r}\|\sigma_0\|_1\vartheta^t\notag\\
&\quad+\frac{16LN^2Sl_f\sqrt{n}L_\sigma}{r}\sum_{k=1}^t\vartheta^{t-k}\alpha_{k-1}\notag\\
&\quad+ \frac{16N^2\sqrt{n}L_\sigma L^2l_fM}{r^3}\sum_{k=1}^t\vartheta^{t-k}\frac{\alpha_{k-1}}{\beta_{k-1}}\notag\\
&\quad+\frac{32 L^2N^2\sqrt{m}}{r}\sum_{k=1}^{t} \vartheta^{t-k}\gamma_{k-1}+\frac{NM^2L^2}{ r^4}\frac{\alpha_t}{\beta_t^2}.
\end{align}

Setting $\mu=\mu^*$ in  (\ref{lem-reg-4}), and based on $\|\mu^*\|\leq \chi$, one can obtain
\begin{align}\label{NE-5}
&\frac{N}{2\gamma_t}\|\bar{\mu}_{t+1}-\mu^*\|^2\notag \\
&\leq \frac{N}{2\gamma_t}\|\bar{\mu}_{t}-\mu^*\|^2+\left(1+\frac{2}{r}\right)\left(1+\frac{N^2}{r^4}\right)NL^2\gamma_t\notag \\
&\quad+\frac{16N^2\sqrt{m}L^2}{r}\left(2+\frac{N}{r^2}\right)\sum_{k=1}^t\vartheta^{t-k}\gamma_{k-1}\notag\\
&\quad+g(x_t)^{\T}(\bar{\mu}_t-\mu^*)+\frac{N^2}{2}\beta_t\chi^2.
\end{align}

Combining (\ref{NE-4}) and (\ref{NE-5}), we can get
\begin{align}\label{NE-6}
&\langle x_t-x^*,G(x_t)-G(x^*)\rangle\notag\\
&\leq\frac{1}{2\alpha_t}\|x_{t}-x^*\|^2 - \frac{1}{2\alpha_t}\|x_{t+1}-x^*\|^2\notag\\
&\quad +\frac{N}{2\gamma_t}\|\bar{\mu}_{t}-\mu^*\|^2- \frac{N}{2\gamma_t}\|\bar{\mu}_{t+1}-\mu^*\|^2\notag \\
&\quad+ \left(1+\frac{2}{r}\right)\left(1+\frac{N^2}{r^4}\right)NL^2\gamma_t\notag\\
&\quad+\frac{16N\sqrt{m}L^2}{r}\left(2(N+1)+\frac{N^2}{r^2}\right)\sum_{k=1}^t\vartheta^{t-k}\gamma_{k-1}\notag\\
&\quad+\frac{N^2}{2}\beta_t\chi^2 + NS^2\alpha_t+\frac{16NLl_f}{r}\|\sigma_0\|_1\vartheta^t\notag\\
&\quad+\frac{16LN^2Sl_f\sqrt{n}L_\sigma}{r}\sum_{k=1}^t\vartheta^{t-k}\alpha_{k-1}\notag\\
&\quad+ \frac{16N^2\sqrt{n}L_\sigma L^2l_fM}{r^3}\sum_{k=1}^t\vartheta^{t-k}\frac{\alpha_{k-1}}{\beta_{k-1}}\notag\\
&\quad+\frac{NM^2L^2}{ r^4}\frac{\alpha_t}{\beta_t^2}+\sum_{i=1}^N\langle \Delta_t,x_{i,t}-x_i^*\rangle.
\end{align}
Next, taking conditional expectation with respect to $\F_t$ on both sides of (\ref{NE-6}), since $\E[\langle\Delta_t,x_{i,t}-x_i^*\rangle|\F_t]=0$ due to Assumption \ref{ass:sg}-i) and $\alpha_t=\gamma_t$, it can be derived that
\begin{align}\label{NE-6-1}
&\E[\|x_{t+1}-x^*\|^2+N\|\bar{\mu}_{t+1}-\mu^*\|^2|\F_k]\notag\\
&\leq\|x_{t}-x^*\|^2 + N\|\bar{\mu}_{t}-\mu^*\|^2\notag\\
&\quad-2\mu\alpha_t\langle x_t-x^*,G(x_t)-G(x^*)\rangle\notag \\
&\quad+ 2N\left(\left(1+\frac{2}{r}\right)\left(1+\frac{N^2}{r^4}\right)L^2+S^2\right)\alpha_t^2\notag\\
&\quad+\frac{32NL}{r}\left(\sqrt{m}L\left(2(N+1)+\frac{N^2}{r^2}\right)+NSl_f\sqrt{n}L_\sigma\right)\notag\\
&\quad\times\alpha_t\sum_{k=1}^t\vartheta^{t-k}\alpha_{k-1}+N^2\alpha_t\beta_t\chi^2 +\frac{32NLl_f}{r}\|\sigma_0\|_1\vartheta^t\alpha_t\notag\\
&\quad+ \frac{32N^2\sqrt{n}L_\sigma L^2l_fM}{r^3}\alpha_t\sum_{k=1}^t\vartheta^{t-k}\frac{\alpha_{k-1}}{\beta_{k-1}}\notag\\
&\quad+\frac{2NM^2L^2}{ r^4}\frac{\alpha_t^2}{\beta_t^2}.
\end{align}

Note that
\begin{align*}
\sum_{t=1}^\infty\vartheta^t\alpha_t&\leq \alpha_0\sum_{t=1}^\infty\vartheta^t\leq\frac{ \alpha_0}{1-\vartheta}.
\end{align*}
Moreover,
\begin{align*}
\sum_{t=1}^\infty\alpha_t\sum_{k=1}^t\vartheta^{t-k}\alpha_{k-1}\leq \sum_{t=1}^\infty\sum_{k=1}^t\vartheta^{t-k}\alpha_{k-1}^2\leq \frac{1}{1-\vartheta}\sum_{t=1}^\infty\alpha_{t}^2,
\end{align*}
and similarly,
\begin{align*}
\sum_{t=1}^\infty\alpha_t\sum_{k=1}^t\vartheta^{t-k}\frac{\alpha_{k-1}}{\beta_{k-1}}\leq\frac{1}{1-\vartheta}\sum_{t=1}^\infty\frac{\alpha_{t}^2}{\beta_t}.
\end{align*}
In view of (\ref{eq-step-con}), by applying Robbins-Siegmund lemma \cite{R-S1971}, it can be concluded that $\|x_k-x^*\|^2+N\|\bar{\mu}_{t}-\mu^*\|^2$, and thus $\|x_k-x^*\|^2$ converges almost surely to some finite random variable, which implies $x_t$ is bounded. Moreover, $\sum_{t=1}^\infty\alpha_t\langle x_t-x^*,G(x_t)-G(x^*)\rangle<\infty$.
%\begin{align*}
%\sum_{t=1}^\infty\alpha_t\|x_t-x^*\|^2<\infty.
%\end{align*}
By noticing $\sum_{k=1}^\infty\alpha_k=\infty$, then $\lim\inf_{t\to\infty}\langle x_t-x^*,G(x_t)-G(x^*)\rangle=0$. As a result, there exists a subsequence $\{x_{t_j}\}\subseteq\{x_t\}$ such that $\lim_{j\to\infty}\langle x_{t_j}-x^*,G(x_{t_j})-G(x^*)\rangle=0$. Since $x_t$ is bounded, consequently, there exist a subsequence $\{x_{t_{j_l}}\}\subseteq\{x_{t_j}\}$ and point $\bar{x}\in X_0$, such that $\lim\limits_{l\to\infty}x_{t_{j_l}}=\bar{x}$, and thus $\lim_{l\to\infty}\langle x_{t_{j_l}}-x^*,G(x_{t_{j_l}})-G(x^*)\rangle=\langle \bar{x}-x^*,G(\bar{x})-G(x^*)\rangle=0$. Hence, the strict monotonicity of $G$ in Assumption \ref{ass:strict-mo} yields $\bar{x}=x^*$, which completes the proof.
\end{proof}

\subsection{Proof of Theorem \ref{thm:track}}\label{proof:thm:track}

\begin{proof}

Replacing $x^\star$ in (\ref{lem-reg-8}) and (\ref{lem-reg-11}) with $x^*$, noting that
$\sum_{t=1}^T\frac{N}{2\gamma_t}(\|\bar{\mu}_{t}-\lambda^*\|^2-\|\bar{\mu}_{t+1}-\mu^*\|^2)
\leq\frac{N\chi^2}{\gamma_{T+1}}+\frac{N^2L^2}{r^4\underline{w}^2\gamma_{T+1}\beta_{T+1}^2}$ by the similar arguments as (\ref{lem-reg-12}) by $\|\mu^*\|\leq \chi$, and then
summing over $t\in[T]$, one has from (\ref{NE-6}) and the strong monotonicity of $G$ in Assumption \ref{ass:strong-mo} that for any  $0<\delta<1$, with probability at least $1-\delta$,
%\begin{align}\label{NE-6}
%&\mu\|x_t-x^*\|^2\notag\\
%&\leq\frac{1}{2\alpha_t}\|x_{t}-x^*\|^2 - \frac{1}{2\alpha_t}\|x_{t+1}-x^*\|^2\notag\\
%&\quad +\frac{N}{2\gamma_t}\|\bar{\mu}_{t}-\mu^*\|^2- \frac{N}{2\gamma_t}\|\bar{\mu}_{t+1}-\mu^*\|^2+2N(1+\frac{2}{r})L^2\gamma_t\notag \\
%&\quad+\frac{80N^2\sqrt{m}L^2}{r}\sum_{k=1}^t\vartheta^{t-k}\gamma_{k-1}\notag\\
%&\quad+\frac{N^2}{2}\beta_t\chi^2 + NB^2\alpha_t+\frac{16NLl_f}{r}\|\sigma_0\|_1\vartheta^t\notag\\
%&\quad+2LN^2Sl_f\sqrt{N}L_\sigma\sum_{k=1}^t\vartheta^{t-k}\alpha_{k-1}\notag\\
%&\quad+\sum_{i=1}^N\langle p_{i,t}(x_{i,t},\tilde{\sigma}_{i,t+1}))-q_{i,t}(x_{i,t},\tilde{\sigma}_{i,t+1},\xi_{i,t}),x_{i,t+1}-x_i^*\rangle\notag\\
%&\quad +\frac{L^2}{\beta_t}+\frac{NM^2L^2}{ r^2}\frac{\alpha_t}{\beta_t^2}
%\end{align}
\begin{align}\label{NE-7}
&\mu\sum_{t=1}^T\|x_t-x^*\|^2\notag\\
&\leq\frac{2L^2}{\alpha_{T+1}} + \frac{N\chi^2}{\gamma_{T+1}}+\frac{N^2L^2}{r^4\underline{w}^2\gamma_{T+1}\beta_{T+1}^2}+NA_1\sum_{t=1}^T\alpha_{t-1}\notag\\
&\quad+ A_2\sum_{t=1}^T\gamma_{t-1}+\frac{N^2}{2}\chi^2\sum_{t=1}^T\beta_t+\frac{16NLl_f}{r(1-\vartheta)}\|\sigma_0\|_1\notag\\
&\quad+ \frac{16N^2\sqrt{n}L_\sigma L^2l_fM}{r^3(1-\vartheta)}\sum_{t=1}^T\frac{\alpha_{t-1}}{\beta_{t-1}}\notag\\
&\quad+\frac{NM^2L^2}{ r^4}\sum_{t=1}^T\frac{\alpha_t}{\beta_t^2}+\sum_{t=1}^T\sum_{i=1}^N\langle \Delta_t,x_{i,t}-x_i^*\rangle\notag\\
%&=\mathcal{O}(T^{\max\{a_1,a_3,1-a_3,1-a_2,1-a_1+2a_2\}}+T^{a_1}\ln\frac{1}{\delta}).
&=\mathcal{O}\bigg(T^{a_1}+T^{2a_2+a_3}+T^{1-a_3}+T^{1-a_2}+T^{a_1}\ln\frac{1}{\delta}\notag\\
&\quad+T^{1-a_1+2a_2}\bigg).
\end{align}

In view of the convexity of $\|\cdot\|^2$, it derives that
\begin{align*}
\left\|\frac{1}{T}\sum_{t=1}^Tx_t-x^*\right\|^2\leq \frac{1}{T}\sum_{t=1}^T\|x_t-x^*\|^2,
\end{align*}
which incorporating with (\ref{NE-7}) yields (\ref{thm-track}). The proof is completed.
\end{proof}

\subsection{Proof of Corollary \ref{cor:track}}\label{proof:cor:track}
\begin{proof}
Let $1-a_1=1-2a_2-a_3=a_2=a_3=a_1-2a_2$. Then it follows that $a_1=3/4$ and $a_2=a_3=1/4$, which completes the proof.
\end{proof}

\end{appendix}

%=================================================================================================
%\ifCLASSOPTIONcaptionsoff
%  \newpage
%\fi
%\bibliographystyle{ieeetr}
\bibliographystyle{IEEEtran}
\bibliography{dos}

\end{document}